\documentclass[aap]{imsart}

\RequirePackage{amsthm,amsmath,amsfonts,amssymb,amscd,mathrsfs}
\RequirePackage[numbers]{natbib}
\RequirePackage[colorlinks,citecolor=blue,urlcolor=blue]{hyperref}%% uncomment this for coloring bibliography citations and linked URLs
\RequirePackage{graphicx}%% uncomment this for including figures

\usepackage{bbm}
\usepackage{tikz-cd}

\startlocaldefs
\newcommand{\ol}{\overline}

\newcommand{\RN}[1]{%
  \textup{\uppercase\expandafter{\romannumeral#1}}%
}
\newtheorem{thm}{Theorem}[section]
\newtheorem{thmx}{Theorem}[section]

\newtheorem{defn}{Definition}[section]
\newtheorem{lem}{Lemma}[section]

\newtheorem{corx}{Corollary}[section]

\newtheorem{prop}{Proposition}[section]

\newtheorem{rem}{Remark}[section]

\numberwithin{equation}{section}

\def\C{\mathbb{C}}
\def\E{\mathbb{E}}

\def\N{\mathbb{N}}
\def\P{\mathbb{P}}

\def\R{\mathbb{R}}

\def\BB{\mathcal{B}}

\def\EE{\mathcal{E}}

\def\II{\mathcal{I}}

\def\LL{\mathcal{L}}

\def\PP{\mathcal{P}}

\def\KK{\mathcal{K}}

\def\andd{\,\,\text{and}\,\,}
\def\supp{\text{\rm supp}}

\def\lan{\langle}
\def\ran{\rangle}

\def\pa{\partial}
\def\stst{\subset\subset}

\def\al{\alpha}
\def\be{\beta}
\def\ep{\epsilon}
\def\ka{\kappa}

\def\la{\lambda}
\def\La{\Lambda}
\def\si{\sigma}

\def\de{\delta}

\def\ga{\gamma}

\endlocaldefs

\begin{document}

\begin{frontmatter}
%%%%%%%%%%%%%%%%%%%%%%%%%%%%%%%%%%%%%%%%%%%%%%
%%                                          %%
%% Enter the title of your article here     %%
%%                                          %%
%%%%%%%%%%%%%%%%%%%%%%%%%%%%%%%%%%%%%%%%%%%%%%
\title{Population dynamics under demographic and environmental stochasticity}
%\title{A sample article title with some additional note\thanksref{T1}}
\runtitle{Population dynamics under demographic and environmental stochasticity}
%\thankstext{T1}{A sample of additional note to the title.}

\begin{aug}
%%%%%%%%%%%%%%%%%%%%%%%%%%%%%%%%%%%%%%%%%%%%%%%
%% Only one address is permitted per author. %%
%% Only division, organization and e-mail is %%
%% included in the address.                  %%
%% Additional information can be included in %%
%% the Acknowledgments section if necessary. %%
%% ORCID can be inserted by command:         %%
%% \orcid{0000-0000-0000-0000}               %%
%%%%%%%%%%%%%%%%%%%%%%%%%%%%%%%%%%%%%%%%%%%%%%%
\author[A]{\fnms{Alexandru}~\snm{Hening}\ead[label=e1]{ahening@tamu.edu}},
\author[B]{\fnms{Weiwei}~\snm{Qi}\ead[label=e2]{wqi2@ualberta.ca}},
\author[B]{\fnms{Zhongwei}~\snm{Shen}\ead[label=e3]{zhongwei@ualberta.ca}}
\and
\author[B,C]{\fnms{Yingfei}~\snm{Yi}\ead[label=e4]{yingfei@ualberta.ca}}
%%%%%%%%%%%%%%%%%%%%%%%%%%%%%%%%%%%%%%%%%%%%%%
%% Addresses                                %%
%%%%%%%%%%%%%%%%%%%%%%%%%%%%%%%%%%%%%%%%%%%%%%
\address[A]{Department of Mathematics,  Texas A\&M University, College Station, TX 77843-3368, United States\printead[presep={,\ }]{e1}}

\address[B]{Department of Mathematical and Statistical Sciences, University of Alberta, Edmonton, AB T6G 2G1, Canada\printead[presep={,\ }]{e2,e3,e4}}

\address[C]{School of Mathematics, Jilin University, Changchun 130012, PRC\printead[presep={,\ }]{e4}}

\end{aug}

\begin{abstract}
The present paper is devoted to the study of the long term dynamics of diffusion processes modelling a single species that experiences both demographic and environmental stochasticity. In our setting, the long term dynamics of the diffusion process in the absence of demographic stochasticity is determined by the sign of $\Lambda_0$, the external Lyapunov exponent, as follows: $\Lambda_0<0$ implies (asymptotic) extinction and $\Lambda_0>0$ implies convergence to a unique positive stationary distribution $\mu_0$. If the system is of size $\frac{1}{\ep^{2}}$ for small $\ep>0$ (the intensity of demographic stochasticity), demographic effects will make the extinction time finite almost surely. This suggests that to understand the dynamics one should analyze the quasi-stationary distribution (QSD) $\mu_\ep$ of the system. The existence and uniqueness of the QSD is well-known under mild assumptions.

We look at what happens when the population size is sent to infinity, i.e., when $\ep\to 0$. We show that the external Lyapunov exponent still plays a key role: 1) If $\Lambda_0<0$, then $\mu_\ep\to \delta_0$, the mean extinction time is of order $|\ln \ep|$ and the extinction rate associated with the QSD $\mu_{\ep}$ has a lower bound of order $\frac{1}{|\ln\ep|}$; 2) If $\Lambda_0>0$, then $\mu_\ep\to \mu_0$, the mean extinction time is polynomial in $\frac{1}{\ep^{2}}$ and the extinction rate is polynomial in $\ep^{2}$. Furthermore, when $\Lambda_0>0$ we are able to show that the system exhibits multiscale dynamics: at first the process quickly approaches the QSD $\mu_\ep$ and then, after spending a polynomially long time there, it relaxes to the extinction state. We give sharp asymptotics in $\ep$ for the time spent close to $\mu_\ep$.

In contrast to models that only take into account demographic stochasticity, our results demonstrate the significant effect of environmental stochasticity -- it turns an exponentially long mean extinction time to a sub-exponential one. 
\end{abstract}

\begin{keyword}[class=MSC]
\kwd[Primary ]{35Q84}
\kwd{35J25}
\kwd[; secondary ]{37B25}
\kwd{60J60}
\end{keyword}

\begin{keyword}
\kwd{population dynamics}
\kwd{demographic stochasticity}
\kwd{environmental stochasticity}
\kwd{quasi-stationary distribution}
\kwd{extinction rate}
\kwd{extinction time}
\kwd{mean extinction time}
\kwd{multiscale dynamics}
\end{keyword}

\end{frontmatter}
%%%%%%%%%%%%%%%%%%%%%%%%%%%%%%%%%%%%%%%%%%%%%%
%% Please use \tableofcontents for articles %%
%% with 50 pages and more                   %%
%%%%%%%%%%%%%%%%%%%%%%%%%%%%%%%%%%%%%%%%%%%%%%
\tableofcontents

%%%%%%%%%%%%%%%%%%%%%%%%%%%%%%%%%%%%%%%%%%%%%%
%%%% Main text entry area:

\section{\bf Introduction}

One of the most important questions from population dynamics is figuring out when a species persists or goes extinct. For deterministic models, persistence is usually quantified via the existence of an attractor that is bounded away from zero (the extinction state). In this setting extinction can only happen asymptotically as time goes to infinity. However, any realistic ecological model has to take into account various intrinsic and extrinsic random environmental fluctuations. Usually there are either ecological models that take into account environmental stochasticity that arises due to fluctuations of the environment, or models from population genetics that focus on demographic stochasticity, which arises because of the randomness due to reproduction in a finite population. There are few analytic models which account for the effects of both types of stochasticity. 

If the system is of size $\frac{1}{\ep^{2}}$ for some small $\ep>0$ (intensity of the demographic noise), the presence of demographic effects will make the extinction time finite almost surely. As a result, in order to gain some information about the behavior of the process before extinction, it is natural and useful to look at quasi-stationary distributions (QSDs) \cite{CMSM13,MV12}, i.e., stationary distributions of the process conditioned on not going extinct. A key problem is to study scaling limits of systems that have QSDs and see what happens with the family of QSDs as the intensity of the demographic noise is sent to zero.

   The main goal of this paper is to analyze the dynamics of systems that have both types of stochasticity and can be modelled by stochastic differential equations (SDEs). We focus on the QSD and the extinction time as well as related quantities such as the extinction rate and the exponential convergence rate to the QSD, and investigate their asymptotic properties as the intensity of the demographic noise vanishes -- a particular emphasis is put on the connections to properties of the limit system.  Models with both types of stochasticity are more realistic as natural systems usually experience both types of randomness. The sharp criteria we find for the persistence and extinction of species are therefore more relevant to the modelling of natural ecosystems -- see \cite{HSL16,E20}.

Systems perturbed by either the environmental or demographic stochasticity have been attracting a lot of attention. If one looks at models that only have environmental stochasticity, there already exist many sharp results in the literature. In the one-dimensional setting, a full classification is possible by the well-known scale function and speed measure description of diffusions \cite{BS15}. In the multi-dimensional setting things are more complicated. Some general theory for the existence and uniqueness of stationary distributions can be found in \cite{K11}, while the most up to date results for Kolmogorov systems are in \cite{HN18, B18, HNS22}.  

For models with only demographic stochasticity, asymptotic properties of QSDs and related quantities as the intensity of the stochasticity vanishes are often the focus of studies. They have been investigated for randomly perturbed dynamical systems and rescaled Markov jump processes. The first work seemingly dates back to \cite{H97}, where the author studied the stochastic Ricker model. This work was generalized in \cite{KLZ98, RZ99} to randomly perturbed interval maps that apply to density-dependent branching processes. Further generalizations were considered in \cite{JS06,FS14}, where general randomly perturbed maps are studied and applied to many population models. These works illuminate two fundamental properties when the unperturbed deterministic system has a global attractor which is bounded away from extinction: 1) QSDs tend to concentrate on the deterministic attractor as the noise intensity vanishes. 2) The extinction rate associated with a QSD is exponentially small with respect to the system size (i.e., the reciprocal of the noise intensity squared), and therefore, the extinction time grows exponentially with the system size  if the initial distribution is given by the QSD.  Concentration properties of QSDs as in 1) are in line with that of stationary distributions for randomly perturbed dynamical systems (see e.g. \cite{Kifer88,Kifer89,K90,FW98,HK11,HJLY18}). For the latter in the case that the unperturbed system has simple dynamics, significantly more refined results are available in the literature (see e.g. \cite{Sheu86,Day87,Mikami88,BB09}).

Rescaled absorbed birth-and-death processes whose mean-field ODEs have a global asymptotically stable equilibrium have been investigated in \cite{Chan98,CCM16,CCM19}. In one dimension, the exponential asymptotic of QSDs and associated extinction rates are established in \cite{Chan98}. When the equilibrium is non-degenerate, these results are improved in \cite{CCM16} by determining the sub-exponential terms, implying in particular that QSDs converge to the Dirac measure at the equilibrium in a Gaussian manner. In higher dimensions, the aforementioned two fundamental properties are obtained in \cite{CCM19}. It is worthwhile to point out that the problem in higher dimensions is much more challenging due to the irreversibility and the lack of simple recursive formulas for QSDs. In \cite{CCM16,CCM19}, the authors also characterize the two-scale dynamics of the solution processes by deriving sophisticated estimates quantifying the distance between the distribution of the solution and the convex combination of the extinction state (more precisely, the Dirac measure at the extinction state) and the QSD. 

In \cite{SWY,JQSY21,QSY} the authors consider one-dimensional absorbed singular diffusion processes of generalized logistic type with small demographic noises -- these models can be derived as diffusion approximations of one-dimensional rescaled absorbed Markov jump processes arising from population dynamics and chemical reactions. When the unperturbed or mean-field ODE has a unique positive equilibrium (which must be globally asymptotically stable), results comparable to those contained in \cite{CCM16} are established. In particular, the noise-vanishing asymptotic of QSDs and associated extinction rates are determined up to the sub-exponential terms, and the two-scale dynamics of the solution process is characterized. The noise-vanishing asymptotic of QSDs and associated extinction rates extends to the case where the unperturbed ODE has multiple positive stable equilibria. We point out that while QSDs for many types of processes have been extensively studied (see \cite{Pokkett-bib,CCLMMS09,vDP13,CMSM13} and reference therein), the fundamental theory of QSDs (i.e., the existence, uniqueness and convergence) for absorbed singular diffusion processes was unavailable until the work \cite{CCLMMS09}. Since then, there have been significant new developments (see e.g. \cite{LC12,MV12,CMSM13,Miura14,CV18,HK19,HQSY}). 

There exist relevant works on overdamped Langevin equations restricted in a bounded domain and killed on its boundary \cite{Mathieu95,BGK05,DLLN19,DLLN20,DLLN21}. In \cite{Mathieu95}, the author derived the exponential asymptotic of the extinction rate (more appropriately, the exit rate for a diffusion process exiting from a bounded domain) and the asymptotic of the principal eigenfunction of the generator in the deepest well of the potential, leading to the sub-exponential asymptotic of the QSD in that well. These results are greatly improved in \cite{BGK05} under generic assumptions on the potential function. In a series of works \cite{DLLN19,DLLN20,DLLN21} examining exit events and the Eyring-Kramers formula, the sub-exponential asymptotic of the exit rate plays a significant role in computing the asymptotic of transition rates and determining the asymptotic exit distribution. 

This paper is a first step towards generalizing the theory of randomly perturbed dynamical systems without absorbing states and randomly perturbed dynamical systems with absorbing states and only demographic noises to a theory of randomly perturbed dynamical systems with absorbing states and multiple types of noise. Inspired by the aforementioned theories of noise-vanishing asymptotics of stationary distributions, QSDs, and related quantities, and motivated by the fact that real systems are subject to both intrinsic and extrinsic stochastic perturbations, we intend to establish an analogous theory for dynamical systems under both environmental and demographic noise perturbations, and study the effects of both types of noises.

In the present paper, we consider one-dimensional SDEs with both environmental and demographic stochasticity: 
\begin{equation*}
dX^{\ep}_t=b(X^{\ep}_t)dt+\si(X^{\ep}_t) dB_t+\ep\sqrt{a(X^{\ep}_t)}dW_t\quad\text{in}\quad[0,\infty),
\end{equation*}
where the coefficients $b$, $\si$ and $a$ satisfy natural assumptions. Let $T^{\ep}_{0}=\inf\left\{t\geq0:X^{\ep}_{t}=0\right\}$ be the extinction time of $X^{\ep}_{t}$. It is finite almost surely. Denote by $\LL_{\ep}$ the self-adjoint extension in $L^{2}(u_{\ep}^{G}):=L^{2}((0,\infty),u_{\ep}^{G}dx)$ of the generator of $X_{t}^{\ep}$, where $u_{\ep}^{G}$ is the non-integrable Gibbs density of $X_{t}^{\ep}$ as it grows like $\frac{1}{x}$ as $x\to0^{+}$. The spectrum of $\LL_{\ep}$ is purely discrete. Depending on the dynamics of the limiting SDE, which only has an environmental stochasticity term:
$$
dX^{0}_t=b(X^{0}_t)dt+\si(X^{0}_t) dB_t\quad\text{in}\quad[0,\infty),
$$
we are able to prove the following results (with rigorous statements given in Section \ref{s:setup}):

\smallskip

\noindent{\bf (I)} Suppose $\La_{0}:=b'(0)-\frac{|\sigma'(0)|^2}{2}>0$ so that $X_{t}^{0}$ has a unique stationary distribution $\mu_0$ that does not put mass on the extinction state $0$. 
\begin{itemize}
\item The unique QSD $\mu_\ep$ of $X_{t}^{\ep}$ converges to $\mu_0$ as the intensity of the demographic noise goes to zero, that is, $\ep\to0$. The associated extinction rate $\la_{\ep,1}$ is given by the principal eigenvalue of $-\LL_{\ep}$, and is polynomially small in $\ep$ with leading order $\la_{\ep,1}\sim\ep^{\frac{4b'(0)}{|\si'(0)|^2}-2}$.

\item The normalized extinction time $\frac{T^{\ep}_0}{\E_{\bullet}^{\ep}[T^{\ep}_0]}$ converges weakly to an exponential random variable of mean $1$ as $\ep\to0$. Moreover, the mean extinction time $\E^{\ep}_{\bullet}[T^{\ep}_0]$ depends polynomially on the system size $\frac{1}{\ep^{2}}$ with leading order 
$$
\E^{\ep}_{\bullet}[T^{\ep}_0]\sim\frac{1}{\la_{\ep,1}}\sim\ep^{2-\frac{4b'(0)}{|\si'(0)|^2}} = \left(\frac{1}{\ep^2}\right)^{\frac{2b'(0)}{|\si'(0)|^2}-1}.
$$

The polynomial asymptotics of the extinction rate $\la_{\ep,1}$ and the mean extinction time $\E^{\ep}_{\bullet}[T^{\ep}_0]$ are \emph{significant changes} from that of models having only demographic noise, see \cite{FS14,JQSY21,QSY}, where the dependence on the noise intensity is exponential. This shows that environmental stochasticity has a significant impact on the time-scales of the dynamics. The fact that the dependence changes from exponential to polynomial in the presence of environmental stochasticity has been recently showcased empirically and numerically in simple ecological models \cite{HSL16,E20}.

\item The eigenfunction $\phi_{\ep,1}$ of $-\LL_{\ep}$ associated with $\la_{\ep,1}$ converges, after appropriate normalization, to $1$ as $\ep\to0$. The second eigenvalue $\la_{\ep,2}$ of $-\LL_{\ep}$ satisfies 
$$
0<\liminf_{\ep\to 0}\la_{\ep,2}\leq\limsup_{\ep\to 0}\la_{\ep,2}<\infty,
$$
yielding in particular the uniform spectral gap $\inf_{\ep}(\la_{\ep,2}-\la_{\ep,1})>0$.

\item The distribution of $X_{t}^{\ep}$ satisfies the multiscale estimate:
\begin{equation}\label{intro-multiscale-estimate}
    \left\|\P^{\ep}_{\bullet}[X^{\ep}_t\in \bullet]-\left[\al_{\ep}e^{-\la_{\ep,1}t } \mu_{\ep}+\left(1-\al_{\ep}e^{-\la_{\ep,1}t }\right)\de_0\right]\right\|_{TV}\leq C e^{-\la_{\ep,2}t},
\end{equation}
where $\al_{\ep}$ is the integral of the appropriately normalized $\phi_{\ep,1}$ with respect to the initial distribution, and the constant $C$ depends on the initial distribution but is independent of $\ep$. This estimate together with information about $\la_{\ep,1}$, $\la_{\ep,2}$ and $\phi_{\ep,1}$ allows us to  quantify the multiscale dynamics of $X_{t}^{\ep}$ as follows.
If $t$ is such that $\frac{1}{\lambda_{\ep,2}}\ll t\ll \frac{1}{\lambda_{\ep,1}}$, then $\|\P^{\ep}_{\bullet}[X^{\ep}_t\in \bullet]-\mu_\ep]\|_{TV}\ll 1$, that is, the distribution of $X^{\ep}_t$ is close the QSD $\mu_\ep$. If $t$ is such that $t\gg \frac{1}{\lambda_{\ep,1}}$, then $\|\P^{\ep}_{\bullet}[X^{\ep}_t\in \bullet]-\de_0]\|_{TV}\ll 1$, that is, the distribution of $X^{\ep}_t$ gets close to $\delta_0$, the Dirac mass at the extinction state.

The estimate \eqref{intro-multiscale-estimate} is powerful -- it has the convergence result of the normalized extinction time $\frac{T^{\ep}_0}{\E_{\bullet}^{\ep}[T^{\ep}_0]}$ and the asymptotic reciprocal relationship $\E^{\ep}_{\bullet}[T^{\ep}_0]\sim\frac{1}{\la_{\ep,1}}$ as immediate consequences.
\end{itemize}

\smallskip

\noindent{\bf (II)} Suppose $\La_{0}<0$ so that $X_{t}^{0}$ goes extinct as $t\to\infty$.
    \begin{itemize}
        \item As $\ep\to 0$, we have $\mu_\ep\to \delta_0$. The extinction rate $\la_{\ep,1}$ vanishes as $\ep\to0$ and has a lower bound of order $\frac{1}{|\ln\ep|}$.
        
        \item The mean extinction time is of order $|\ln \ep |$, that is, $\E^{\ep}_{\bullet}[T^{\ep}_{0}]\sim |\ln \ep|$. 
    \end{itemize}

The quantity $\Lambda_0$ is often referred to as the \emph{stochastic growth rate} (it is also called the \emph{invasion rate} or the \emph{external Lyapunov exponent}) -- it determines the stability of the extinction state $0$ for $X_{t}^{0}$. As $\Lambda_0$ increases and crosses $0$, the stable extinction state loses its stability and bifurcates into an unstable extinction state and the globally asymptotically stable persistent state $\mu_{0}$. As it is seen from {\bf(I)} and {\bf(II)} such a bifurcation has a strong effect on the asymptotics of the extinction rate $\la_{\ep,1}$ and the mean extinction time $\E^{\ep}_{\bullet}[T^{\ep}_{0}]$.

To this end, we briefly comment on the ideas, methods and techniques used to establish the above results, as well as the difficulties overcome in the course of the proof. We pay particular attention to the comparison with the model that only has demographic stochasticity, that is,
\begin{equation*}
d\tilde{X}^{\ep}_t=b(\tilde{X}^{\ep}_t)dt+\ep\sqrt{a(\tilde{X}^{\ep}_t)}dW_t\quad\text{in}\quad[0,\infty).
\end{equation*}
For clarity, we assume $b$ is just the standard logistic growth rate function with $x_{*}$ being the only positive zero. Denote by $\tilde{\LL}_{\ep}$ the self-adjoint extension of the generator of $\tilde{X}_{t}^{\ep}$. Under natural assumptions on $a$, the spectrum of $\tilde{\LL}_{\ep}$ is purely discrete. Denote by $\tilde{\la}_{\ep,1}$ and $\tilde{\la}_{\ep,2}$ the first two eigenvalues of $-\tilde{\LL}_{\ep}$.
\begin{itemize}
    \item It is known (see e.g. \cite{Mathieu95,BGK05,JQSY21}) that the asymptotic of $\tilde{\la}_{\ep,1}$ and $\tilde{\la}_{\ep,2}$ are respectively determined by the potential function $\tilde{V}:=-\int_{0}^{\bullet}\frac{b}{a}ds$ and the vector field $b$ at $x_{*}$. More precisely, $\lim_{\ep\to0}\frac{\ep^{2}}{2}\ln\tilde{\la}_{\ep,1}=\tilde{V}(x_{*})$, and $\lim_{\ep\to0}\tilde{\la}_{\ep,2}=-b'(x_{*})$. The behavior of $\la_{\ep,1}$ and $\la_{\ep,2}$ is completely different: we can show in the case $\La_{0}>0$ that the leading asymptotic of $\la_{\ep,1}$ is determined by $b'$ and $\si'$ at the extinction state $0$. This shows that environmental stochasticity significantly alters the `hidden mechanisms' which affect the mean extinction time.  
    
    \item Denote by $\LL_{0}$ the self-adjoint extension of the generator of $X_{t}^{0}$. One expects that the asymptotics of $\la_{\ep,1}$ and $\la_{\ep,2}$ are governed by the spectral properties of $-\LL_{0}$. However, this is not clear at all because of the \textit{singular} limit ``$\lim_{\ep\to0}\LL_{\ep}=\LL_{0}$". The coefficient of the second-order term of $\LL_{\ep}$ has a first-order degeneracy at $0$, while that of $\LL_{0}$ has a second-order degeneracy at $0$. One of the unpleasant consequences of this singularity is that the structure of the spectrum of $\LL_{0}$ differs significantly from that of $\LL_{\ep}$. The reader is referred to Remark \ref{rem-troubles-from-spectrum} for details.

    \item Proving that $\inf_{\ep}\la_{\ep,2}>0$ is hard in part due to the singularity of the limit ``$\lim_{\ep\to0}\LL_{\ep}=\LL_{0}$". The way we prove this builds on the simple fact that the eigenfunctions associated with $\la_{\ep,1}$ and $\la_{\ep,2}$ are orthogonal. Assuming the failure of $\inf_{\ep}\la_{\ep,2}>0$, we manage to show the loss of the orthogonality of eigenfunctions. A crucial ingredient leading to this contradiction is to acquire certain compactness of appropriately normalized eigenfunctions associated with $\la_{\ep,1}$ and $\la_{\ep,2}$. 
    
    \item Given the fact that eigenfunctions of $-\LL_{\ep}$ span $L^{2}(u_{\ep}^{G})$, the multiscale estimate \eqref{intro-multiscale-estimate} follows essentially from the eigenfunction expansion of the Markov semigroup $P_{t}^{\ep}$ associated with $X_{t}^{\ep}$ before hitting $0$, saying particularly that all the terms in the estimate arise naturally except the property that the constant $C$ on the right hand side is independent of $\ep$. The key to obtaining this is to derive good pointwise estimates of $P_{t}^{\ep}Q_{2}^{\ep}f$ for $f\in C_{b}((0,\infty))$ by lifting the integrability, as we know $\|P^{\ep}_tQ^{\ep}_2\|_{L^2(u^{G}_{\ep})\to L^2(u^{G}_{\ep})}\leq e^{-\la_{\ep,2}t}$ from $P_{t}^{\ep}$ being generated by $\LL_{\ep}$, where $Q^{\ep}_2$ is the spectral projection of $\LL_{\ep}$ corresponding to eigenvalues $\si(\LL_{\ep})\setminus\{-\la_{\ep,1}\}$. This, however, is not an easy job due to the degeneracy of $\LL_{\ep}$ at $0$ and the singularity of $u_{\ep}^{G}$ at $0$. We overcome the difficulties by examining the Schr\"{o}dinger operator and the associated semigroup that are respectively unitarily equivalent to $\LL_{\ep}$ and $P^{\ep}_{t}$. It is the blowup feature of the potential of the Schr\"{o}dinger operator that helps to lift the integrability and reach the goal.
    
    \item The asymptotic of the extinction rate $\la_{\ep,1}$ in the case $\La_{0}>0$ is tackled from two perspectives. The first approach uses only the classical variational formula. A careful analysis of the eigen-equation (written in the quadratic form) near the extinction state $0$ allows us to derive the sharp lower bound. The analysis extends to the case $\La_{0}<0$. A non-sharp upper bound is obtained by constructing test functions. The other approach, which leads to the sharp asymptotic, builds on two independently established results: the asymptotic reciprocal relationship $\E^{\ep}_{\bullet}[T^{\ep}_0]\sim\frac{1}{\la_{\ep,1}}$ and the sharp asymptotic of $\E^{\ep}_{\bullet}[T^{\ep}_0]$. The former is an immediate consequence of the multiscale estimate \eqref{intro-multiscale-estimate} as mentioned in {\bf (I)}. The latter is achieved by a probabilistic approach that extends to the case $\La_{0}<0$.
\end{itemize}

The preprint \cite{PS21} has recently come to our attention. The authors have been able to prove results analogous to ours for a specific stochastic SIS epidemic model in randomly switched environments. The SIS model is described by a multitype birth-and-death process $X^K$ in a randomly switched environment -- the infection and recovery rates depend on the state of a finite Markov process, which model the environment, whose transition rates in turn depend on the number of infected individuals. The total population size $K$ is fixed and the authors show that as $K\to \infty$ the process converges to a piecewise deterministic Markov process that lives on a compact state space. The behavior of the limiting process is determined by the top Lyapunov exponent, $\Lambda$, of the linearized system. The authors are able to show that as $K\to\infty$ the time to extinction is, when $\Lambda<0$, at the most of order $\ln K$ and, when $\Lambda>0$, at least of polynomial order in $K$. We note that our results are for SDEs and are significantly sharper.

The paper is organized as follows. In Section \ref{s:setup} we offer the rigorous mathematical setup of the problem and exhibit our main results. Section \ref{s:prelim} offers some preliminary results that are needed for the main results. The analysis of the tightness and concentration of the measures $\mu_\ep$ as well as the proof of Theorem \ref{thm-concentration} are provided in Section \ref{s:tight}. Section \ref{s:asym} deals with the proof of Theorem \ref{thm-bounds-eigenvalues} which is about the asymptotic bounds on the first two eigenvalues $\la_{\ep,1}$ and $\la_{\ep,2}$ of the generator. The multiscale dynamics of $X_{t}^{\ep}$ and proof of Theorem \ref{thm-multiscale} appear in Section \ref{s:multi}. Section \ref{s:ext} is about the asymptotic of the mean extinction time $\E^{\ep}_{\bullet}[T^{\ep}_0]$ and the proof of Theorem \ref{thm-asymptotic-mean-extinction}. %In Appendix \ref{appendix}, we derive the leading order asymptotic of $\la_{\ep,1}$ on the basis of asymptotic expansions.

\section{\bf Mathematical setup and main results} \label{s:setup}

We consider the following family of SDEs:
\begin{equation}\label{main-diffusion-eqn}
dX^{\ep}_t=b(X^{\ep}_t)dt+\si(X^{\ep}_t) dB_t+\ep\sqrt{a(X^{\ep}_t)}dW_t\quad\text{in}\quad[0,\infty),
\end{equation}
where $0<\ep\ll 1$ is a small parameter, $b,\si:[0,\infty)\to \R$, $a:[0,\infty)\to [0,\infty)$ and $B_t, W_t$ are two independent standard one-dimensional Brownian motions on some probability space. Here, $\si dB_t$ models environmental stochasticity and $\ep\sqrt{a}dW_t$ represents demographic stochasticity. Hence, $\ep$ stands for the intensity of the demographic stochasticity. We point out that $\ep^{2}$ is inversely proportional to the population size, and hence, tends to $0$ as the population size goes to infinity. 

Throughout this paper, we make the following assumptions on the coefficients $b$, $\si$ and $a$.
\begin{itemize}
\item[\bf(H)]
The functions $b:[0,\infty)\to\R$, $\si:[0,\infty)\to\R$ and $a:[0,\infty)\to[0,\infty)$ are assumed to satisfy the following conditions:
\begin{itemize}
\item[\rm(1)] $b\in C^{1}([0,\infty))$,
$b(0)=0$, $b'(0)>0$, and $\limsup_{x\to\infty}b(x)<0$;

\item[\rm(2)] $\si\in C^{2}([0,\infty))$, $\si(0)=0$ and $\si'(0)\neq 0$;

\item[\rm(3)] $a\in C^{2}([0,\infty))$, $a(0)=0$, $a'(0)>0$, and $a>0$ on $(0,\infty)$;

\item[\rm(4)] there holds 
\begin{gather*}
     \limsup_{x\to \infty}\frac{a(x)}{\si^2(x)}<\infty,\quad\lim_{x\to\infty}\frac{b(x)}{|\si(x)|}=\lim_{x\to \infty}\frac{xb(x)}{\si^2(x)}=-\infty,\\
    \limsup_{x\to \infty}\frac{\si^2(x)}{|b(x)|} \max\left\{\frac{a'(x)}{a(x)}, \frac{|\si'(x)|}{|\si(x)|}\right\}=0,\quad \andd\quad \\
    \limsup_{x\to \infty}\frac{\si^2(x)}{b^2(x)} \max\left\{a''(x), (\si^2(x))'', |b'(x)|\right\}=0.
\end{gather*}

% {\rd $$
% \limsup_{x\to\infty}\frac{\max\left\{a(x),|a'(x)|,|a''(x)|,|b'(x)|\right\}}{|b(x)|}<\infty.
% $$
% }
\end{itemize}
\end{itemize}

{\bf (H)}(1) says that $b$ is a logistic-type growth rate function -- these types of growth rates play an important roles in many biological and ecological applications. In particular, $b(x)$ looks like $b'(0)x$ around $0$ and the per-capita growth rate at zero is positive, $b'(0)>0$, something which ensures persistence if there is no demographic or environmental stochasticity. {\bf (H)}(2) is satisfied if $\si(x)=xf(x)$ for some $f\in C^2([0,\infty))$ and often appears in modeling environmental stochasticity. We note that in most applications one has $\sigma(x)=\sigma x$ for some $\sigma>0$. {\bf (H)}(3) assumes that $a$ is degenerate at $0$ and behaves like $a'(0)x$ near 0. It is worthwhile to point out that the singularity of $a$ at $0$ causes the non-integrability of the Gibbs density near $0$, and thus, leads to substantial difficulties in the analysis of \eqref{main-diffusion-eqn}. The condition $\limsup_{x\to \infty}b(x)<0$ in {\bf (H)}(1) and the growth conditions on $a$, $b$ and $\si$ in {\bf (H)}(4) guarantee \eqref{main-diffusion-eqn} forms a dissipative system. Other conditions in {\bf (H)}(4) restricting the derivatives of $a$, $b$ and $\si^2$ near $\infty$ are mild technical assumptions. The assumption {\bf (H)} applies in particular to the logistic diffusion:
$$
dX_t=X_t(\mu-\kappa X_t)dt+\sigma X_tdB_t+\ep\sqrt{\gamma X_t}dW_t\quad\text{in}\quad[0,\infty)
$$
and the stochastic theta logistic model (with $\theta>0$):
\begin{equation*}
    dX_t=X_t(\mu-\ka X_t^{\theta})dt+\si X_t dB_t+\ep\sqrt{\ga X_t}dW_t \quad\text{in}\quad[0,\infty).
\end{equation*}
%$\sigma(x) = \sigma x, a(x)=\gamma x$ a) theta logistic $b(x) = x (\mu - \kappa x^\theta)$, b)  Nisbet and Gurney '82 $b(x) = x(r e^{-x/L}-d)$

Denote by $X_{t}^{\ep}$ the diffusion process on $[0,\infty)$ generated by  solutions of \eqref{main-diffusion-eqn}. For singular diffusion processes like \eqref{main-diffusion-eqn}, the strong uniqueness is ensured by the well-known Yamada-Watanabe theory \cite{YW71,WY71}. Moreover, $X^{\ep}_{t}$ gets absorbed by the absorbing state $0$ in finite time almost surely (see e.g. \cite{CCLMMS09,IW81}), leading eventually to extinction dynamics. However, $X^{\ep}_t$ can display long interesting dynamics before hitting $0$. To capture such dynamics, we use {\em quasi-stationary distributions} of $X^{\ep}_t$ -- these are initial distributions of $X^{\ep}_t$ on $(0,\infty)$ such that the distribution of $X_{t}^{\ep}$ conditioned on not reaching $0$ up to time $t$ is independent of $t\geq0$. 

Let $T^{\ep}_{0}$ be the first time that $X^{\ep}_{t}$ hits $0$ (often called the \emph{extinction time}), that is, 
$$
T^{\ep}_{0}=\inf\left\{t\geq0:X^{\ep}_{t}=0\right\}.
$$ 
Then, $\P^{\ep}_{\mu}[T^{\ep}_0<\infty]=1$ as mentioned above (see also \cite[Chapter VI, Section 3]{IW81}), where $\P^{\ep}_{\mu}$ the law of $X_{t}^{\ep}$ with initial distribution $\mu$. The associated expectation is denoted by $\E^{\ep}_{\mu}$. If $\mu=\de_{x}$, we simply write $\P^{\ep}_{x}=\P^{\ep}_{\de_x}$ and $\E^{\ep}_{x}=\E^{\ep}_{\de_x}$. 

\begin{defn}[Quasi-stationary distribution]
A Borel probability measure $\mu_{\ep}$ on $(0,\infty)$ is
called a \emph{quasi-stationary distribution (QSD)} of
$X^{\ep}_{t}$ if
$$
\P^{\ep}_{\mu_{\ep}}\left[X_{t}^{\ep}\in B|T_{0}^{\ep}>t\right]=\mu_{\ep}(B),\quad \forall t\geq0,\quad B\in\BB((0,\infty)),
$$
where $\BB((0,\infty))$ is the Borel $\si$-algebra of $(0,\infty)$.
\end{defn}

The general theory of QSDs (see e.g. \cite{MV12,CMSM13}) says that if $\mu_{\ep}$ is a QSD of $X_{t}^{\ep}$, then there is a unique number $\la_{\ep,1}>0$ such that $T_{0}^{\ep}$ is exponentially distributed with rate $\la_{\ep,1}$ provided $X_{0}^{\ep}\sim\mu_{\ep}$, that is,
\begin{equation}\label{met-qsd}
\P^{\ep}_{\mu_{\ep}}[T_0^{\ep}>t]=e^{-\lambda_{\ep,1}t},\quad \forall t\geq0.
\end{equation}
For this reason, $\la_{\ep,1}$ is often referred to as the \emph{extinction rate}.

% Note that the distribution of solutions of \eqref{main-diffusion-eqn} is the same as that of solutions of 
% \begin{equation*}
%     d\tilde{X}_{t}=b(\tilde{X}_{t})dt+\sqrt{\si^2(\tilde{X_{t}})+\ep^2 a(\tilde{X_{t}})}d\tilde{W}_t,
% \end{equation*}
% where $\tilde{W}_t$ is a standard one-dimensional Brownian motion. Thanks to {\bf (H)}, we apply results in \cite{CCLMMS09} to see that the above SDE admits a unique QSD $\mu_{\ep}$ which is the same as that of \eqref{main-diffusion-eqn}.  

Following \cite{CCLMMS09}, we check that under {\bf(H)}, $X_{t}^{\ep}$ admits a unique QSD $\mu_{\ep}$ with a positive $C^{2}$ density $u_{\ep}$ (see Lemma \ref{qsd-existence-uniqueness} for details). Moreover, the associated extinction rate $\la_{\ep,1}$ is given by the principal (or the first) eigenvalue of $-\LL_{\ep}$, where $\LL_{\ep}$ is an appropriate extension of the generator of $X_{t}^{\ep}$ and acts on functions in $C^{2}((0,\infty))$ as:
\begin{equation}\label{e:LE}
\LL_{\ep}\phi=\frac{1}{2}(\ep^2a+\si^2)\phi''+b\phi',\quad\forall\phi\in C^{2}((0,\infty)).
\end{equation}
The rigorous definition of $\LL_{\ep}$ is given in Subsection \ref{subsec-generator-spectrum-dynamics}. In addition, the spectral gap between the first and second eigenvalues, $\la_{\ep,1}$ and $\la_{\ep,2}$, of the operator $-\LL_{\ep}$ characterizes the exponential convergence rate of $\P^{\ep}_{\mu}\left[X_{t}^{\ep}\in \bullet|t<T_{0}^{\ep}\right]$ to the QSD $\mu_{\ep}$ as $t\to\infty$ whenever the initial distribution $\mu$ is compactly supported in $(0,\infty)$.

The main goal of this paper is to analyze the combined effects of environmental and demographic noises on population persistence and extinction. In order to achieve this, it is of paramount importance to investigate the asymptotic properties of $\mu_{\ep}$, $\la_{\ep,1}$ and $\la_{\ep,2}$ as $\ep\to0$. We are able to provide detailed information about the diffusion process $X_{t}^{\ep}$ governed by the QSD $\mu_{\ep}$, and to characterize the extinction time $T_{0}^{\ep}$ (especially, the mean extinction time $\E_{\bullet}[T_{0}^{\ep}]$) as well as the global multiscale dynamics of $X_{t}^{\ep}$.  

Investigating the asymptotic properties of $X_{t}^{\ep}$ and related quantities (i.e., $\mu_{\ep}$, $\la_{\ep,1}$, $\la_{\ep,2}$ and $T_{0}^{\ep}$) as $\ep\to0$ leads naturally to the limiting equation of \eqref{main-diffusion-eqn}, namely, 
\begin{equation}\label{eqn-diffusion-extrinsic}
dX^0_{t}=b(X^0_{t})dt+\si(X^0_{t}) dB_t\quad\text{in}\quad[0,\infty). 
\end{equation}
Intuitively, the first step towards a good understanding of these asymptotic properties is to acquire relevant information about the diffusion process $X_{t}^{0}$ on $[0,\infty)$ generated by solutions of \eqref{eqn-diffusion-extrinsic}. It is worthwhile to point out that $X_{t}^{0}$ behaves fundamentally different from $X_{t}^{\ep}$ over large time scales. For instance, if $X_{0}^{0}=x\in(0,\infty)$, then $X_{t}^{0}$ does not reach the absorbing state $0$ in finite time almost surely, that is, $X_{t}^{0}>0$ $\P_{x}$-a.e. for all $t>0$ (see Proposition \ref{prop-existence-uniqueness-sd}). Moreover, the spectral structure of the generator of \eqref{eqn-diffusion-extrinsic} differs very much from that of the generator of \eqref{main-diffusion-eqn}, i.e., $\LL_{\ep}$. More precisely, the latter is purely discrete (see Lemma \ref{prop-spectral-structure}), while the former is not (see Remark \ref{rem-troubles-from-spectrum}). 

The dynamics of $X_{t}^{0}$ is very well understood. Following for example \cite{Kallenberg02, EHS15, HN18}, we define the \textit{stochastic growth rate} (also called \textit{invasion rate} or \textit{external Lyapunov exponent})
\begin{equation}\label{e:lambda0}
\Lambda_0 := b'(0) - \frac{|\sigma'(0)|^2}{2}.    
\end{equation}
In population dynamics, the condition $\Lambda_0>0$ implies that a species tends to increase when it is at a low density, and therefore, persists in the long run (see \cite{Kallenberg02, EHS15, HN18}). The following sharp threshold result is part of Proposition \ref{prop-existence-uniqueness-sd}: 
\begin{itemize}
    \item if $\Lambda_0<0$, then $\de_{0}$ is the only stationary distribution of $X_{t}^{0}$; 
    
    \item if $\Lambda_0>0$, then $X_{t}^{0}$ admits a unique positive stationary distribution $\mu_0$ with a positive density $u_0\in C^2((0,\infty))$ given by the normalized Gibbs density, namely, $u_{0}=\frac{u_0^G}{\|u_0^G\|_{L^{1}((0,\infty))}}$, where
\begin{equation}\label{gibbs-environmental-only}
u_0^G:=\frac{1}{\si^2} e^{2\int_1^{\bullet}\frac{b}{\si^2}ds}\quad\text{in}\quad(0,\infty).
\end{equation}
\end{itemize}
More detailed information is given in Subsection \ref{subsec-sde-environmental}. Our main focus is on the case $\La_{0}>0$ -- in this setting the persistence of a species whose dynamics is modelled by \eqref{eqn-diffusion-extrinsic} becomes a transient property when the model \eqref{main-diffusion-eqn} is used. This is how things usually behave in nature where a population persists for a long time after which it eventually goes extinct. Our purpose is to give quantitative and qualitative characterizations of this phenomenon. We are also able to establish interesting results in the case $\La_{0}<0$, demonstrating significant changes as $\La_{0}$ crosses $0$, where a bifurcation occurs.

Our first result addresses the limiting behaviors of $\mu_{\ep}$ as $\ep\to0$. The space $C^2((0,\infty))$ is equipped with the topology of locally uniform convergence up to the second derivative.

\begin{thmx}\label{thm-concentration}
Assume {\bf (H)}.
\begin{itemize}
    \item[\rm(1)] If $\Lambda_0<0$, then $\lim_{\ep\to0}\int_{0}^{\infty}\phi d\mu_{\ep}=0$ for any $\phi\in C_{b}([0,\infty))$ with $\phi(0)=0$.
    
    \item[\rm(2)] If $\Lambda_0>0$, then $\lim_{\ep\to0}\mu_{\ep}=\mu_0$ weakly, and $\lim_{\ep\to0}u_{\ep}=u_0$ in $C^2((0,\infty))$. 
\end{itemize}
\end{thmx}

Given the aforementioned sharp threshold result of $X_{t}^{0}$ and the fact that $X_{t}^{\ep}$ is a small random perturbation of $X_{t}^{0}$ (or, \eqref{main-diffusion-eqn} is a small random perturbation of \eqref{eqn-diffusion-extrinsic}), the conclusions from Theorem \ref{thm-concentration} are expected and look pretty straightforward. This, however, is completely deceptive from a technical perspective, especially in the case $\Lambda_{0}>0$. Indeed, when $\Lambda_{0}>0$, it is not hard to show that any limiting measure of $\{\mu_{\ep}\}_{\ep}$ must be $\mu_{0}$ (up to multiplication by a constant), and hence, the weak convergence $\lim_{\ep\to0}\mu_{\ep}=\mu_0$ follows if $\{\mu_{\ep}\}_{\ep}$ is tight. The tightness of $\{\mu_{\ep}\}_{\ep}$ comes from studying their concentration near $0$ and $\infty$. The concentration near $\infty$ follows mainly from the dissipativity and is obtained by means of the usual technique on the basis of Lyapunov-type functions (see Proposition \ref{prop-concentration-infty}). Establishing the concentration near $0$ is however troublesome due to the following technical problems: (i) both the vector field $b$ and the noise terms $\si$ and $\ep\sqrt{a}$ vanish at $0$; such degeneracies are known to cause difficulties in the analysis and are often avoided in the literature when treating noise-vanishing problems; (ii) techniques based on Lyapunov-type functions do not apply because of the demographic noise term which causes the finite time extinction of $X_{t}^{\ep}$; otherwise, a unique non-trivial stationary distribution would exist, instead of the QSD. These issues are circumvented by a two-step approach: an $\ep$-dependent upper bound of $u_{\ep}$ is first established (see Lemma \ref{bound-u_ep-near-0}); it is followed by an argument of maximum principle type (see the proof of Proposition \ref{prop-concentration-0}). As a result, we establish in Proposition \ref{prop-concentration-0} the following concentration estimate of the densities $\{u_{\ep}\}_{\ep}$:
$$
\sup_{\ep}u_{\ep}(x)\leq \frac{C}{x^{k}},\quad \forall x\in (0,x_*)
$$
for some $k\in (0,1)$, $x_*>0$ and $C>0$. Such an upper bound is more or less inspired by the expectation $\lim_{\ep\to0}u_{\ep}=u_0$ and the behavior of $u_{0}$ near $0$, i.e., $u_{0}(x)\sim C_{0}x^{\frac{2b'(0)}{|\si'(0)|^{2}}-2}$ as $x\to0$ for some $C_{0}>0$. Note that under the assumption $\La_{0}>0$ one has $2-\frac{2b'(0)}{|\si'(0)|^{2}}<1$.

% {\rd Set 
% $$
% R_{0}:=\frac{2b'(0)}{|\si'(0)|^{2}}.
% $$
% This quantity is frequently used in the sequel and plays the role of the basic reproduction number because $R_{0}<1$, $R_{0}=1$ and $R_{0}>1$ if and only if $\La_{0}<0$, $\La_{0}=0$ and $\La_{0}>0$, respectively. 
% }

Our second result establishes asymptotic bounds for $\la_{\ep,1}$ and $\la_{\ep,2}$, the first two eigenvalues of $-\LL_{\ep}$. Throughout this paper, for positive numbers $A_{\ep}$ and $B_{\ep}$ indexed by $\ep$, we write 
$$
A_{\ep}\approx_{\ep}B_{\ep},\quad A_{\ep}\lesssim_{\ep} B_{\ep}\quad\text{and}\quad A_{\ep}\gtrsim_{\ep} B_{\ep}
$$ 
if $\lim_{\ep\to0}\frac{A_{\ep}}{B_{\ep}}=1$, $\limsup_{\ep\to0}\frac{A_{\ep}}{B_{\ep}}\leq1$ and  $\liminf_{\ep\to0}\frac{A_{\ep}}{B_{\ep}}\geq1$, respectively.

\begin{thmx}\label{thm-bounds-eigenvalues}
Assume {\bf (H)}. Then, $\lim_{\ep\to 0}\la_{\ep,1}=0$. Moreover, the following hold.
\begin{enumerate}
    \item[\rm(1)] If $\Lambda_0<0$, then there is $C>0$ such that $\la_{\ep,1}\gtrsim_{\ep} \frac{C}{|\ln \ep|}$.
    
    \item[\rm(2)] If $\Lambda_0>0$, then 
\begin{itemize}
\item for each $0<\ga\ll 1$, there holds 
$$
\ep^{(1+\ga)\frac{4b'(0)}{|\si'(0)|^2}-2}\lesssim_{\ep} \la_{\ep,1}\lesssim_{\ep} \ep^{(1-\ga)\frac{2b'(0)}{|\si'(0)|^2}-1};
$$
\item $0<\liminf_{\ep\to 0}\la_{\ep,2}\leq\limsup_{\ep\to 0}\la_{\ep,2}<\infty$.
\end{itemize}
\end{enumerate}
\end{thmx}

\begin{rem}\label{rem-eigenvalues-asymptotic}
We offer some comments regarding Theorem \ref{thm-bounds-eigenvalues}.
\begin{itemize}
    \item[\rm(1)] We first exhibit the significant effects  the environmental noise $\si(X^{\ep}_t) d{B}_t$ has by comparing \eqref{main-diffusion-eqn} with
\begin{equation}\label{main-diffusion-eqn-only-demographic}
d\tilde{X}^{\ep}_t=b(\tilde{X}^{\ep}_t)dt+\ep\sqrt{a(\tilde{X}^{\ep}_t)}dW_t\quad\text{in}\quad[0,\infty).
\end{equation}
Just like $X_{t}^{\ep}$, the diffusion process $\tilde{X}^{\ep}_t$ reaches the extinction state $0$ in finite time almost surely and admits a unique QSD with extinction rate $\tilde{\la}_{\ep,1}$ with $-\tilde{\la}_{\ep,1}$ being the first eigenvalue of the (appropriately extended) generator of \eqref{main-diffusion-eqn-only-demographic}. It is shown in \cite[Theorem A]{JQSY21} that 
$\lim_{\ep\to0}\frac{\ep^{2}}{2}\ln\tilde{\la}_{\ep,1}=-d$ for some $d>0$ that can be computed in terms of $a$ and $b$. In particular, $\tilde{\la}_{\ep,1}$ is exponentially small in $\ep$. Hence, the asymptotic of $\la_{\ep,1}$ is fundamentally different from that of $\tilde{\la}_{\ep,1}$, manifesting the significance of the environmental noise. More importantly, turning an exponentially small extinction rate into a sub-exponentially small one, greatly improves the observability of the extinction of a species, making \eqref{main-diffusion-eqn} a much better model than \eqref{main-diffusion-eqn-only-demographic}. 

The effects of the environmental noise extend to the extinction time, especially, the mean extinction time, thanks to the relationship between the extinction time $\E_{\bullet}^{\ep}[T_{0}^{\ep}]$ and the extinction rate $\la_{\ep,1}$. See \eqref{met-qsd} and Corollary \ref{c:ext}.

    \item[\rm(2)]  When $\La_{0}<0$, we believe that the lower bound of order $\frac{1}{|\ln\ep|}$ for $\la_{\ep,1}$ is sharp in the sense that there is $\tilde{C}>0$ such that $\la_{\ep,1}\lesssim_{\ep}\frac{\tilde{C}}{|\ln\ep|}$. We offer some further explanation in Remark \ref{rem-conjecture} after Corollary \ref{corx-first-eigenvalue-sharp-asymptotic}.

    \item[\rm(3)] When $\La_{0}>0$, the upper bound of $\la_{\ep,1}$ is improved to $\la_{\ep,1}\lesssim_{\ep} \ep^{(1-\ga)\frac{4b'(0)}{|\si'(0)|^2}-2}$ for each $0<\ga\ll1$ in Corollary \ref{corx-first-eigenvalue-sharp-asymptotic}. This says that the leading order of $\la_{\ep,1}$ is $\ep^{\frac{4b'(0)}{|\si'(0)|^2}-2}$. The reason why we still include this as a main result is that its proof, relying only on the classical variational formula, is elementary, while the proof of Corollary \ref{corx-first-eigenvalue-sharp-asymptotic} uses heavy machinery (see comments after Corollary \ref{corx-first-eigenvalue-sharp-asymptotic} for details). 
    
    \item[\rm(4)] When $\La_{0}>0$, the asymptotic bounds of $\la_{\ep,1}$ and $\la_{\ep,2}$ imply that $\inf_{\ep}(\la_{\ep,2}-\la_{\ep,1})>0$. As mentioned earlier (or see Lemma \ref{qsd-existence-uniqueness}), $\la_{\ep,2}-\la_{\ep,1}$ is the exponential convergence rate of $\P^{\ep}_{\mu}\left[X_{t}^{\ep}\in \bullet|t<T_{0}^{\ep}\right]$ to $\mu_{\ep}$ as $t\to\infty$ whenever the initial distribution $\mu$ is compactly supported in $(0,\infty)$. These facts tell us that the distribution of $X_{t}^{\ep}$ quickly approaches $\mu_{\ep}$, then stays close to $\mu_{\ep}$ until the time scale $\frac{1}{\la_{\ep,1}}$, after which it finally relaxes to the extinction state. The multiscale dynamics is precisely characterized in Theorem \ref{thm-multiscale}.
\end{itemize}

\end{rem}

% {\bl

% As the first thought about the asymptotic of $\la_{\ep,1}$ and $\la_{\ep,2}$ as $\ep\to0$, we would guess that they converge to the first and second eigenvalues of $\LL_{0}$, which is an appropriate extension of the generator of $X_{t}^{0}$ or \eqref{eqn-diffusion-extrinsic}, and takes the following form when applied to functions in $C^{2}((0,\infty))$:
% \begin{equation*}\label{e:LE-environmental}
% \LL_{0}\phi=\frac{\si^2}{2}\phi''+b\phi',\quad\forall\phi\in C^{2}((0,\infty)).
% \end{equation*}
% }

In our third result, we characterize the multiscale dynamics of the distribution of $X_{t}^{\ep}$ in the case $\La_{0}>0$ as described in Remark \ref{rem-eigenvalues-asymptotic} (4). Denote by $\phi_{\ep,1}$ the positive eigenfunction of $-\LL_{\ep}$ associated with $\la_{\ep,1}$ subject to the normalization $\|\phi_{\ep,1}\|_{L^2(u^G_{\ep})}=1$ (see Lemma \ref{prop-spectral-structure}). Let $\PP((0,\infty))$ be the set of Borel probability measures on $(0,\infty)$. For $\mu\in\PP((0,\infty))$ and for any measurable function $f:(0,\infty)\to\R$ we write $\langle\mu,f\rangle:=\int_{0}^{\infty}f d\mu$.

\begin{thmx}\label{thm-multiscale}
Assume {\bf (H)} and $\Lambda_0>0$. For any $\KK\stst(0,\infty)$, there is $C=C(\KK)>0$ such that
\begin{equation*}
    \sup_{\substack{\mu\in \PP((0,\infty))\\\supp(\mu)\subset \KK}}\left\|\P^{\ep}_{\mu}[X^{\ep}_t\in \bullet]-\left[e^{-\la_{\ep,1}t }\langle \mu,\al_{\ep,1}\rangle \mu_{\ep}+\left(1-e^{-\la_{\ep,1}t }\langle \mu,\al_{\ep,1}\rangle\right)\de_0\right]\right\|_{TV}\leq Ce^{-\la_{\ep,2}t}
\end{equation*}
holds for all $t\geq0$ and $0<\ep\ll 1$, where $\al_{\ep,1}:=\|\phi_{\ep,1}\|_{L^1(u^G_{\ep})}\phi_{\ep,1}$ satisfies
$$
\lim_{\ep\to0}\al_{\ep,1}=1\quad\text{locally uniformly in}\,\,(0,\infty).
$$
% \begin{equation*}
%     \left|\E^{\ep}_{\mu}[f(X^{\ep}_t)\mathbbm{1}_{t<T^{\ep}_0}]-e^{-\la_{\ep,1}t}\langle \mu,\al_{\ep,1}\rangle\int_0^{\infty}f u_{\ep}dx\right|\leq C e^{-\la_{\ep,2}(t-1)}\|f\|_{\infty},\quad \forall t>1\andd 0<\ep<\ep_0,
% \end{equation*}
% where $\al_{\ep,1}:=\|\phi_{\ep,1}\|_{L^1(u^G_{\ep})}\phi_{\ep,1}$ is positive in $(0,\infty)$ and locally uniformly converges to $1$.
\end{thmx}
Built on the eigenfunction expansion of the Markov semigroup associated with $X_{t}^{\ep}$ before hitting $0$, Theorem \ref{thm-multiscale} establishes a sharp estimate quantifying the total variation distance between the distribution of $X_{t}^{\ep}$ and the convex combination of the QSD $\mu_{\ep}$ and the extinction state $\de_{0}$. The locally uniform limit $\lim_{\ep\to0}\al_{\ep,1}=1$ and the fact that the constant $C$ is independent of $\ep$ are what make this estimate powerful. Together with the asymptotic bounds of $\la_{\ep,1}$ and $\la_{\ep,2}$ in Theorem \ref{thm-bounds-eigenvalues} (2), Theorem \ref{thm-multiscale} has the following important dynamical implications: 
\begin{itemize}
    \item if $t_{\ep}^{(1)}<t_{\ep}^{(2)}$ are such that $\lim_{\ep\to0}t_{\ep}^{(1)}=\infty$ and $\lim_{\ep\to0}\la_{\ep,1}t_{\ep}^{(2)}=0$, then
    $$
    \lim_{\ep\to0}\sup_{t\in\left[t_{\ep}^{(1)},t_{\ep}^{(2)}\right]}\sup_{\substack{\mu\in \PP((0,\infty))\\\supp(\mu)\subset \KK}}\left\|\P^{\ep}_{\mu}[X^{\ep}_t\in \bullet]-\mu_{\ep}\right\|_{TV}=0;
    $$
    
    \item if $t_{\ep}^{(3)}$ is such that $\lim_{\ep\to0}\la_{\ep,1}t_{\ep}^{(3)}=\infty$, then
        $$
    \lim_{\ep\to0}\sup_{t\in\left[t_{\ep}^{(3)},\infty\right)}\sup_{\substack{\mu\in \PP((0,\infty))\\\supp(\mu)\subset \KK}}\left\|\P^{\ep}_{\mu}[X^{\ep}_t\in \bullet]-\de_{0}\right\|_{TV}=0.
    $$
\end{itemize}

% It is well known that if we start with an initial distribution given by the QSD $\mu_\ep$, the time to extinction is exponential, that is
% \begin{equation}\label{e:ext}
% \P^{\ep}_{\mu_{\ep}}[T_0^{\ep}>t]=e^{-\lambda_{\ep,1}t},\quad \forall t\geq0.
% \end{equation}

Theorems \ref{thm-bounds-eigenvalues} (2) and \ref{thm-multiscale} have as immediate consequences the expected but far-reaching asymptotic reciprocal relationship between the extinction time $T_{0}^{\ep}$ and the extinction rate $\la_{\ep,1}$, and the asymptotic distribution of the normalized extinction time $\frac{T^{\ep}_0}{\E_{\bullet}^{\ep}[T^{\ep}_0]}$. 

\begin{corx} \label{c:ext}
Assume {\bf (H)} and $\Lambda_0>0$. For each $\mu\in \PP((0,\infty))$ having compact support in $(0,\infty)$, there exists $C=C(\mu)>0$ such that
    $$
    \left|\P^{\ep}_{\mu}[T^{\ep}_0>t]-e^{-\la_{\ep,1}t}\langle\mu,\al_{\ep,1}\rangle\right|\leq Ce^{-\la_{\ep,2}t},\quad\forall t>0\andd0<\ep\ll1,
    $$
or equivalently,
    $$
    \left|\P^{\ep}_{\mu}[\la_{\ep,1}T^{\ep}_0>t]-e^{-t}\langle\mu,\al_{\ep,1}\rangle\right|\leq Ce^{-\frac{\la_{\ep,2}}{\la_{\ep,1}}t},\quad\forall t>0\andd0<\ep\ll1.
    $$
In particular, 
\begin{itemize}
    \item $\lim_{\ep\to0}\P^{\ep}_{\mu}[\la_{\ep,1}T^{\ep}_0>t]=e^{-t}$ locally uniformly in $t\in(0,\infty)$;
    
    \item $\E^{\ep}_{\mu}[T^{\ep}_0]\approx_{\ep}\frac{1}{\la_{\ep,1}}$;
    
    \item $\lim_{\ep\to0}\P^{\ep}_{\mu}\left[\frac{T^{\ep}_0}{\E_{\mu}^{\ep}[T^{\ep}_0]}>t\right]=e^{-t}$ locally uniformly in $t\in(0,\infty)$.
\end{itemize}
\end{corx}
\begin{proof}
Let $\mu$ be as in the statement. Since $\P^{\ep}_{\mu}[X^{\ep}_t\in (0,\infty)]=\P^{\ep}_{\mu}[T^{\ep}_0>t]$, $\mu_{\ep}((0,\infty))=1$ and $\de_{0}((0,\infty))=0$, we apply Theorem \ref{thm-multiscale} and the definition of the total variation distance to find some $C=C(\mu)>0$ such that $\left|\P^{\ep}_{\mu}[T^{\ep}_0>t]-e^{-\la_{\ep,1}t}\langle\mu,\al_{\ep,1}\rangle\right|\leq Ce^{-\la_{\ep,2}t}$ for all $t>0$. Replacing $t$ by $\frac{t}{\la_{\ep,1}}$ leads to
$$
\left|\P^{\ep}_{\mu}[\la_{\ep,1}T^{\ep}_0>t]-e^{-t}\langle\mu,\al_{\ep,1}\rangle\right|\leq Ce^{-\frac{\la_{\ep,2}}{\la_{\ep,1}}t},\quad\forall t>0.
$$
In particular, $\lim_{\ep\to0}\P^{\ep}_{\mu}[\la_{\ep,1}T^{\ep}_0>t]=e^{-t}$ locally uniformly in $t\in(0,\infty)$.

% for any $0<\de\ll 1$, there exists $t_0=t_0(\de)>0$ such that
% $$
% \sup_{0<\ep\ll 1}e^{(\la_{\ep,1}-\la_{\ep,2})t}\leq \frac{\de}{2},\quad \forall  t\geq t_0.
% $$
% Since $\lim_{\ep\to 0}\langle\mu,\al_{\ep,1}\rangle=1$, we see from \eqref{eqn-2022-06-05} and the above inequality that
% $$
% (1-\de)e^{-\la_{\ep,1}t}\lesssim_{\ep} \P^{\ep}_{\mu}[T^{\ep}_0>t]\lesssim_{\ep}(1+\de)e^{-\la_{\ep,1}t},\quad \forall t>t_0.
% $$
%letting $\ep\to 0$ in the above inequality yields $\lim_{\ep\to 0}e^{\la_{\ep,1}t} \P^{\ep}_{\mu}[T^{\ep}_0>t]=1$ for $t>0$.

Integrating the above inequality with respect to $t$ over $(0,\infty)$ yields 
$$
\left|\la_{\ep,1}\E^{\ep}_{\mu}[T^{\ep}_0]-\langle \mu, \al_{\ep,1}\rangle\right|\leq C\frac{\la_{\ep,1}}{\la_{\ep,2}}. 
$$
This together with Theorem \ref{thm-bounds-eigenvalues} (2) and $\lim_{\ep\to 0}\langle \mu, \al_{\ep,1}\rangle=1$ (by Theorem \ref{thm-multiscale}) implies that $\lim_{\ep\to 0}\la_{\ep,1}\E^{\ep}_\mu [T^{\ep}_0]=1$. The remaining result follows immediately.
\end{proof}

Corollary \ref{c:ext} says in particular that the normalized extinction time $\frac{T^{\ep}_0}{\E_{\mu}^{\ep}[T^{\ep}_0]}$ weakly converges to an exponential random variable of mean $1$ as $\ep\to0$. The asymptotic reciprocal relationship $\E^{\ep}_{\mu}[T^{\ep}_0]\approx_{\ep}\frac{1}{\la_{\ep,1}}$ is a fundamental principle connecting the asymptotics of $\E_{\mu}^{\ep}[T_{0}^{\ep}]$ and $\la_{\ep,1}$ -- this allows using information about one of these quantities to analyze the other one. In particular, (non-sharp) asymptotic bounds of the mean extinction time $\E_{\bullet}^{\ep}[T_{0}^{\ep}]$ in the case $\La_{0}>0$ can be obtained from the asymptotic bounds of $\la_{\ep,1}$ in Theorem \ref{thm-bounds-eigenvalues} (2). However, we wanted to improve this and get sharp bounds.

Our last result is devoted to the investigation of the sharp asymptotic bounds of the mean extinction time $\E_{\bullet}^{\ep}[T_0^\ep]$.

\begin{thmx}\label{thm-asymptotic-mean-extinction}
Assume {\bf(H)}. The following hold for each $\mu\in \PP((0,\infty))$ having compact support in $(0,\infty)$.
\begin{enumerate}
   \item If $\La_0<0$, then there exist $C_1,C_2>0$ such that
   $$
  C_1|\ln \ep|\lesssim_{\ep}  \E^{\ep}_{\mu}[T^{\ep}_{0}]\lesssim_{\ep} C_2|\ln \ep|. 
   $$
   
   \item If $\La_0>0$, then for each $0<\ga\ll 1$, 
   $$
   \ep^{2-(1-\ga)\frac{4b'(0)}{|\si'(0)|^2}}\lesssim_{\ep}\E^{\ep}_{\mu}[T^{\ep}_0]\lesssim_{\ep} \ep^{2-(1+\ga)\frac{4b'(0)}{|\si'(0)|^2}}.
   $$
\end{enumerate}
\end{thmx}

Theorem \ref{thm-asymptotic-mean-extinction} is established by adopting a probabilistic approach focusing on analyzing the behaviours of $X_{t}^{\ep}$ near $0$. It is independent of Theorems \ref{thm-concentration}-\ref{thm-multiscale}. 

As an immediate consequence of Corollary \ref{c:ext} and Theorem \ref{thm-asymptotic-mean-extinction} (2), we get the following sharp asymptotic of $\la_{\ep,1}$ when $\La_{0}>0$, improving the one given in Theorem \ref{thm-bounds-eigenvalues} (2).

\begin{corx}\label{corx-first-eigenvalue-sharp-asymptotic}
Assume {\bf(H)} and $\La_{0}>0$. Then, for each $0<\ga\ll1$,
$$
\ep^{(1+\ga)\frac{4b'(0)}{|\si'(0)|^2}-2}\lesssim_{\ep}\la_{\ep,1}\lesssim_{\ep}\ep^{(1-\ga)\frac{4b'(0)}{|\si'(0)|^2}-2}.
$$
\end{corx}

\begin{rem}\label{rem-conjecture}
When $\La_{0}<0$, we are unable to establish the relationship $\E_{\mu}^{\ep}[T_{0}^{\ep}]\approx_{\ep}\frac{1}{\la_{\ep,1}}$ for $\mu\in \PP((0,\infty))$ having compact support in $(0,\infty)$, and hence, can not apply Theorem \ref{thm-asymptotic-mean-extinction} (1) to conclude $\frac{C_{1}}{|\ln \ep|}\lesssim_{\ep}\la_{\ep,1}\lesssim_{\ep} \frac{C_{2}}{|\ln \ep|}$. Nonetheless, we believe that the lower bound for $\la_{\ep,1}$ obtained in Theorem \ref{thm-bounds-eigenvalues} (1) is sharp. 
\end{rem}
%%%%%%%%%%%%%%%%%%%%%%%%%%%%%%%%%
%%%%%%%%%%%%%%%%%%%%%%%%%%%%%%%%%

\section{\bf Preliminary} \label{s:prelim}

This is a service section. We collect basic materials for later purposes.

%%%%%%%%%%%%%%%%%%%%%%%%%%%%%%%%%%%%%

\subsection{Generator, spectral theory and dynamics}\label{subsec-generator-spectrum-dynamics}

In this subsection, we present some general results about the spectral theory of the generator of $X_{t}^{\ep}$ and the dynamics of the corresponding semigroup.  

We start with the rigorous formalism of the generator of $X_{t}^{\ep}$. Set $$
\al_{\ep}:=\ep^2 a+\si^2\quad\text{on}\quad[0,\infty),\qquad V_{\ep}:=-\int_1^{\bullet}\frac{b}{\al_{\ep}}ds\quad\text{on}\quad(0,\infty).  
$$
Thanks to {\bf (H)}, we have
\begin{equation}\label{limit-alpha-V}
\begin{split}
\lim_{\ep\to0}\al_{\ep}=\si^2 \quad\text{in}\quad C^{2}([0,\infty)),\qquad \lim_{\ep\to0}V_{\ep}=-\int_1^{\bullet}\frac{b}{\si^2}ds\quad\text{in}\quad C^{2}((0,\infty)).
\end{split}
\end{equation}

Consider the symmetric quadratic form $\EE_{\ep}:C_{0}^{\infty}((0,\infty))\times C_{0}^{\infty}((0,\infty))\to\R$ defined by
$$
\EE_{\ep}(\phi,\psi)=\frac{1}{2}\int_{0}^{\infty}\al_{\ep}\phi'\psi'u^{G}_{\ep}dx,\quad \forall \phi,\psi\in C^{\infty}_{0}((0,\infty)),
$$
where 
$$
u_{\ep}^{G}:=\frac{1}{\al_{\ep}}e^{-2V_{\ep}}=\frac{1}{\al_{\ep}}e^{2\int_1^{\bullet}\frac{b}{\al_{\ep}}ds}\quad\text{in}\quad (0,\infty)
$$ 
is the non-integrable Gibbs density. The non-integrability of $u_{\ep}^{G}$ comes from the singularity of order $\frac{1}{x}$ near $0$. Recall $u_{0}^{G}$ from \eqref{gibbs-environmental-only} and note that clearly one has
\begin{equation}\label{gibbs-to-gibbs}
    \lim_{\ep\to0}u_{\ep}^{G}=u_{0}^{G}\quad\text{in}\quad C^{2}((0,\infty)).
\end{equation}
This fact is frequently used in the sequel. Under {\bf (H)}, it is not hard to check that the form $\EE_{\ep}$ is Markovian and closable (see e.g. \cite{Fukushima80}). Its smallest closed extension, still denoted by $\EE_{\ep}$, is a Dirichlet form with domain $D(\EE_{\ep})$ being the closure of $C_0^{\infty}((0,\infty))$ under the norm $\|\phi\|^2_{D(\EE_{\ep})}:=\|\phi\|^2_{L^2(u_{\ep}^G)}+\EE_{\ep}(\phi,\phi)$, where $L^{2}(u_{\ep}^{G}):=L^{2}((0,\infty),u_{\ep}^{G}dx)$. Denote by $(\LL_{\ep}, D(\LL_{\ep}))$ the non-positive self-adjoint operator associated with $(\EE_{\ep}, D(\EE_{\ep}))$, that is, 
\begin{equation*}\label{symmetric-form-vs-generator}
\EE_{\ep}(\phi,\psi)=\lan-\LL_{\ep}\phi,\psi\ran_{L^{2}(u_{\ep}^{G})},\quad \forall\phi\in D(\LL_{\ep}),\,\,\psi\in D(\EE_{\ep}),
\end{equation*}
where  
$$
D(\LL_{\ep}):=\left\{u\in D(\EE_{\ep}): \exists f\in L^2(u^G_{\ep})\,\, \text{s.t.}\,\,\EE_{\ep}(u,\phi)=\langle f,\phi\rangle_{L^2(u^G_{\ep})},\forall \phi\in D(\EE_{\ep})\right\}.
$$ 
It is informative to mention that
$$
\LL_{\ep}\phi=\frac{1}{2}(\ep^2a+\si^2)\phi''+b\phi',\quad\forall\phi\in C_{0}^{\infty}((0,\infty)).
$$
The operator $\LL_{\ep}$ is a self-adjoint extension in $L^2(u^G_{\ep})$ of the generator of \eqref{main-diffusion-eqn}.

In the next result, we collect basic properties about the spectrum of $-\LL_{\ep}$ and the semigroup generated by $\LL_{\ep}$.

\begin{lem}[\cite{CCLMMS09,JQSY21}]\label{prop-spectral-structure}
Assume {\bf (H)}. For each $0<\ep\ll1$, the following hold.
\begin{enumerate}
\item[\rm(1)] $-\LL_{\ep}$ has purely discrete spectrum contained in $(0,\infty)$ and listed as follows: $\la_{\ep,1}<\la_{\ep,2}<\la_{\ep,3}<\cdots\to\infty$.

\item[\rm(2)] Each $\la_{\ep,i}$ is associated with a unique eigenfunction $\phi_{\ep,i}\in D(\LL_{\ep})\cap L^1(u_{\ep}^G)\cap C^{2}((0,\infty))$ subject to the normalization $\|\phi_{\ep,i}\|_{L^{2}(u_{\ep}^{G})}=1$. Moreover, $\phi_{\ep,1}>0$.

\item[\rm(3)] The set $\{\phi_{\ep,i},i\in\N\}$ forms an orthonormal basis of $L^{2}(u_{\ep}^{G})$.

\item[\rm(4)] $\LL_{\ep}$ generates a positive analytic semigroup $(P^{\ep}_t)_{t\geq 0}$ of contractions on $L^2(u_{\ep}^G)$ having the stochastic representation (or Feynman–Kac formula):
$$
P^{\ep}_tf=\E^{\ep}_{\bullet}[f(X_t^{\ep})\mathbbm{1}_{t<T^{\ep}_0}],\quad \forall f\in L^2(u_{\ep}^G)\cap C_b((0,\infty))\andd t\geq0.
$$

\item[\rm(5)] For each $k\in\N$, $f\in L^2(u^{G}_{\ep})$ and $t>0$,
\begin{equation}\label{eqn-2020-8-13-4}
    \begin{split}
        P^{\ep}_tf=\sum_{i=1}^{k-1}e^{-\la_{\ep,i}t}\langle f,\phi_{\ep,i}\rangle_{L^2(u_{\ep}^G)}\phi_{\ep,i}+P^{\ep}_tQ^{\ep}_kf,
    \end{split}
\end{equation}
where $Q^{\ep}_k$ is the spectral projection of $\LL_{\ep}$ corresponding to $\{-\la_{\ep,j}\}_{j\geq k}$. Moreover, 
$$
\|P^{\ep}_tQ^{\ep}_k\|_{L^2(u^{G}_{\ep})\to L^2(u^{G}_{\ep})}\leq e^{-\la_{\ep,k}t},\quad t\geq0.
$$

\item[\rm(6)] For each $f\in C_b((0,\infty))$, the stochastic representation in {\rm (4)} and \eqref{eqn-2020-8-13-4} hold pointwisely.
\end{enumerate}
\end{lem}

The following result addressing the uniform-in-$\ep$ boundedness of the $i$-th eigenvalue of $-\LL_{\ep}$ is useful.

\begin{lem}\label{lem-upper-bound-eigenvalues}
Assume {\bf (H)}. For each $i\in\N$, there holds $\limsup_{\ep\to 0}\la_{\ep,i}<\infty$.
\end{lem} 
\begin{proof}
Let $\{\phi_\ell\}_{\ell\in\N}\subset C^{\infty}((0,\infty))$ satisfy $\supp(\phi_{\ell})\subset ({\ell},{\ell}+1)$ and $\|\phi_{\ell}\|_{L^2(u^G_{\ep})}=1$. We find from \eqref{limit-alpha-V} that the limit $\ga_{\ell}:=\lim_{\ep\to 0}\EE_{\ep}(\phi_{\ell},\phi_{\ell})>0$ exists for each $\ell\in\N$.

Fix $i\in\N$ and set $S_i:=\text{span}\{\phi_1,\dots, \phi_{i}\}$. Since $-\LL_{\ep}$ is self-adjoint in $L^2(u^G_{\ep})$, the Min-Max principle says in particular 
\begin{equation*}\label{eqn-2022-03-18-2}
    \la_{\ep,i}\leq\max_{\phi\in S_i}\frac{\lan-\LL_{\ep}\phi,\phi\ran_{L^{2}(u_{\ep}^{G})}}{\|\phi\|^2_{L^2(u^G_{\ep})}}=\max_{\phi\in S_i}\frac{\EE_{\ep}(\phi,\phi)}{\|\phi\|^2_{L^2(u^G_{\ep})}}.
\end{equation*}
Note that each element $\phi\in S_i$ can be written as $\phi:=\sum_{\ell=1}^i c_{\ell}\phi_{\ell}$ for some $c_{\ell}\in \R$, $\ell=1,\dots, i$. As the supports of $\{\phi_{\ell}\}_{\ell}$ are disjoint, we calculate $\|\phi\|^2_{L^2(u^G_{\ep})}=\sum_{\ell=1}^i c_{\ell}^2$ and $\EE_{\ep}(\phi,\phi)=\sum_{\ell=1}^i c_{\ell}^2 \EE_{\ep}(\phi_{\ell},\phi_{\ell})$. It follows that
\begin{equation*}
    \begin{split}
        \la_{\ep,i}\leq\max_{c_{\ell}\in\R,\ell=1,\dots,i}\frac{\sum_{\ell=1}^i c_\ell^2 \EE_{\ep}(\phi_{\ell},\phi_{\ell})}{\sum_{\ell=1}^n c_{\ell}^2 }&\leq \max_{\ell=1,\dots, i}\EE_{\ep}(\phi_{\ell},\phi_{\ell}),
    \end{split}
\end{equation*}
leading to $\limsup_{\ep\to0}\la_{\ep,i}\leq\max_{\ell=1,\dots, i}\ga_{\ell}$.
\end{proof}

%%%%%%%%%%%%%%%%%%%%%%
\subsection{Schr\"odinger operators}\label{subsec-schrodinger-eqn}

In this subsection, we follow the canonical procedure (see e.g. \cite{CCLMMS09}) to derive the Schr\"odinger operator that is unitarily equivalent to $\LL_{\ep}$ and establish some properties of its potential. These results will play a significant technical role in the sequel.

Note that $X^{\ep}_t$ has the same distribution as the solution process of
\begin{equation}\label{main-diffusion-eqn-equivalent}
    d\tilde{X}^{\ep}_t=b(\tilde{X}^{\ep}_t)dt+\sqrt{\al_{\ep}(\tilde{X}^{\ep}_t)}d\tilde{W}_t\quad\text{in}\quad[0,\infty),
\end{equation}
where $\tilde{W}_t$ is a standard one-dimensional Brownian motion. Consider the change of variable 
$$
y=\xi_{\ep}(x):=\int_0^{x}\frac{1}{\sqrt{\al_{\ep}}}ds=\int_0^x \frac{1}{\sqrt{\ep^2 a+\si^2}}ds,\quad x\in (0,\infty).
$$
Clearly, $\xi_{\ep}$ is increasing and satisfies $\xi_{\ep}(0+)=0$.
Set $y_{\ep,\infty}:=\xi_{\ep}(\infty)$. Then, $\xi_{\ep}:(0,\infty)\to(0,y_{\ep,\infty})$ is invertible. Its inverse is denoted by $\xi_{\ep}^{-1}:(0,y_{\ep,\infty})\to(0,\infty)$. This is the canonical transform converting the SDE \eqref{main-diffusion-eqn-equivalent} into the one with the simplest noise coefficient. More precisely, applying It\^{o}'s formula, we find that  $Y^{\ep}_t:=\xi_{\ep}(\tilde{X}^{\ep}_t)$ solves
\begin{equation}\label{SDE-simplest-noise}
    d Y^{\ep}_t=q_{\ep}(Y^{\ep}_t) dt+d\tilde{W}_t,
\end{equation}
where $q_{\ep}:=\left(\frac{b}{\sqrt{\al_{\ep}}}-\frac{\al'_{\ep}}{4\sqrt{\al_{\ep}}}\right)\circ \xi^{-1}_{\ep}$. 

Set $v_{\ep}^G:=(u^G_{\ep}\sqrt{\al_{\ep}})\circ\xi^{-1}_{\ep}$ and $L^{2}(v_{\ep}^{G}):=L^{2}((0,y_{\ep,\infty}),v_{\ep}^{G}dy)$. The generator of \eqref{SDE-simplest-noise} is given by 
$$
\LL^Y_{\ep}:=\frac{1}{2}\frac{d^2}{dy^2} +q_{\ep}(y)\frac{d}{dy}\quad\text{in}\quad L^{2}(v_{\ep}^{G}).
$$
It is straightforward to check that $\LL_{\ep}^{Y}$ is unitarily equivalent to $\LL_{\ep}$. More precisely, there holds $U_{\ep}\LL_{\ep}=\LL_{\ep}^{Y}U_{\ep}$, where $U_{\ep}: L^{2}(u_{\ep}^{G})\to L^{2}(v_{\ep}^{G})$, $f\mapsto f\circ\xi^{-1}_{\ep}$ is unitary.

Now, consider the Schr\"{o}dinger operator
\begin{equation}\label{semi-classical-S-op}
\LL_{\ep}^{S}:=\frac{1}{2}\frac{d^{2}}{dy^{2}}-\frac{1}{2}\left(q^2_{\ep}(y)+q'_{\ep}(y)\right)\quad\text{in}\quad L^{2}((0,y_{\ep,\infty})).
\end{equation}
It is not hard to check that $\tilde{U}_{\ep}\LL_{\ep}^{Y}=\LL_{\ep}^{S}\tilde{U}_{\ep}$, where $\tilde{U}_{\ep}:L^{2}(v_{\ep}^{G})\to L^{2}((0,y_{\ep,\infty}))$, $ f\mapsto f\sqrt{v_{\ep}^{G}}$ is unitary.

%In particular, if we set $\psi_{\ep}:=\tilde{U}_{\ep}U_{\ep}\phi$ for each $\phi\in C^2((0,\infty))$ and $\ep>0$, then  
%$$
%\LL^S_{\ep}\psi_{\ep,i}=\tilde{U}_{\ep}U_{\ep}\LL_{\ep}\psi\quad \andd\quad \int_{\xi_{\ep}(x_1)}^{\xi_{\ep}(x_2)}|\psi_{\ep}|^2 dy=\int_{x_1}^{x_2}|\phi|^2 dx,\quad \forall  0<x_1<x_2.
%$$

We include the following commutative diagram for readers' convenience:
$$
\begin{tikzcd}
L^{2}(u_{\ep}^{G}) \arrow[d, "\LL_{\ep}"] \arrow[r, "U_{\ep}"]  &\quad  L^{2}(v_{\ep}^{G})\arrow[d, "\LL_{\ep}^{Y}"] \arrow[r, "\tilde{U}_{\ep}"]   &  \quad L^{2}((0,y_{\ep,\infty}))\arrow[d, "\LL_{\ep}^{S}"] \\
L^{2}(u_{\ep}^{G})  \arrow[r, "U_{\ep}"]  & \quad L^{2}(v_{\ep}^{G}) \arrow[r, "\tilde{U}_{\ep}"]  &\quad L^{2}((0,y_{\ep,\infty}))
\end{tikzcd}
$$

We point out that rigorous definitions of $\LL_{\ep}^{Y}$ and $\LL_{\ep}^{S}$ can be done using quadratic forms as done in Subsection \ref{subsec-generator-spectrum-dynamics} for  $\LL_{\ep}$. By the unitary transforms, the domains of $\LL_{\ep}^{Y}$ and $\LL_{\ep}^{S}$ are respectively given by $U_{\ep}D(\LL_{\ep})$ and $\tilde{U}_{\ep}U_{\ep}D(\LL_{\ep})$, and the domains of quadratic forms associated with $\LL_{\ep}^{Y}$ and $\LL_{\ep}^{S}$ are respectively given by $U_{\ep}D(\EE_{\ep})$ and $\tilde{U}_{\ep}U_{\ep}D(\EE_{\ep})$.

The potential of the Schr\"{o}dinger operator $\LL^S_{\ep}$ is 
denoted by 
$$
W_{\ep}:= \frac{1}{2}\left(q^2_{\ep}+q'_{\ep}\right)\quad\text{on}\quad (0,y_{\ep,\infty}).
$$
Some elementary properties of $W_{\ep}$ are collected in the following result.

\begin{lem}\label{properties-schrodinger-potential}
Assume {\bf (H)}. The following hold.
\begin{enumerate}
\item[\rm(1)] $2W_{\ep}\circ\xi_{\ep}=\frac{3|\al'_{\ep}|^2}{16\al_{\ep}}-\frac{\al_{\ep}''}{4}+b'-\frac{b\al_{\ep}'}{\al_{\ep}}+\frac{b^2}{\al_{\ep}}$. 

\item[\rm(2)] There exist $y_1\in (0,\infty)$ and $C>0$ such that $\inf_{\ep}W_{\ep}(y)\geq \frac{C}{y^2}$ for all $y\in(0,y_{1}]$.

\item[\rm(3)] There exists $x_*\in (0,\infty)$ such that
$$
W_{\ep}(y)\geq \frac{b^{2}(\xi^{-1}_{\ep}(y))}{4\si^2(\xi^{-1}_{\ep}(y))},\quad\forall y\in[\xi_{\ep}(x_*),y_{\ep, \infty})\andd 0<\ep\ll1.
$$

\item[\rm(4)] The family $\{W_{\ep}\}_{\ep}$ is uniformly lower bounded, that is, $\inf_{\ep}\min W_{\ep}>-\infty$.
\end{enumerate}
\end{lem}
\begin{proof}
(1) Straightforward calculations yield
\begin{equation*}
    \begin{split}
        q^2_{\ep}\circ\xi_{\ep}=\frac{b^2}{\al_{\ep}}-\frac{b\al'_{\ep}}{2\al_{\ep}}+\frac{|\al'_{\ep}|^2}{16\al_{\ep}},\quad q'_{\ep}\circ\xi_{\ep}=b'-\frac{b\al'_{\ep}}{2\al_{\ep}}-\frac{\al''_{\ep}}{4}+\frac{|\al'_{\ep}|^2}{8\al_{\ep}}.
    \end{split}
\end{equation*}
The expression for $W_{\ep}$ follows immediately. 
 
(2) Thanks to {\bf (H)}(1)-(3) and Taylor's expansion at $x=0$, 
\begin{equation*}
    \begin{split}
    \al_{\ep}(x)&=\ep^2 a(x)+\si^2(x)=\ep^2 a'(0)x+\left(\frac{\ep^2}{2}a''(0)+|\si'(0)|^2\right)x^2+o(x^2),\\
    \al'_{\ep}(x)&=\ep^2 a'(x)+2\si(x)\si'(x)=\ep^2 a'(0)+\left(\ep^2 a''(0)+2|\si'(0)|^2\right)x+o(x),\\
    \al''_{\ep}(x)&=\ep^2 a''(x)+2(|\si'(x)|^2+\si(x)\si''(x))=\ep^2 a''(0)+2|\si'(0)|^2+o(1),\quad\andd\\
     b(x)&=b'(0)x+o(1).
    \end{split}
\end{equation*}
For fixed $\de\in(0,\frac{1}{2})$ (to be specified), we find $0<\ka\ll1$ such that 
\begin{gather*}
        1-\de\leq \frac{\al_{\ep}(x)}{\ep^2 a'(0)x}\leq1+\de,\quad 1-\de\leq \frac{\al'_{\ep}(x)}{\ep^2 a'(0)}\leq1+\de,\\
        |\al''_{\ep}(x)|\leq 3|\si'(0)|^2,  \quad  b'(x)>0\quad \andd \quad 0<b(x)<2b'(0)x,\quad\forall  x\in (0,\ka\ep^2).
\end{gather*}
Hence, with $y=\xi_{\ep}(x)$,
\begin{equation*}
    \begin{split}
    2W_{\ep}(y)&\geq  \frac{3|\al'_{\ep}(x)|^2}{16\al_{\ep}(x)}-\frac{\al_{\ep}''(x)}{4}-\frac{b(x)\al_{\ep}'(x)}{\al_{\ep}(x)}\\
    &\geq \frac{3(1-\de)^2}{16(1+\de)}\frac{\ep^2 a'(0)}{x}-\frac{3}{4}|\si'(0)|^2-\frac{2b'(0)(1+\de)}{(1-\de)},\quad\forall x\in (0,\ka\ep^2).
    \end{split}
\end{equation*}
Since
$$
\frac{3}{4}|\si'(0)|^2+\frac{2b'(0)(1+\de)}{(1-\de)}<\frac{a'(0)}{18\ka}<\frac{1}{18}\frac{\ep^2 a'(0)}{x},\quad\forall x\in (0,\ka\ep^2), 
$$
where the first inequality is due to the smallness of $\ka$, we arrive at 
\begin{equation*}
    2W_{\ep}(y)\geq \frac{3(1-\de)^2}{16(1+\de)}\frac{\ep^2 a'(0)}{x}- \frac{1}{18}\frac{\ep^2a'(0)}{x}=\frac{1}{9}\frac{\ep^2a'(0)}{x},\quad\forall x\in (0,\ka\ep^2),
\end{equation*}
where we fixed $\de$ so that $\frac{(1-\de)^2}{(1+\de)}=\frac{8}{9}$ in the equality.

Note that 
$$
y=\xi_{\ep}(x)\geq \int_0^x \frac{ds}{\sqrt{(1+\de)\ep^2 a'(0)s}}=\frac{2\sqrt{x}}{\sqrt{(1+\de)\ep^2a'(0)}},\quad\forall x\in (0,\ka\ep^2),
$$
leading to 
$$
2W_{\ep}(y)\geq \frac{1}{9}\frac{4\ep^2 a'(0)}{(1+\de)\ep^2 a'(0) y^2}=\frac{4}{9(1+\de)}\frac{1}{y^2},\quad  \forall y\in (0,y_1),
$$
where $y_1:=\frac{2\sqrt{\ka}}{\sqrt{(1+\de)a'(0)}}\leq\xi_{\ep}(\ka\ep^2)$. 
This proves (2).

(3) Obviously, $2W_{\ep}\circ\xi_{\ep}\geq \frac{b^2}{\al_{\ep}}-\frac{b\al'_{\ep}}{\al_{\ep}}-\frac{\al''_{\ep}}{4}+b'$. By {\bf (H)}(4), there is $x_*>0$ such that
$$
\frac{b^2}{\al_{\ep}}\geq \frac{3b^2}{4\si^2},\quad 
\frac{\al'_{\ep}}{\al_{\ep}}\leq \frac{|b|}{8\si^2}\quad\andd\quad \frac{\al''_{\ep}}{4}-b'\leq \frac{b^2}{8\si^2}\quad \text{in}\quad (x_*,\infty).
$$
Then, $2W_{\ep}\geq \frac{b^2\circ\xi^{-1}_{\ep}}{2\si^2\circ\xi^{-1}_{\ep}}$ in $(\xi^{-1}_{\ep}(x_*),\infty)$, verifying (3). 

(4) Let $x_*$ be as in (3). The assumption {\bf (H)} implies $\sup_{[0,x_*]}\max\left\{|\al''_{\ep}|, |b'|,\frac{|b\al'_{\ep}|}{\al_{\ep}}\right\}<\infty$. The conclusion then follows from (1) and (3).
 \end{proof}

Lemma \ref{properties-schrodinger-potential} says that the potential $W_{\ep}$ is lower bounded and satisfies $W_{\ep}(y)\to\infty$ as $y\to0^{+}$ and $y_{\ep,\infty}^{-}$. Classical spectral theory of Schr\"{o}dinger operators then ensures that the spectrum of $-\LL_{\ep}^{S}$ is purely discrete, and so is that of $-\LL_{\ep}$ by the unitary equivalence.  This is the idea in \cite{CCLMMS09} of obtaining the spectral structure of $-\LL_{\ep}$.

%%%%%%%%%%%%%%%%%%%%%%%%%%

\subsection{Quasi-stationary distributions}

The existence and uniqueness of QSDs of $X_{t}^{\ep}$ and their properties are investigated in \cite{CCLMMS09} (see also \cite{SE07, LC12,MV12, KS12, CMSM13,Miura14,CV18,HK19,HQSY}). We summarize relevant results in the following lemma. Denote by $\LL^{*}_{\ep}$ the Fokker-Planck operator associated with $X_{t}^{\ep}$, namely,
$$
\LL^{*}_{\ep}\phi=\frac{1}{2}(\al_{\ep}\phi)''-(b\phi)',\quad\forall\phi\in C^{2}((0,\infty)),
$$
where we recall $\al_{\ep}:=\ep^2 a+\si^2$. Recall from Lemma \ref{prop-spectral-structure} that $\la_{\ep,1}$ and $\la_{\ep,2}$ are the first two eigenvalues of $-\LL_{\ep}$. The associated normalized eigenfunctions are denoted by $\phi_{\ep,1}$ and $\phi_{\ep,2}$ with $\phi_{\ep,1}>0$.

\begin{lem}[\cite{CCLMMS09}]\label{qsd-existence-uniqueness}
Assume {\bf (H)}. For each $0<\ep\ll 1$, $X^{\ep}_t$ admits a unique QSD $\mu_{\ep}$ with the extinction rate $\la_{\ep,1}$. Moreover, $\mu_{\ep}$ admits a positive density $u_{\ep}\in C^2((0,\infty))$ satisfying $\LL^{*}_{\ep} u_{\ep}=-\la_{\ep,1} u_{\ep}$ and given by $u_{\ep}=\frac{\phi_{\ep,1}u_{\ep}^G}{\|\phi_{\ep,1}\|_{L^1(u^G_{\ep})}}$. In addition, if $\mu\in\PP((0,\infty))$ has compact support in $(0,\infty)$, then for any $B\in\BB((0,\infty))$,
\begin{equation*}\label{convergence-rate-to-qsd}
    \begin{split}
        &\lim_{t\to\infty}e^{(\la_{\ep,2}-\la_{\ep,1})t}\left(\P^{\ep}_{\mu}\left[X_{t}^{\ep}\in B|t<T_{0}^{\ep}\right]-\mu_{\ep}(B)\right)\\
        &\qquad=\frac{\displaystyle\int_{0}^{\infty}\phi_{\ep,2} d\mu }{\displaystyle\int_{0}^{\infty}\phi_{\ep,1} d\mu }\left(\frac{\langle\mathbbm{1}_B,\phi_{\ep,2} \rangle_{L^2(u_{\ep}^G)}}{\|\phi_{\ep,1}\|_{L^1(u_{\ep}^G)}}-\frac{\langle \mathbbm{1}_B, \phi_{\ep,1}\rangle_{L^2(u_{\ep}^G)}\langle \mathbbm{1}, \phi_{\ep,2}\rangle_{L^2(u_{\ep}^G)}}{\|\phi_{\ep,1}\|^2_{L^1(u_{\ep}^G)}} \right).
    \end{split}
\end{equation*}
\end{lem}

We point out that $d\mu_{\ep}=u_{\ep}dx$ being a QSD of $X_{t}^{\ep}$ is a direct consequence of Lemma \ref{prop-spectral-structure}. Verifying the uniqueness is however much more challenging. In \cite{CCLMMS09}, the authors achieve this by exploring the so-called ``coming down from infinity" saying that $\infty$ is an entrance boundary for $X_{t}^{\ep}$, and obtain a necessary and sufficient condition. As a result, they show for any initial distribution $\mu\in\PP((0,\infty))$ the conditioned dynamics $\P_{\mu}^{\ep}[X_{t}^{\ep}\in\bullet|t<T_{0}^{\ep}]$ converges to $\mu_{\ep}$ under the topology of weak convergence as $t\to\infty$. This can be improved to exponential convergence with rate $\la_{\ep,2}-\la_{\ep,1}$ if $\mu$ is compactly supported in $(0,\infty)$ as stated in Lemma \ref{qsd-existence-uniqueness}. 

The next result is a stepping stone to obtaining finer results of the QSD $\mu_{\ep}$ or its density $u_{\ep}$ near $0$.

\begin{lem}\label{bound-u_ep-near-0}
Assume {\bf (H)}. For each $0<\ep\ll 1$,  there holds $\limsup_{x\to 0}u_{\ep}(x)<\infty$.
\end{lem}
\begin{proof}
It is actually a special case of \cite[Corollary 3.1]{SWY}. Indeed, the authors consider in \cite{SWY} the following SDE:
\begin{equation}\label{sde-in-earlier-work}
d\tilde{X}_t=\tilde{b}(\tilde{X}_t)dt+\ep\sqrt{\tilde{a}(\tilde{X}_t)}dW_t,
\end{equation}
where $\tilde{b}:[0,\infty)\to\R$ satisfies $\tilde{b}\in C([0,\infty))\cap C^{1}((0,\infty))$, $\tilde{b}(0)=0$, $\tilde{b}(x)>0$ for all $0<x\ll1$, and $\tilde{b}(x)<0$ for all $x\gg1$, and $\tilde{a}:[0,\infty)\to[0,\infty)$ satisfies $\tilde{a}\in C^{2}([0,\infty))$, $\tilde{a}(0)=0$, $\tilde{a}>0$
on $(0,\infty)$ and $\int_{0}^{1}\frac{1}{\sqrt{\tilde{a}}}ds<\infty$. See  \cite[{\bf (A1)} and {\bf(A3)}]{SWY}; these are assumptions on $\tilde{a}$ and $\tilde{b}$ needed to prove \cite[Corollary 3.1]{SWY}. Assuming the existence of a QSD $\tilde{\mu}_{\ep}$ with density $\tilde{u}_{\ep}$ (whose regularity is guaranteed by the elliptic regularity) and extinction rate $\tilde{\la}_{\ep}$, the authors show that $\limsup_{x\to 0}\tilde{u}_{\ep}(x)<\infty$.
The proof given in \cite{SWY} is analytic and utilizes the eigen-equation satisfied by $\tilde{u}_{\ep}$, namely, $\frac{\ep^{2}}{2}(\tilde{a}\tilde{u}_{\ep})''-(\tilde{b}\tilde{u}_{\ep})'=-\tilde{\la}_{\ep}\tilde{u}_{\ep}$. In particular, the proof is insensitive to the form of the noise term in the SDE \eqref{sde-in-earlier-work}. It is crucial to mention that the above result is pointwise in $\ep$. 

In our case, Lemma \ref{qsd-existence-uniqueness} says that $u_{\ep}$ obeys $\LL^{*}_{\ep} u_{\ep}=-\la_{\ep,1} u_{\ep}$, that is, $\frac{1}{2}(\al_{\ep}u_{\ep})''-(bu_{\ep})'=-\la_{\ep,1} u_{\ep}$. It is easy to see from {\bf(H)}(1)-(3) that $b$ and $\frac{\al_{\ep}}{2}$($=\frac{1}{2}(\ep^{2}a+\si)$) satisfy conditions for $\tilde{b}$ and $\frac{\ep^{2}}{2}\tilde{a}$. As a result, we conclude $\limsup_{x\to 0}u_{\ep}(x)<\infty$.
\end{proof}

%%%%%%%%%%%%%%%%%%%%%%

\subsection{SDE with only environmental noise}\label{subsec-sde-environmental}

In this subsection, we consider the SDE \eqref{eqn-diffusion-extrinsic}, whose solutions generate the diffusion process $X_{t}^{0}$. Recall from \eqref{gibbs-environmental-only} that 
$u_0^G=\frac{1}{\si^2} e^{2\int_1^{\bullet}\frac{b}{\si^2}ds}$ in $(0,\infty)$, which is a non-normalized and not necessarily integrable Gibbs density associated with \eqref{eqn-diffusion-extrinsic} or $X_{t}^{0}$.

The following lemma addresses the integrability/non-integrability of $u_0^G$. Recall that the external Lyapunov exponent $\Lambda_{0}$ is defined in \eqref{e:lambda0}.

\begin{lem}\label{lem-integral-gibbs-zero-demographic}
Assume ${\bf(H)}$. Then, $u_{0}^{G}\in L^{1}((1,\infty))$, and $u_{0}^{G}\in L^{1}((0,1))$ if  $\La_0>0$ and $u_{0}^{G}\notin L^{1}((0,1))$ if  $\La_0<0$.
\end{lem}
\begin{proof}
We first prove $u_{0}^{G}\in L^{1}((1,\infty))$. Note that $u^G_0=e^{f}$, where $f=2\int_1^{\bullet}\frac{b}{\si^2}ds-\ln \si^2$. Clearly, $
f'=\frac{2b}{\si^2}-\frac{(\si^2)'}{\si^2}$. Since $\limsup_{x\to \infty}\frac{\si^2}{|b|}\frac{|(\si^2)'|}{\si^2}=0$ (by {\bf (H)}(4)) and $b(x)<0$ for $x\gg1$, there is $x_1\gg1$ such that $f'\leq \frac{b}{\si^2}$ in $[x_1,\infty)$. Thus, $u_0^G(x)\leq \exp\left\{f(x_1)+\int_{x_1}^x \frac{b}{\si^2}ds\right\}$ for all $x\geq x_{1}$. Thanks to $\lim_{x\to \infty}\frac{xb(x)}{\si^2(x)}=-\infty$ (by {\bf (H)}(4)), we find $M\gg 1$ and $x_2>x_1$ such that $\frac{xb(x)}{\si^2(x)}\leq -M$ for $x\geq x_{2}$. It follows that
\begin{equation*}
    \begin{split}
        u_0^G(x)&\leq \exp\left\{f(x_1)+\int_{x_1}^{x_2} \frac{b}{\si^2}ds-M\int_{x_2}^x \frac{1}{s}ds\right\}\\
        &=\exp\left\{f(x_1)+\int_{x_1}^{x_2} \frac{b}{\si^2}ds\right\}\left(\frac{x}{x_2}\right)^{-M},\quad \forall x\geq x_2. 
    \end{split}
\end{equation*}
The integrability of $u_0^G$ in $(1,\infty)$ follows. 

Thanks to {\bf (H)} (1)-(2), the Taylor expansions of $b$ and $\si$ near $0$ give $u_0^{G}(x)\approx x^{\ga}$ in the vicinity of $0$, where $\ga=\frac{2b'(0)}{|\si'(0)|^2}-2$. It follows the integrability (resp. non-integrability) of $u_0^G$ in $(0,1)$ when $\La_0>0$ (resp. $\La_0<0$). 
% It remains to study the integrability/non-integrability of $u_0^G$ in $(0,1)$. Assume  $\La_0>0$. By {\bf (H)}(1)-(2), there exist $\de>0$ and $x_0\in (0,1)$ such that $\frac{2(1-\de)b'(0)}{|\si'(0)|^2}>1$ and
% $$
% \frac{b(x)}{\si^2(x)}\geq \frac{2(1-\de)b'(0)}{|\si'(0)|^2}\frac{1}{x}\quad \andd \quad \si^2(x)\geq (1-\de)|\si'(0)|^2 x^2,\quad \forall x\in (0,x_0).
% $$
% Then, for each $x\in (0,x_0)$,
%     \begin{equation*}
%         2\int_1^x \frac{b}{\si^2}ds=2\int_1^{x_0} \frac{b}{\si^2}ds+2\int_{x_0}^x \frac{b}{\si^2}ds\leq 2\int_1^{x_0} \frac{b}{\si^2}ds+\frac{2b'(0)(1-\de)}{|\si'(0)|^2}\ln \left(\frac{x}{x_0}\right),
%     \end{equation*}
%     and thus, 
%     \begin{equation*}
%         u_0^G(x)\leq \frac{1}{(1-\de)|\si'(0)|^2}\frac{1}{x^2} \left(\frac{x}{x_0}\right)^{\frac{2b'(0)(1-\de)}{|\si'(0)|^2}} e^{2\int_1^{x_0} \frac{b}{\si^2}ds}. 
%     \end{equation*}
% Since $\frac{2b'(0)(1-\de)}{|\si'(0)|^2}-2>-1$, we find $u^G_0\in L^1(0,1)$.
% The non-integrability of $u_0^G$ in $(0,1)$ when $\La_0<0$ follows from similar arguments. This completes the proof. 
\end{proof}

The next result concerning the global dynamics of $X_{t}^{0}$ is classical (see e.g. \cite{Kallenberg02, EHS15, HN18}). 

\begin{prop}\label{prop-existence-uniqueness-sd}
Assume {\bf(H)}. Then, for any $\mu\in \PP((0,\infty))$ and $t\geq 0$, there holds
    $X^{0}_t>0$ $\P_{\mu}$-a.e. Furthermore, the following hold.
\begin{itemize}
    \item[\rm(1)] If $\La_0<0$, then $\de_0$ is the unique stationary distribution of $X_{t}^{0}$. Moreover, for any $\mu\in\PP((0,\infty))$, $\lim_{t\to\infty}X^0_{t}=0$ $\P_{\mu}$-a.e.
    
    \item[\rm(2)] If $\La_0>0$, then $X_{t}^{0}$ admits a unique stationary distribution $\mu_{0}$ with a positive density $u_{0}\in C^{2}((0,\infty))$ given by the normalized Gibbs density, that is, $u_0=\frac{u_0^G}{\|u_0^G\|_{L^1((0,\infty))}}$. Moreover, there is some $\ga>0$ such that for any $\mu\in\PP((0,\infty))$, there exists $C=C(\mu)>0$ such that
$$
\left\|\P_{\mu}^{0}\left[X_{t}^{0}\in\bullet\right]-\mu_{0}\right\|_{TV}\leq Ce^{-\ga t},\quad\forall t\geq0.
$$
\end{itemize}
\end{prop}

%Under the assumptions in Proposition \ref{prop-existence-uniqueness-sd}, the space consisting of $L^2((0,1))$-integrable solutions to $\LL^*_0 u=0$ is one-dimensional with $u_0^G$ being a basis. In particular, there holds $u_0=\frac{1}{\|u_0^G\|_{L^1((0,\infty))}}u_0^G$. 

%Thanks to  $\frac{2b'(0)}{|\si'(0)|^2}>1$ given in {\bf (H)}(4), we show in the following that any $L^1$-integrable solution of \eqref{fpe-extrinsic} is proportional to $u^G_0:=\frac{1}{\si^2}e^{2\int_1^{\bullet}\frac{b}{\si^2}ds}$.

Denote by $\LL_{0}^{*}$ the Fokker-Planck operator associated with $X_{t}^{0}$, that is, 
$$
\LL^*_0\phi:=\frac{1}{2}(\si^2 \phi)''-(b \phi)',\quad \forall \phi\in C^2((0,\infty)).
$$
We need the following uniqueness result about solutions of the stationary Fokker-Planck equation $\LL^*_0 u=0$.

\begin{lem}\label{lem-unique-L^1-soln}
Assume {\bf (H)}{\rm(1)}-{\rm(2)} and $\La_0>0$. If $u\in C^2((0,\infty))\cap L^1((0,1))$ solves $\LL^*_0 u=0$, then $u=Cu^G_0$ for some $C\in\R$.
\end{lem}
\begin{proof}
Since $\LL^*_0 u=\frac{1}{2}(\si^2 u)''-(bu)'=0$, we integrate to find $C_1\in \R$ such that $\frac{1}{2}(\si^2 u)'-bu=\frac{1}{2}C_1$, which is rewritten as $(\si^2 u)'-\frac{2b}{\si^2}(\si^2 u)=C_1$. Applying the variation of constants formula yields the existence of $C_2\in\R$ such that 
\begin{equation*}
u(x)=\frac{C_1}{\si^2}\int_1^x e^{2\int_y^x\frac{b}{\si^2} ds}dy+\frac{C_2}{\si^2}e^{2\int_1^x\frac{b}{\si^2} ds}=:C_1 \RN{1}(x)+C_2 u_0^G(x),\quad\forall x\in (0,\infty).
\end{equation*}

By Lemma \ref{lem-integral-gibbs-zero-demographic}, $u_0^G\in L^1((0,1))$. We show that $\RN{1}$ is not integrable near $0$. Then, the assumption $u\in L^{1}((0,1))$ implies that $C_1=0$, leading to the conclusion. 

Let $0<\de\ll 1$ and set $\ka:=\frac{2(1+\de)b'(0)}{(1-\de)|\si'(0)|^2}$. Note that $\ka>1$ due to $\La_0>0$. By {\bf (H)}(1)-(2), there exists $x_*\in(0,1)$ such that $b(x)\leq (1+\de)b'(0)x$ and $1-\de\leq \frac{\si^2(x)}{|\si'(0)|^2 x^2}\leq 1+\de$ for all $x\in (0,x_*)$. Then, we derive from
$$
\int_y^{x}\frac{b}{\si^2}ds\geq \int_{y}^x \frac{(1+\de)b'(0)}{(1-\de)|\si'(0)|^2 s}ds=\frac{\ka}{2}\ln \left(\frac{x}{y}\right),\quad \forall 0<x<y<x_*
$$
that 
\begin{equation*}
    \begin{split}
       \int_{x}^{x_*}e^{2\int_{y}^{x}\frac{b}{\si^2}ds}dy\geq \int_{x}^{x_*} \left(\frac{x}{y}\right)^{\ka}dy=\frac{x^{\ka}}{\ka-1}\left(\frac{1}{x^{\ka-1}}-\frac{1}{x_*^{\ka-1}}\right),\quad \forall x\in (0,x_*).
    \end{split}
\end{equation*}
From which, it follows that 
\begin{equation*}
    \begin{split}
    -\RN{1}(x)%&=\frac{1}{\si^2(x)}\int^{1}_x e^{2\int_{y}^{x}\frac{b}{\si^2}ds}dy\\
    \geq \frac{1}{\si^2(x)}\int^{x_*}_x e^{2\int_{y}^{x}\frac{b}{\si^2}ds}dy&\geq\frac{1}{(1+\de)|\si' (0)|^2x^2}\frac{x^{\ka}}{\ka-1}\left(\frac{1}{x^{\ka-1}}-\frac{1}{x_*^{\ka-1}}\right)\\
    &=\frac{1}{(1+\de)(\ka-1)|\si'(0)|^2}\left(\frac{1}{x}-\frac{x^{\ka-2}}{x_*^{\ka-1}}\right),\quad\forall x\in (0,x_*).
    \end{split}
\end{equation*}
Since $\RN{1}<0$ in $(0,x_{*})$, the non-integrability of $\RN{1}$ near $0$ follows. This completes the proof.
\end{proof}

Denote by $\LL_{0}$ the generator associated with $X_{t}^{0}$, that is, 
$$
\LL_0 \phi:=\frac{\si^2}{2} \phi''+b\phi',\quad \forall \phi\in C^2((0,\infty)).
$$
The generator $\LL_{0}$ extends to a self-adjoint operator in $L^{2}(u_{0}^{G}):=L^{2}((0,\infty);u_{0}^{G}dx)$. The rigorous formalism can be done using quadratic forms as it is done for $\LL_{\ep}$ in Subsection \ref{subsec-generator-spectrum-dynamics}. We end up this subsection with some discussion regarding the spectral properties of $\LL_{0}$. While not needed in the sequel, this will provide evidence that the spectrum of $\LL_{\ep}$ (in particular, $\la_{\ep,1}$ and $\la_{\ep,2}$) behaves in a non-trivial way as $\ep\to0$.  

\begin{rem}\label{rem-troubles-from-spectrum}
Note that the coefficient of the second order term of $\LL_{\ep}$ vanishes like $\frac{\ep^{2}a'(0)}{2}x$ as $x\to0$, while that of $\LL_{0}$ vanishes like $\frac{(\si'(0))^{2}}{2}x^{2}$ as $x\to0$. This singular limit of $\LL_{\ep}$ as $\ep\to0$ accounts for the non-trivial behaviour of the spectrum of $\LL_{\ep}$ as $\ep\to0$. Below are some consequences.

\begin{itemize}
\item[\rm(1)] Unlike $\LL_{\ep}$, $\si(\LL_0)$ -- the spectrum of $\LL_{0}$ -- is not purely discrete. To see this, we modify calculations in Subsection \ref{subsec-schrodinger-eqn} to convert $\LL_{0}$ to an unitarily equivalent Schr\"{o}dinger operator. Since $\frac{1}{\si}$ is non-integrable near $0$, we consider the change of variable
$$
y=\xi_{0}(x):=\int_1^x \frac{1}{\si}ds,\quad x\in (0,\infty).
$$
Clearly, $\xi_{0}$ is increasing and satisfies $\xi_{0}(0+)=-\infty$.
Set $y_{0,\infty}:=\xi_{0}(\infty)$. Then, $\xi_{0}:(0,\infty)\to(-\infty,y_{0,\infty})$ is invertible. Its inverse is denoted by $\xi_{0}^{-1}$. Then, $Y^{0}_t:=\xi_{0}(X^{0}_t)$ solves
\begin{equation}\label{SDE-simplest-noise-environmental}
    d Y^{0}_t=q_{0}(Y^{0}_t) dt+dW_t,
\end{equation}
where $q_{0}:=\left(\frac{b}{\si}-\frac{\si'}{2}\right)\circ\xi_{0}^{-1}$. Set $v_{0}^{G}:=(u_{0}^{G}\si)\circ\xi_{0}^{-1}$ and $L^{2}(v_{0}^{G}):=L^{2}((-\infty,y_{0,\infty}),v_{0}^{G}dy)$. Note that
$U_{0}: L^{2}(u_{0}^{G})\to L^{2}(v_{0}^{G})$, $f\mapsto f\circ\xi^{-1}_{0}$ and $\tilde{U}_{0}:L^{2}(v_{0}^{G})\to L^{2}((-\infty,y_{0,\infty}))$, $f\mapsto f\sqrt{v_{0}^{G}}$ are unitary transforms. The operator $\LL^{S}_{0}:=\tilde{U}_{0}U_{0}\LL_{0}U_{0}^{-1}\tilde{U}_{0}^{-1}$ turns out to be a Schr\"{o}dinger operator on $(-\infty,y_{0,\infty})$ and is given by
    \begin{equation*}
\LL_{0}^{S}:=\frac{1}{2}\frac{d^{2}}{dy^{2}}-\frac{1}{2}\left(q^2_{0}(y)+q'_{0}(y)\right)\quad\text{in}\quad L^{2}((-\infty,y_{0,\infty})).
\end{equation*}
It is easy to check that the potential $W_{0}:= \frac{1}{2}\left(q^2_{0}+q'_{0}\right)$ of $-\LL_{0}^{S}$ satisfies $W_{0}(-\infty)\in\R$ and $W_{0}(y_{0,\infty}-)=\infty$. Hence, the spectrum of $\LL_{0}^{S}$ is not purely discrete; neither is the spectrum of $\LL_{0}$.

\item[\rm(2)] When $\La_{0}<0$, the bottom of the spectrum of $\LL_0$ is positive, namely, $\inf\si(\LL_{0})>0$. Clearly, $\inf\si(\LL_{0})\geq0$ as $\LL_{0}$ is self-adjoint and non-negative. To see $0\notin\si(\LL_{0})$, we note that two linearly independent solutions of $\LL_{0}u=0$ are given by $u_{1}\equiv1$ and $u_{2}=\int_{1}^{\bullet}e^{-\int_{1}^{y}\frac{2b}{\si^{2}}ds}dy$. It is elementary to verify that $\lim_{x\to0^{+}}u_{2}(x)$ exists and is negative. Since $u_{0}^{G}\notin L^{1}((0,1))$ in this case by Lemma \ref{lem-integral-gibbs-zero-demographic}, we conclude $u_{1},u_{2}\notin L^{2}((0,1),u_{0}^{G}dx)$. Moreover, it is not hard to see that $u_1\in L^{2}((1,\infty),u_{0}^{G}dx)$ and $u_{2}\notin L^{2}((1,\infty),u_{0}^{G}dx)$. Hence, $C_{1}u_{1}+C_{2}u_{2}\notin L^{2}(u_{0}^{G})$ for any $(C_{1},C_{2})\neq(0,0)$, implying $0\notin\si(\LL_{0})$.

Theorem \ref{thm-bounds-eigenvalues} gives $\lim_{\ep\to0}\la_{\ep,1}=0$, saying that the limit of the principal eigenvalue $\la_{\ep,1}$ of $\LL_{\ep}$ is not an eigenvalue, but a generalized eigenvalue of $\LL_{0}$.

\item[\rm(3)] When $\La_{0}>0$, $0=\inf\si(\LL_{0})$ is a simple eigenvalue with constant eigenfunctions. However, obtaining information about the bottom of the rest of the spectrum, namely, $\inf\si(\LL_{0})\setminus\{0\}$, is difficult. Given the complicated structure of $\si(\LL_{0})$, it is even hard to determine whether $\inf\si(\LL_{0})\setminus\{0\}$ is an eigenvalue. This is what prevents us from establishing a more precise asymptotic of $\la_{\ep,2}$ beyond what we were able to show in Theorem \ref{thm-bounds-eigenvalues} (2).
\end{itemize}
\end{rem}

%%%%%%%%%%%%%%%%%%%%%%
%%%%%%%%%%%%%%%%%%%%%%

\section{\bf Tightness and concentration of QSDs} \label{s:tight}

In this section, we study the tightness and concentration of $\mu_{\ep}$ as $\ep\to 0$, and prove Theorem \ref{thm-concentration} in particular. We study concentration properties of $\mu_{\ep}$ near $\infty$ and $0$ in Subsections \ref{subsec-near-infinity} and \ref{subsec-near-0}, respectively, leading to the tightness of $\{\mu_{\ep}\}_{\ep}$. Theorem \ref{thm-concentration} is proven in Subsection \ref{subsec-proof}.

\subsection{Concentration near infinity}\label{subsec-near-infinity}
%Recall that $\{\LL_{\ep}\}_{\ep}$ admits a uniform Lyapunov function $V$ such that $\limsup_{x\to \infty}\sup_{0<\ep<\ep_*}\LL_{\ep}V=-\infty$ for some fixed $\ep_*>0$. 

We prove the following result addressing the concentration of $\mu_{\ep}$ near $\infty$. The proof mainly uses techniques on the basis of Lyapunov-type functions.

\begin{prop}\label{prop-concentration-infty}
Assume {\bf (H)}. Then, $\lim_{x\to\infty}\sup_{\ep}\mu_{\ep}((x,\infty))=0$.
\end{prop}
\begin{proof}
Set $V:=-\int_1^{\bullet}\frac{b}{\si^2}ds$ in $(0,\infty)$. Then,  
\begin{equation*}
    \LL_{\ep} V=\frac{\ep^2}{2}\left(\frac{ab(\si^2)'}{\si^4}-\frac{a b'}{\si^2}\right)+\frac{1}{2}\left(\frac{b(\si^2)'}{\si^2}-b'\right)-\frac{b^2}{\si^2}.
\end{equation*}
Thanks to {\bf (H)}(4), we find some $N_0\in\N$ such that $\LL_{\ep}V\leq -\frac{b^2}{2\si^2}$ in $(N_0,\infty)$. As $V(\infty)=\infty$ by {\bf (H)}, there is $n_{0}\gg1$ such that $\{n_{0}\leq V\}\subset(N_{0},\infty)$, and hence,
\begin{equation}\label{LLV-2022-03-10}
     \LL_{\ep}V\leq -\frac{b^2}{2\si^2}\quad\text{in}\quad\{n_0\leq V\}.
\end{equation}

Let $\{\zeta_n\}_{n>n_0}$ be a sequence of smooth and non-decreasing functions on $(0,\infty)$ satisfying
\begin{equation*}
\zeta_n(x)=\begin{cases}
0,& x\in (0,n_0),\\
x,& x\in (n_0+1,n),\\
n+1,& x\in (n+2, \infty),
\end{cases}
\quad\andd\quad \zeta''_n\leq0\text{ on }[n,n+2].
\end{equation*}
In addition, we let $\{\zeta_n\}_{n}$ coincide on $[0,n_0+1]$. 

Due to $V(\infty)=\infty$ and {\bf (H)}(1)-(2), the function $\zeta_n(V)-(n+1)$ is twice continuously differentiable and compactly supported. As $\LL^*_{\ep}u_{\ep}=-\la_{\ep,1}u_{\ep}$ (in the weak sense), we derive 
\begin{equation*}
\begin{split}
0&=\int_0^{\infty}\LL_{\ep}\left[ \zeta_n(V)-(n+1)\right] u_{\ep}dx +\la_{\ep,1}\int_0^{\infty}\left[\zeta_n(V)-(n+1)\right] u_{\ep}dx\\
& 
= \int_0^{\infty}\left[\zeta'_n(V)\LL_{\ep}V+\frac{1}{2}\al_{\ep}\zeta''_n(V) |V'|^2\right] u_{\ep} dx+\la_{\ep,1}\int_0^{\infty}\left[\zeta_n(V)-(n+1)\right] u_{\ep}dx
\\
&\leq\int_0^{\infty}\left[\zeta'_n(V)\LL_{\ep}V+\frac{1}{2}\al_{\ep}\zeta''_n(V) |V'|^2\right] u_{\ep} dx\\
&=\int_{\{n_{0}\leq V\leq n+2\}}\left[\zeta'_n(V)\LL_{\ep}V+\frac{1}{2}\al_{\ep}\zeta''_n(V) |V'|^2\right] u_{\ep} dx
\end{split}
\end{equation*}
where we used $\la_{\ep,1}>0$ and $\zeta_n-( n+1)\leq 0$ in the inequality, and $\zeta_{n}=0$ on $(0,n_{0})$ and $\zeta_{n}'=\zeta_{n}''=0$ on $(n+2,\infty)$ in the last equality. 

We deduce from $\zeta_{n}'=1$ on $[n_0+1, n]$, $\zeta_{n}'\geq0$ and \eqref{LLV-2022-03-10} that
\begin{equation*}
    \begin{split}
        &\int_{\{n_{0}\leq V\leq n+2\}}\zeta'_n(V)\LL_{\ep}Vu_{\ep} dx\\
        &\qquad  =\int_{\{n_0\leq V\leq n_0+1\}\cup\{n\leq V\leq n+2\}}\zeta'_n(V)\LL_{\ep}V u_{\ep}d x+\int_{\{n_0+1\leq V\leq n\}}\LL_{\ep}V u_{\ep}d x\\
        &\qquad\leq\int_{\{n_0+1\leq V\leq n\}}\LL_{\ep}V u_{\ep}d x\\
        &\qquad\leq-\frac{1}{2}\left(\inf_{(n_0+1,\infty)}\frac{b^2}{\si^2}\right) \mu_{\ep}(\{n_0+1\leq V\leq n\}).
    \end{split}
\end{equation*}
As $\zeta_{n}''=0$ on $[n_0+1, n]$ and $\zeta''_n\leq 0$ on $[n,n+2]$, we find
\begin{equation*}
    \begin{split}
        \int_{\{n_{0}\leq V\leq n+2\}}\frac{1}{2}\al_{\ep}\zeta''_n(V) |V'|^2 u_{\ep} dx\leq \frac{C_{\ep,n}}{2}\int_{\{n_0\leq V\leq n_0+1\}}u_{\ep} dx\leq \frac{C_{\ep,n}}{2},
    \end{split}
\end{equation*}
where $C_{\ep,n}=\max_{\{n_0\leq V\leq n_0+1\}}\al_{\ep}|\zeta''_n(V)| |V'|^2$.
Hence, we find
\begin{equation*}
\begin{split}
\mu_{\ep}(\{n_0+1\leq V\})=\lim_{n\to\infty}\mu_{\ep}(\{n_0+1\leq V\leq n\})\leq C_{\ep,n}\left(\inf_{(n_0+1,\infty)}\frac{b^2}{\si^2}\right)^{-1}.
\end{split}
\end{equation*}
% Since $u_{\ep}$ locally uniformly converges to $u$, an application of Fatou's lemma gives
% $$
% \mu(\{V>n_0+1\})\leq \liminf_{\ep\to 0}\mu_{\ep}(\{V>n_0+1\})\leq  \frac{1}{\inf_{(n_0,\infty)}\frac{b^2}{\si^2}}\times \max_{\{n_0\leq V\leq n_0+1\}}\al_{\ep} |\zeta''_n(V)| |V'|^2.
% $$ 
Recalling that $\zeta''_n$ is independent of $n$ on $[n_0,n_0+1]$ and $\sup_{\ep}\al_{\ep}$ is locally bounded on $(0,\infty)$, we find $\sup_{\ep,n}C_{\ep,n}<\infty$. Since {\bf (H)} ensures $\lim_{x\to \infty}\frac{b^2(x)}{\si^2(x)}=\infty$, there must hold $\lim_{n_0\to \infty}\sup_{\ep}\mu_{\ep}(\{n_0+1\leq V\})=0$. The conclusion follows. 
\end{proof}

% \begin{cor}\label{cor-gibbs-concentration-infinity}
% Assume {\bf (H)} and $\frac{b'(0)}{|\si'(0)|^2}>\frac{1}{2}$. For any $x_*>0$, there holds $\sup_{\ep}\int_{x_*}^{\infty} u^G_{\ep}dx<\infty$. 
% \end{cor}
% {\rd(Is the proof necessary? Does it follow directly from the expression of $u_{\ep}^{G}$?)}
% \begin{proof}
% Let $n_0\in \N$ be as in the proof of Proposition \ref{prop-concentration-infty}. It is straightforward checked that $\LL^*_{\ep} u^G_{\ep}=0$. Hence, we follow the arguments in the proof of Proposition \ref{prop-concentration-infty} to see the existence of $C>0$ and $\ep_0>0$ such that
% \begin{equation*}
%     \int_{n_0+1}^{\infty} u^G_{\ep}dx\leq C \int_{n_0}^{n_0+1} u^G_{\ep}dx,\quad \forall 0<\ep<\ep_0.
% \end{equation*}
% Noting that $u^G_{\ep}=\frac{1}{\al_{\ep}}e^{-2V_{\ep}}$ locally uniformly converges to $u^G_0$ on $[n_0,n_0+1]$ as $\ep\to 0$, we find for any  $x_*>0$,
% \begin{equation*}
%     \begin{split}
%     \sup_{0<\ep<\ep_0}\int_{x_*}^{\infty} u^G_{\ep}dx&\leq \sup_{0<\ep<\ep_0} \int_{\min\{x_*,n_0+1\}}^{n_0+1} u^G_{\ep}dx+ \sup_{0<\ep<\ep_0} \int_{n_0+1}^{\infty} u^G_{\ep}dx\\
%     &\leq  \int_{\min\{x_*,n_0+1\}}^{n_0+1} \sup_{0<\ep<\ep_0} u^G_{\ep}dx+ C \int_{n_0}^{n_0+1} \sup_{0<\ep<\ep_0}u^G_{\ep}dx<\infty.
%     \end{split}
% \end{equation*}
% \end{proof}

%%%%%%%%%%%%%%%%%%%%%%
\subsection{Concentration near the extinction state}\label{subsec-near-0}

We prove the following result quantifying $
\mu_{\ep}$ or $u_{\ep}$ near $0$ in the case $\Lambda_{0}>0$. 

\begin{prop}\label{prop-concentration-0}
Assume {\bf (H)} and $\La_0>0$. Then, there are $k\in (0,1)$, $x_*>0$ and $C>0$ such that 
$$
\sup_{\ep}u_{\ep}(x)\leq \frac{C}{x^{k}},\quad \forall x\in (0,x_*).
$$
In particular, $\lim_{x\to0}\sup_{\ep}\mu_{\ep}((0,x))=0$.
\end{prop}

We establish some results before proving Proposition \ref{prop-concentration-0}. The following result addressing the limit of $\la_{\ep,1}$ as $\ep\to0$ is the general part of Theorem \ref{thm-bounds-eigenvalues}.

\begin{thm}\label{thm-0-super-limit-principal-eigenvalue}
Assume {\bf (H)}. Then, $\lim_{\ep\to 0}\la_{\ep,1}=0$.
\end{thm}
\begin{proof}
We extend $\mu_{\ep}$ to be a Borel probability measure on $[0,\infty)$ by setting $\mu_{\ep}(\{0\})=0$. Proposition \ref{prop-concentration-infty} ensures that $\{\mu_{\ep}\}_{\ep}$ is tight as Borel probability measures on $[0,\infty)$. We assume, up to a sequence, that $\mu_{\ep}$ weakly converges to some Borel probability measure $\mu_*$ on $[0,\infty)$ as $\ep\to0$. Since $\limsup_{\ep\to 0}\la_{\ep,1}<\infty$ by Lemma \ref{lem-upper-bound-eigenvalues}, we assume without loss of generality that $\lim_{\ep\to 0}\la_{\ep,1}=\la_*\geq 0$.

Let $f:[0,\infty)\to\R$ be bounded and uniformly continuous. We claim that 
\begin{equation}\label{eqn-markov-group-extrinsic-only}
    \E_{\mu_*}[f(X^0_{t})]=e^{-\la_*t}\int_0^{\infty}fd\mu_*,\quad \forall t\geq 0.
\end{equation}
Setting $f\equiv 1$ yields $1=e^{-\la_*t}$ for all $t\geq 0$, resulting $\la_*=0$. The theorem then follows.

It remains to prove \eqref{eqn-markov-group-extrinsic-only}. Fix any $t>0$. Note that for any $\de>0$,
\begin{equation*}
    \begin{split}
    &\left|\E_{x}[f(X^{\ep}_t)]-\E_{x}[f(X^{0}_t)]\right|\\
    &\qquad\leq \int_{|X^{\ep}_t-X^0_t|>\de}|f(X^{\ep}_t)-f(X^0_t)|d\P_x +\int_{|X^{\ep}_t-X^0_t|\leq \de}|f(X^{\ep}_t)-f(X^0_t)|d\P_x\\
    &\qquad\leq 2\|f\|_{\infty}\P_x \left\{\max_{0\leq s\leq t}|X^{\ep}_s-X^0_s|>\de\right\}+\int_{|X^{\ep}_t-X^0_t|\leq \de}|f(X^{\ep}_t)-f(X^0_t)|d\P_x.
    \end{split}
\end{equation*}
As \eqref{main-diffusion-eqn} is a small random perturbation of \eqref{eqn-diffusion-extrinsic}, we apply \cite[Theorem 2.1.2]{FW98} with standard modifications to find  
\begin{equation*}
    \lim_{\ep\to 0}\P_x \left\{\max_{0\leq s\leq t}|X^{\ep}_s-X^0_s|>\de\right\}=0 \quad \text{locally uniformly in } x\in [0,\infty).
\end{equation*} 
The uniform continuity of $f$ implies
$$
\lim_{\de\to 0}\limsup_{\ep\to 0}\int_{|X^{\ep}_t-X^0_t|\leq \de}|f(X^{\ep}_t)-f(X^0_t)|d\P_x=0 \quad \text{locally uniformly in } x\in [0, \infty).
$$ 
Hence, we arrive at $\lim_{\ep\to 0}\E^{\ep}_x[f(X^{\ep}_t)]=\E_x[f(X^0_t)]$ locally uniformly in $x\in [0, \infty)$. It follows that
\begin{equation*}
    \begin{split}
   &\limsup_{\ep\to 0} \int_0^{\infty}\left|\E^{\ep}_{\bullet}[f(X^{\ep}_t)]-\E_{\bullet}[f(X^0_t)]\right|d\mu_{\ep}\\
   & 
   \qquad\leq \limsup_{\ep\to 0}\int_0^{A}\left|\E^{\ep}_{\bullet}[f(X^{\ep}_t)]-\E_{\bullet}[f(X^0_t)]\right|d\mu_{\ep} +2\|f\|_{\infty}\times \limsup_{\ep\to 0}\mu_{\ep}((A,\infty))
   \\
  &\qquad \leq 2\|f\|_{\infty}\sup_{\ep}\mu_{\ep}((A,\infty)),\quad\forall A>0.
   \end{split}
\end{equation*}
Thanks to Proposition \ref{prop-concentration-infty}, we pass to the limit $A\to\infty$ to find
$$
\limsup_{\ep\to 0} \int_0^{\infty}\left|\E^{\ep}_{\bullet}[f(X^{\ep}_t)]-\E_{\bullet}[f(X^0_t)]\right|d\mu_{\ep}=0.
$$

The regularity of $b$ and $\si$ ensures that $\E_{\bullet}[f(X^0_t)]\in C_b([0,\infty))$. Hence, the weak limit $\lim_{\ep\to 0}\mu_{\ep}=\mu_*$ implies that $\lim_{\ep\to 0} \int_0^{\infty} \E_{\bullet}[f(X^0_t)]d\mu_{\ep}=\int_0^{\infty}\E_{\bullet}[f(X^0_t)]d\mu_*$. As a result, 
\begin{equation*}
    \begin{split}
    &\left|\E^{\ep}_{\mu_{\ep}}[f(X^{\ep}_t)]-\E_{\mu_*}[f(X^{0}_t)]\right|\\
    &\qquad\leq \int_0^{\infty}\left|\E^{\ep}_{\bullet}[f(X^{\ep}_t)]-\E_{\bullet}[f(X^0_t)]\right|d\mu_{\ep}\\
    &\qquad\quad+\left|\int_0^{\infty} \E_{\bullet}[f(X^0_t)]d\mu_{\ep}-\int_0^{\infty}\E_{\bullet}[f(X^0_t)]d\mu_*\right|\to 0\quad \text{as}\quad \ep\to 0.
    \end{split}
\end{equation*}
Considering the facts $\E^{\ep}_{\mu_{\ep}}[f(X^{\ep}_t)]=e^{-\la_{\ep,1}t}\int_0^{\infty}fd\mu_{\ep}$, $\lim_{\ep\to 0}\int_0^{\infty}fd\mu_{\ep}=\int_0^{\infty}fd\mu_*$ and $\lim_{j\to \infty} \la_{\ep_j}=\la_*$, we deduce
$$
\E_{\mu_*}[f(X^0_{t})]=\lim_{\ep\to 0}\E^{\ep}_{\mu_{\ep}}[f(X^{\ep}_{t})]=\lim_{\ep\to 0}e^{-\la_{\ep,1}t}\int_0^{\infty}fd\mu_{\ep}= e^{-\la_{*}t}\int_0^{\infty}fd\mu_{*},
$$
leading to \eqref{eqn-markov-group-extrinsic-only}. This completes the proof.
\end{proof}

The next technical lemma is needed.

\begin{lem}\label{lem-upper-soln-x^-k}
Assume {\bf (H)} and $\La_0>0$. Then, there is $k_*\in (0,1)$ such that for any $k\in (k_*,1)$, there exist $x_*>0$, $\ep_*=\ep_*(k)>0$ and $C=C(k)>0$ such that 
$$
\LL^*_{\ep}x^{-k}\leq -Cx^{-k}\quad\text{in}\quad (0,x_*),\quad\forall \ep\in (0,\ep_*).
$$
\end{lem}
\begin{proof}
Let $k\in (0,1)$. Straightforward calculations yield
\begin{equation}\label{eqn-2021-04-09-1}
\begin{split}
\LL^*_{\ep}x^{-k}%&=\frac{1}{2}\left[(\ep^2 a+\si^2)x^{-k}\right]''-(b x^{-k})'\\
%&=\frac{1}{2}(\ep^2 a+\si^2)'' x^{-k}-k(\ep^2 a'+(\si^2)') x^{-k-1}+\frac{k(k+1)}{2}(\ep^2 a+\si^2)x^{-k-2}\\
%&\quad -b' x^{-k}+kb x^{-k-1}\\
&=\left[\frac{\ep^2}{2} a''+\frac{(\si^2)''}{2}-k\frac{(\si^2)'}{x} +\frac{k(k+1)\si^2}{2x^2}-b'+k\frac{b}{x}\right] x^{-k}\\
&\quad+\left[\frac{k(k+1)\ep^2 a}{2x}-k\ep^2 a'\right] x^{-k-1}.
\end{split}
\end{equation}
We see from {\bf(H)}(1)-(3) that for $0<x\ll 1$
\begin{equation*}
\begin{split}
&\frac{(\si^2)''}{2}-k\frac{(\si^2)'}{x} +\frac{k(k+1)\si^2}{2x^2}%&=\si \si''+|\si'|^2-\frac{2k\si\si'}{x}+\frac{k(k+1)\si^2}{2x^2}\\
%&=\left[1-2k+\frac{k(k+1)}{2}\right]|\si'(0)|^2+o(1)\\
=\frac{k^2-3k+2}{2}|\si'(0)|^2+o(1),\\
&-b'+k\frac{b}{x}=-(1-k)b'(0)+o(1),\quad \frac{k(k+1) a}{2x}-k a'=\frac{k(k-1)}{2} \left(a'(0)+o(1)\right)<0.
\end{split}
\end{equation*}
It follows from \eqref{eqn-2021-04-09-1} that 
\begin{equation*}
\begin{split}
\LL^*_{\ep}x^{-k}\leq\left[\frac{\ep^2}{2}a''+\frac{k^2-3k+2}{2}|\si'(0)|^2-(1-k)b'(0)+o(1)\right] x^{-k},\quad\forall 0<x\ll 1.
\end{split}
\end{equation*}
Since $\lim_{\ep\to 0}\frac{\ep^2}{2}\sup_{(0,1)}a=0$, the conclusion of the lemma follows if we show the existence of some $k_*\in (0,1)$ such that 
\begin{equation}\label{eqn-2022-03-16-1}
    \frac{k^2-3k+2}{2}|\si'(0)|^2-(1-k)b'(0)<0,\quad\forall k\in (k_*,1).
\end{equation}

Since $\La_0>0$,  there exists $\de_*>0$ such that $b'(0)>\left(\frac{1}{2}+\de_*\right)|\si'(0)|$, and thus,
\begin{equation*}
\begin{split}
\frac{k^2-3k+2}{2}|\si'(0)|^2-(1-k)b'(0)%&\leq \left[\frac{k^2-3k+2}{2}-(1-k)\left(\frac{1}{2}+\de_{*}\right)\right]|\si'(0)|^2\\
\leq\frac{1}{2}(k-1)(k-1+2\de_*)|\si'(0)|^2.
\end{split}
\end{equation*}
Setting $k_*:=1-2\de_*$ leads to \eqref{eqn-2022-03-16-1}.
\end{proof}

Now, we prove Proposition \ref{prop-concentration-0}.

\begin{proof}[Proof of Proposition \ref{prop-concentration-0}]
By Theorem \ref{thm-0-super-limit-principal-eigenvalue} and Lemma \ref{lem-upper-soln-x^-k}, there are $k\in(0,1)$ and $x_*>0$ such that 
\begin{equation}\label{inequality-2022-04-13}
(\LL^*_{\ep}+\la_{\ep,1})x^{-k}<0\quad\text{in}\quad (0,x_*).    
\end{equation}
 
Set $v_{\ep}:=x^{k}u_{\ep}$. The fact $u_{\ep}>0$ and Lemma \ref{bound-u_ep-near-0} imply that 
$$
0\leq \liminf_{x\to 0}v_{\ep}(x)\leq  \limsup_{x\to 0}v_{\ep}(x)\leq \lim_{x\to 0} x^k \times  \limsup_{x\to 0}u_{\ep}(x)=0.
$$
That is, 
\begin{equation}\label{value-of-v-ep-at-zero}
    \lim_{x\to 0}v_{\ep}(x)=0.
\end{equation}  
Noting that $(\LL^*_{\ep}+\la_{\ep,1}) u_{\ep}=0$ (by Lemma \ref{qsd-existence-uniqueness}), $u_{\ep}=x^{-k} v_{\ep}$ and 
$$
(\LL^*_{\ep}+ \la_{\ep,1})u_{\ep}=\frac{1}{2}\al_{\ep} u''_{\ep}+(\al'_{\ep}-b)u_{\ep} +\left(\frac{1}{2}\al''_{\ep}-b'+ \la_{\ep,1}\right) u_{\ep},
$$
we calculate
$$
0=(\LL^*_{\ep}+\la_{\ep,1}) u_{\ep}=x^{-k}\left( \frac{1}{2}\al_{\ep} v''_{\ep} +(\al'_{\ep}-b)v'_{\ep}\right)+v_{\ep}(\LL^*_{\ep}+\la_{\ep,1}) x^{-k}+\al_{\ep} (-k x^{-k-1}) v'_{\ep}.
$$
Multiplying the above equation by $x^k$ and rearranging the terms, we arrive at 
\begin{equation}\label{eqn-v-ep}
\frac{1}{2}\al_{\ep} v''_{\ep}+\left(\al'_{\ep}-b-\frac{k}{x}\al_{\ep}\right) v'_{\ep}+\frac{(\LL^*_{\ep}+\la_{\ep,1})x^{-k}}{x^{-k}}v_{\ep}=0.
\end{equation}

Note that $(\LL_{\ep}^{*}+\la_{\ep,1})u_{\ep}=0$ is the same as $\frac{1}{2}(\al_{\ep}u'_{\ep})'+\left[\left(\frac{\al'_{\ep}}{2}-b\right)u_{\ep} \right]'+\la_{\ep,1}u_{\ep}=0$. Considering the first limit in \eqref{limit-alpha-V} and Theorem \ref{thm-0-super-limit-principal-eigenvalue}, we apply Harnack's inequality to find  $C_1>0$ (independent of $\ep$) such that
\begin{equation*}
\sup_{(\frac{x_*}{4},\frac{x_*}{2})} u_{\ep}\leq C_1 \inf_{(\frac{x_*}{4},\frac{x_*}{2})} u_{\ep}\leq \frac{4C_1}{x_*}\int_{\frac{x_*}{4}}^{\frac{x_*}{2}}u_{\ep}dx \leq \frac{4C_1}{x_*}.
\end{equation*}
Hence, $\sup_{(\frac{x_*}{4},\frac{x_*}{2})}v_{\ep}=\sup_{(\frac{x_*}{4},\frac{x_*}{2})}x^k u_{\ep}\leq \frac{4C_1}{x_*} \left(\frac{x_*}{2}\right)^k$.

Due to \eqref{inequality-2022-04-13}, the coefficient of $v_{\ep}$ in \eqref{eqn-v-ep} is negative on $(0,x_{*})$. Given \eqref{value-of-v-ep-at-zero}, we apply the maximum principle to $v_{\ep}$ on $(0,\frac{x_*}{2})$ to conclude that
$\max_{(0,\frac{x_*}{2})}v_{\ep}=v_{\ep}\left(\frac{x_*}{2}\right)\leq \frac{4C_1}{x_*} \left(\frac{x_*}{2}\right)^k$. The conclusion follows from the relation $u_{\ep}=\frac{v_{\ep}}{x^{k}}$. 
\end{proof}

%%%%%%%%%%%%%%%%%%%%%%

\subsection{Proof of Theorem \ref{thm-concentration}}\label{subsec-proof}

\medskip

(1) If $\La_0<0$, we extend $\mu_{\ep}$ to be a Borel probability measure on $[0,\infty)$ by setting $\mu_{\ep}(\{0\})=0$. Arguments as in the proof of Theorem 
\ref{thm-0-super-limit-principal-eigenvalue} show that up to a sequence $\mu_{\ep}$ weakly converges to some Borel probability measure $\mu_*$ on $[0,\infty)$ as $\ep\to0$. Moreover, $\E_{\mu_*}[\phi(X^0_t)]=\int_0^{\infty} \phi d\mu_*$ for all $t\geq 0$ and $\phi\in C_b([0,\infty))$. 

Since Proposition \ref{prop-existence-uniqueness-sd} says $\lim_{t\to \infty}X^0_t=0$ $\P_x$-a.e. for any $x>0$, we deduce from the dominated convergence theorem that 
$\int_0^{\infty} \phi d\mu_*=\lim_{t\to \infty}\E_{\mu_*}[\phi(X^0_t)]=\phi(0)$ for all $\phi\in C_b([0,\infty))$, leading to $\mu_*=\de_0$. As a result, $\lim_{\ep\to0}\mu_{\ep}=\de_{0}$ weakly, and in particular, $\lim_{\ep\to 0} \int_0^{\infty}\phi d\mu_{\ep}=0$ for all $\phi\in C_b([0,\infty))$ with $\phi(0)=0$.

% , or equivalently, the only finite Borel measure that satisfies $\LL^*_0 u=0$ in the weak sense. Hence, $\mu_*=0$, leading to  
% \begin{equation}\label{eqn-2022-04-10-1}
%     \lim_{\ep\to 0}\int_0^{\infty}\phi d\mu_{\ep} =0,\quad \forall \phi\in C_c((0,\infty)).
% \end{equation}

% By Proposition \ref{prop-concentration-infty}, for each $\de>0$, there is $k=k(\de)>0$ such that $\sup_{\ep}\mu_{\ep}((k,\infty))\leq\de$. Let $\eta_{k}\in C_b((0,\infty))$ take values in $[0,1]$ and satisfy $\eta_k=1$ on $[0,k]$ and $\eta_k=0$ on $[k+1,\infty)$. Then, for any $\phi\in C_b((0,\infty))$ with $\inf\supp(\phi)>0$, 
% $$
% \left|\int_0^{\infty}\phi d\mu_{\ep}\right|\leq  \left|\int_0^{\infty}\phi \eta_{k}d\mu_{\ep}\right|+\left|\int_0^{\infty}(1-\eta_k)\phi d\mu_{\ep}\right|\leq \left|\int_0^{\infty}\phi \eta_{k}d\mu_{\ep}\right|+ \de\|\phi\|_{\infty}.
% $$
% It follows from \eqref{eqn-2022-04-10-1} and the arbitrariness of $\de>0$ that $\lim_{\ep\to 0}\int_0^{\infty}\phi d\mu_{\ep}=0$. 

\medskip

(2) If $\La_0>0$, Propositions \ref{prop-concentration-infty} and \ref{prop-concentration-0} ensure the tightness of $\{\mu_{\ep}\}_{\ep}$. We assume up to a sequence that $\mu_{\ep}$ weakly converges to some Borel probability measure $\mu_*$ on $(0,\infty)$ as $\ep\to0$. By Lemma \ref{qsd-existence-uniqueness}, the density $u_{\ep}$ of $\mu_{\ep}$ satisfies
$\frac{1}{2}(\al_{\ep}u_{\ep})''-(bu_{\ep})'+\la_{\ep,1}u_{\ep}=0$. This together with the first limit in \eqref{limit-alpha-V} and Theorem \ref{thm-0-super-limit-principal-eigenvalue} implies that $\mu_*$ must satisfy $\LL^*_0 u=0$ in the weak sense, that is, $\int_{0}^{\infty}\LL_{0}\phi d\mu_{*}=0$ for all $\phi\in C_{0}^{2}((0,\infty))$.

We claim $\mu_*$ admits a non-negative density $u_*\in C^{2}((0,\infty))$ and $\lim_{j\to \infty}u_{\ep_j}=u_*$ in $C^{2}((0,\infty))$. Then,  $\LL^*_0 u_*=0$, and hence,  $u_*=u_0$ and $\mu_*=\mu_0$ by Lemma \ref{lem-unique-L^1-soln}. That is, $\mu_0$ is the unique limiting measure of $\{\mu_{\ep}\}$ and $\lim_{\ep\to 0}u_{\ep}=u_0$ locally in $C^2((0,\infty))$, giving the desired result. 

It remains to prove the claim. Let $\II_1$ and $\II_2$ be open intervals in $(0,\infty)$ and satisfy $\II_1\stst \II_2\stst (0,\infty)$. Given \eqref{limit-alpha-V} and Theorem \ref{thm-0-super-limit-principal-eigenvalue}, we apply Harnack's inequality to $u_{\ep}$ on $\II_2$ to find $C_1=C_1(\II_1,\II_2)>0$ (independent of $\ep$) such that 
\begin{equation*}\label{eqn-2022-04-1-3}
    \sup_{\II_1}u_{\ep}\leq C_1\inf_{\II_1}u_{\ep}\leq \frac{C_1}{|\II_1|}\int_{\II_1}u_{\ep}dx\leq \frac{C_1}{|\II_1|}. 
\end{equation*}
Setting $\phi_{\ep}:=\frac{u_{\ep}}{u_{\ep}^G}$,  we find from \eqref{gibbs-to-gibbs} that $\sup_{\II_1}\phi_{\ep}\leq \frac{2\sup_{\II_1}u_{\ep}}{\inf_{\II_1}u^G_0} \leq\frac{2C_1}{|\II_1|\inf_{\II_1}u^G_0}$. That is, $\{\phi_{\ep}\}_{\ep}$ is locally uniformly bounded. In comparison with the expression for $u_{\ep}$ given in Lemma \ref{qsd-existence-uniqueness}, we readily see that $\phi_{\ep}$ is a positive eigenfunction of $-\LL_{\ep}$ associated with $\la_{\ep,1}$, and hence, satisfies
$\frac{1}{2}\al_{\ep}\phi''_{\ep}+b\phi'_{\ep}=-\la_{\ep,1}\phi_{\ep}$. Given the first limit in \eqref{limit-alpha-V} and Theorem \ref{thm-0-super-limit-principal-eigenvalue}, we apply the classical interior Schauder estimates to $\{\phi_{\ep}\}_{\ep}$ to arrive at $\sup_{\ep}\sup_{\II}\left(|\phi'_{\ep}|+|\phi''_{\ep}|+|\phi'''_{\ep}|\right)<\infty$ for any $\II\stst(0,\infty)$. An application of the  Arzel\`a-Ascoli theorem then yields the precompactness of $\{\phi'_{\ep}\}_{\ep}$ and  $\{\phi''_{\ep}\}_{\ep}$ in $C(\ol{\II})$. Since $\II\stst (0,\infty)$ is arbitrary, we may assume without loss of generality according to the diagonal argument that $\phi_{\ep_j}$ locally converges to some non-negative $\phi_*$ in $ C^2((0,\infty))$ as $j\to\infty$. Thanks to \eqref{gibbs-to-gibbs} and the weak limit $\lim_{j\to \infty}\mu_{\ep_j}=\mu_*$, we find $d\mu_*=u_*dx$ with $u_*:=\phi_*u^G_0$ and $u_{\ep}$ converges to $u_*$ in $C^2((0,\infty))$ as $\ep\to0$. This proves the claim, and thus, completes the proof.

%%%%%%%%%%%%%%%%%%%%%%
%%%%%%%%%%%%%%%%%%%%%%

\section{\bf  Asymptotic bounds of the first two eigenvalues} \label{s:asym}

This section is devoted to the proof of Theorem \ref{thm-bounds-eigenvalues}. The asymptotic bounds of the first and second eigenvalues are respectively treated in Subsections \ref{subsec-asymptotic-first-eigenvalue} and \ref{subsec-asymptotic-second-eigenvalue}.

We start with a technical result that is frequently used in the sequel. It says that appropriately normalized eigenfunctions of $-\LL_{\ep}$ have uniform-in-$\ep$ small tails (against a weight) near $\infty$, and is only used for eigenfunctions associated with the first two eigenvalues. 

\begin{lem}\label{lem-compactness-2-infty}
Assume {\bf (H)} and fix $i\in\N$. For each $0<\ep\ll 1$, let $\tilde{\phi}_{\ep,i}$ be an eigenfunction of $-\LL_{\ep}$ associated with the eigenvalue $\la_{\ep,i}$. If $\sup_{\ep}\int_{x_0}^{\infty}|\tilde{\phi}_{\ep,i}|^2u^G_{\ep}dx\leq 1$ for some $x_0>0$, then 
$$
\lim_{z\to\infty}\sup_{\ep}\int_{z}^{\infty}|\tilde{\phi}_{\ep,i}|^2 u^G_{\ep} dx=0.
$$
\end{lem}
\begin{proof}
Set $\psi_{\ep,i}:=\tilde{U}_{\ep}U_{\ep}\tilde{\phi}_{\ep,i}$, where $\tilde{U}_{\ep}$ and $U_{\ep}$ are unitary transforms defined in Subsection \ref{subsec-schrodinger-eqn}. Then, 
\begin{equation}\label{estimate-2022-03-20}
\int_{\xi_{\ep}(x_1)}^{\xi_{\ep}(x_2)}|\psi_{\ep,i}|^2 dy=\int_{x_1}^{x_2}|\tilde{\phi}_{\ep,i}|^2u_{\ep}^{G}dx\leq 1,\quad\forall x_0\leq x_1<x_2\leq\infty, 
\end{equation}
where the inequality is a result of the assumption. Moreover,  $-\LL^S_{\ep}\psi_{\ep,i}=\la_{\ep,i}\psi_{\ep,i}$, that is, 
\begin{equation}\label{eqn-2021-09-28-1}
-\frac{1}{2}\psi''_{\ep,i}+W_{\ep}\psi_{\ep,i}=\la_{\ep,i}\psi_{\ep,i}\quad \text{in}\quad (0,y_{\ep,\infty}). 
\end{equation}

Fix some integer $n_{0}>x_0+1$. Let  $\{\eta_n\}_{n>n_{0}}$ be a  sequence of functions in $ C^{\infty}_0((0,\infty))$, take values in $[0,1]$ and satisfy 
\begin{equation*}
\eta_n(x)=\begin{cases}
0,& x\in (0,n_{0}-1)\cup (2n,\infty),\\
1,& x\in (n_{0}, n),
\end{cases}
\quad\andd \quad |\eta'_n(x)|\leq
\begin{cases}
2,&  x\in [n_{0}-1,n_{0}],\\
\frac{2}{n},& x\in [n,2n].
\end{cases}
\end{equation*} 
In addition, we require $\{\eta_n\}_{n>n_{0}}$ to coincide on $(0,n_{0}]$. Clearly, as $n\to\infty$, $\eta_n$ converges (uniformly in $(0,M)$ for any $M>0$) to some function $\eta\in C^{\infty}((0,\infty))$ taking values in $[0,1]$ and satisfying $\eta=\eta_{n_{0}+1}$ on $(0,n_{0}]$ and $\eta=1$ on $(n_{0},\infty)$. 

Set $\tilde{\eta}_{n,\ep}:=\eta_n\circ\xi^{-1}_{\ep}$. Obviously, $\tilde{\eta}_{n,\ep}\in C_0^{2}((0,y_{\ep,\infty}))$ with $\supp(\tilde{\eta}_{n,\ep})\subset (\xi_{\ep}(n_{0}-1), y_{\ep,\infty})$. Multiplying \eqref{eqn-2021-09-28-1} by $\tilde{\eta}^2_{n,\ep}\psi_{\ep,i}$ and integrating over $(0,y_{\ep,\infty})$, we find from integration by parts and \eqref{estimate-2022-03-20} that
\begin{equation*}\label{eqn-2021-05-02-3}
\begin{split}
&\frac{1}{2}\int_0^{y_{\ep,\infty}}\tilde{\eta}^2_{n,\ep}|\psi'_{\ep,i}|^2 dy+\int_0^{y_{\ep,\infty}}\tilde{\eta}_{n,\ep}\tilde{\eta}'_{n,\ep}\psi_{\ep,i}\psi'_{\ep,i} dy+\int_0^{y_{\ep,\infty}}\tilde{\eta}^2_{n,\ep}W_{\ep}|\psi_{\ep,i}|^2 dy\\
&\qquad=\la_{\ep,i}\int_0^{y_{\ep,\infty}}\tilde{\eta}^2_{n,\ep}|\psi_{\ep,i}|^2 dy \leq\la_{\ep,i}\int_{\xi_{\ep}(n_{0}-1)}^{y_{\ep,\infty}}|\psi_{\ep,i}|^2dy\leq \la_{\ep,i}.
\end{split}
\end{equation*}
An application of H\"{o}lder's inequality yields 
\begin{equation*}
\begin{split}
\left|\int_0^{y_{\ep,\infty}}\tilde{\eta}_{n,\ep}\tilde{\eta}'_{n,\ep}\psi'_{\ep,i}\psi_{\ep,i} dy\right|&\leq \left(\frac{1}{2}\int_0^{y_{\ep,\infty}}|\tilde{\eta}_{n,\ep}\psi'_{\ep,i}|^2 dy\right)^{\frac{1}{2}}\left(2\int_0^{y_{\ep,\infty}}|\tilde{\eta}'_{n,\ep}\psi_{\ep,i}|^2 dy\right)^{\frac{1}{2}}\\
&\leq \frac{1}{4}\int_0^{y_{\ep,\infty}}\tilde{\eta}^2_{n,\ep}|\psi'_{\ep,i}|^2 dy+\int_0^{y_{\ep,\infty}}|\tilde{\eta}'_{n,\ep}|^2|\psi_{\ep,i}|^2 dy.
\end{split}
\end{equation*}
% \begin{equation*}
%     \frac{1}{4}\int_0^{y_{\ep,\infty}}\tilde{\eta}^2_{n,\ep}|\psi'_{\ep,i}|^2 dy+\int_0^{y_{\ep,\infty}}\tilde{\eta}^2_{n,\ep}W_{\ep}|\psi_{\ep,i}|^2 dy\leq \la_{\ep,i}+\int_0^{y_{\ep,\infty}}|\tilde{\eta}'_{n,\ep}|^2|\psi_{\ep,i}|^2 dy.
% \end{equation*}
Absorbing $\frac{1}{4}\int_0^{y_{\ep,\infty}}\tilde{\eta}^2_{n,\ep}|\psi'_{\ep,i}|^2 dy$ and dropping the remaining $\frac{1}{4}\int_0^{y_{\ep,\infty}}\tilde{\eta}^2_{n,\ep}|\psi'_{\ep,i}|^2 dy$ yield
\begin{equation}\label{eqn-2022-02-09-8}
    \int_0^{y_{\ep,\infty}}\tilde{\eta}^2_{n,\ep}W_{\ep}|\psi_{\ep,i}|^2 dy\leq \la_{\ep,i}+\int_0^{y_{\ep,\infty}}|\tilde{\eta}'_{n,\ep}|^2|\psi_{\ep,i}|^2 dy.
\end{equation}

Since $\tilde{\eta}'_{n,\ep}=(\eta'_{n}\sqrt{\al_{\ep}})\circ\xi^{-1}_{\ep}$ and $\{\eta_n\}_{n>n_{0}}$ coincide on $[n_{0}-1, n_{0}]$, we see from the first limit in \eqref{limit-alpha-V} that there is $M_1>0$ such that
$$
\sup_{[\xi_{\ep}(n_{0}-1),\xi_{\ep}(n_{0})]}|\tilde{\eta}'_{n,\ep}|=\sup_{[n_{0}-1,n_{0}]}|\eta'_{n}|\sqrt{\al_{\ep}}<\sqrt{M_1},\quad \forall n>n_{0}.
$$
Thanks to {\bf (H)}(4), we can make $n_{0}$ larger if necessary to ensure the existence of some $M_2>0$ such that $\al_{\ep}\leq M_2\si^2$ and $\frac{|\si|}{x}\leq \frac{|b|}{8\sqrt{2 M_2}|\si|}$ in $(n_{0},\infty)$. As $|\eta'_n|\leq \frac{2}{n}$ on $[n,2n]$, we derive that for $n>n_{0}$ and $y=\xi_{\ep}(x)\in [\xi_{\ep}(n),\xi_{\ep}(2n)]$,
$$
|\tilde{\eta}'_{n,\ep}(y)|=|\eta'_n(x)|\sqrt{\al_{\ep}(x)}\leq \frac{2}{n}\sqrt{\al_{\ep}(x)}\leq 4\frac{\sqrt{\al_{\ep}(x)}}{x}\leq 4\sqrt{M_2}  \frac{|\si(x)|}{x}\leq \frac{|b(x)|}{2\sqrt{2}|\si(x)|}.
$$
Therefore,
\begin{equation*}
\begin{split}
    \int_0^{y_{\ep,\infty}}|\tilde{\eta}'_{n,\ep}|^2|\psi_{\ep,i}|^2 dy&=\int_{\xi_{\ep}(n_{0}-1)}^{\xi_{\ep}(n_{0})}|\tilde{\eta}'_{n,\ep}|^2|\psi_{\ep,i}|^2 dy+\int_{\xi_{\ep}(n)}^{\xi_{\ep}(2n)}|\tilde{\eta}'_{n,\ep}|^2|\psi_{\ep,i}|^2 dy\\
    &\leq M_{1}\int_{\xi_{\ep}(n_{0}-1)}^{\xi_{\ep}(n_{0})}|\psi_{\ep,i}|^2 dy+\frac{1}{8}\int_{\xi_{\ep}(n_{0})}^{\infty}\frac{b^2\circ\xi^{-1}_{\ep}}{\si^2\circ\xi^{-1}_{\ep}}|\psi_{\ep,i}|^2 dy\\
    &\leq M_{1}+\frac{1}{8}\int_{\xi_{\ep}(n_{0})}^{\infty}\frac{b^2\circ\xi^{-1}_{\ep}}{\si^2\circ\xi^{-1}_{\ep}}|\psi_{\ep,i}|^2 dy,
\end{split}
\end{equation*}
where we used \eqref{estimate-2022-03-20} in the last inequality. It follows from \eqref{eqn-2022-02-09-8} that 
\begin{equation*}
    \begin{split}
         \int_0^{y_{\ep,\infty}}\tilde{\eta}^2_{n,\ep}W_{\ep}|\psi_{\ep,i}|^2 dy & \leq \la_{\ep,i}+M_{1}+\frac{1}{8}\int_{\xi_{\ep}(n_{0})}^{\infty}\frac{b^2\circ\xi^{-1}_{\ep}}{\si^2\circ\xi^{-1}_{\ep}}|\psi_{\ep,i}|^2 dy\\
         &\leq 2M_1+\frac{1}{8}\int_{\xi_{\ep}(n_{0})}^{\infty}\frac{b^2\circ\xi^{-1}_{\ep}}{\si^2\circ\xi^{-1}_{\ep}}|\psi_{\ep,i}|^2 dy,
    \end{split}
\end{equation*}
where we assumed without loss of generality that $\limsup_{\ep\to 0}\la_{\ep,i}<M_1$ in the last inequality (ensured by Lemma \ref{lem-upper-bound-eigenvalues}). Since $\eta_n\uparrow\eta$ as $n\to\infty$, letting $n\to \infty$ in the above inequality leads to
\begin{equation*}
    \begin{split}
       \int_0^{y_{\ep,\infty}}\tilde{\eta}^2_{\ep}W_{\ep}|\psi_{\ep,i}|^2 dy\leq 2M_1+\frac{1}{8}\int_{\xi_{\ep}(n_{0})}^{\infty}\frac{b^2\circ\xi^{-1}_{\ep}}{\si^2\circ\xi^{-1}_{\ep}}|\psi_{\ep,i}|^2 dy,
    \end{split}
\end{equation*}
where $\tilde{\eta}_{\ep}:=\eta\circ \xi^{-1}_{\ep}$ satisfies $\tilde{\eta}_{\ep}=1$ on $[\xi_{\ep}(n_{0}),y_{\ep,\infty})$. By Lemma \ref{properties-schrodinger-potential} (3), we can make $n_{0}$ larger if necessary so that $W_{\ep}\geq\frac{b^2\circ\xi^{-1}_{\ep}}{4\si^2\circ\xi^{-1}_{\ep}}$ in $(\xi_{\ep}(n_{0}),y_{\ep,\infty})$. As a result, 
\begin{equation*}
    \begin{split}
        \frac{1}{4}\int_{\xi_{\ep}(n_{0})}^{\infty}\frac{b^2\circ\xi^{-1}_{\ep}}{\si^2\circ\xi^{-1}_{\ep}}|\psi_{\ep,i}|^2 dy\leq 2M_1+\frac{1}{8}\int_{\xi_{\ep}(n_{0})}^{\infty}\frac{b^2\circ\xi^{-1}_{\ep}}{\si^2\circ\xi^{-1}_{\ep}}|\psi_{\ep,i}|^2 dy.
    \end{split}
\end{equation*}
Hence,  we see from \eqref{estimate-2022-03-20} that
\begin{equation*}
    \begin{split}
        \frac{1}{8}\int_{n_{0}}^{\infty}\frac{b^2}{\si^2}|\tilde{\phi}_{\ep,i}|^2 u^G_{\ep}dx=\frac{1}{8}\int_{\xi_{\ep}(n_{0})}^{y_{\ep,\infty}}\frac{b^2\circ\xi^{-1}_{\ep}}{\si^2\circ\xi^{-1}_{\ep}}|\psi_{\ep,i}|^2 dy\leq 2M_1,
    \end{split}
\end{equation*}
giving $\int_{z}^{\infty}|\tilde{\phi}_{\ep,i}|^2 u^G_{\ep}dx\leq \frac{16 M_1}{\inf_{(z,\infty)}\frac{b^2}{\si^2}}$ for any $z>n_{0}$. The conclusion follows immediately from the fact $\lim_{z\to \infty}\frac{b^2(z)}{\si^2(z)}=\infty$ ensured by {\bf (H)}(4). This completes the proof. 
\end{proof}

%%%%%%%%%%%%%%%%%%

\subsection{Asymptotic bounds of the first eigenvalue}\label{subsec-asymptotic-first-eigenvalue}

Note that the limit $\lim_{\ep\to0}\la_{\ep,1}=0$ has been established in Theorem \ref{thm-0-super-limit-principal-eigenvalue}. In the rest of this subsection, we prove finer asymptotic bounds of the first eigenvalue $\la_{\ep,1}$ of $-\LL_{\ep}$ stated in Theorem \ref{thm-bounds-eigenvalues}.

The asymptotic bounds of $\la_{\ep,1}$ under the condition $\La_0>0$ stated in Theorem \ref{thm-bounds-eigenvalues} (2) is restated in the following result.

\begin{thm}\label{Lower-bound-principal-eigenvalue}
Assume {\bf (H)} and $\La_0>0$. For each $0<\ga\ll 1$, there holds
$$
\ep^{(1+\ga)\frac{4b'(0)}{|\si'(0)|^2}-2}\lesssim_{\ep}\la_{\ep,1}\lesssim_{\ep}  \ep^{(1-\ga)\frac{2b'(0)}{|\si'(0)|^2}-1}.
$$
\end{thm}
\begin{proof}
The upper bound and lower bound are treated separately.

\medskip

\noindent{\bf Upper bound.} As the first eigenvalue of the self-adjoint operator $-\LL_{\ep}$,  $\la_{\ep,1}$ admits the variational formula:
\begin{equation}\label{upper-bound-2022-03-13}
\la_{\ep,1}=\inf_{\phi\in D(\EE_{\ep})} \frac{\displaystyle\int_0^{\infty}\al_{\ep}|\phi'|^2 u_{\ep}^{G}dx}{\displaystyle\int_0^{\infty}|\phi|^2 u_{\ep}^{G}dx}\leq \frac{\displaystyle\int_0^{\infty}|\phi'_{\ep}|^2 e^{-2V_{\ep}}dx}{\displaystyle\int_0^{\infty}|\phi_{\ep}|^2 \frac{1}{\al_{\ep}}e^{-2V_{\ep}}dx},
\end{equation}
where $\phi_{\ep}\in C_0^{\infty}((0,\infty))$ is non-decreasing and satisfies
\begin{equation*}
\phi_{\ep}(x)=\begin{cases}
0,& 0<x<\ep,\\
1,& x>2\ep,
\end{cases}
\quad\andd\quad 0<\phi'_{\ep}\leq \frac{2}{\ep}\text{ on } \left(\ep,2\ep\right).
\end{equation*}

By {\bf (H)}(1), $b>0$ in $(0,x_*)$ for some $x_*>0$. Split
$-V_{\ep}=\int_1^{x_*}\frac{b}{\al_{\ep}}ds +\int_{x_*}^{\bullet} \frac{b}{\al_{\ep}}ds=:A_{\ep}+B_{\ep}$. Clearly, $B_{\ep}$ is increasing in  $(0,x_*)$. Hence,
\begin{equation}\label{eqn-2021-05-05-1}
\int_0^{\infty}|\phi'_{\ep}|^2 e^{-2V_{\ep}}dx=\int_{\ep}^{2\ep}|\phi'_{\ep}|^2 e^{-2V_{\ep}}dx\leq \frac{4}{\ep^2} e^{2A_{\ep}} \int_{\ep}^{2\ep}e^{2B_{\ep}(x)}dx\leq \frac{4}{\ep} e^{2A_{\ep}+2B_{\ep}(2\ep)},
\end{equation}
and 
\begin{equation}\label{eqn-2021-05-05-2}
\begin{split}
\int_0^{\infty}|\phi_{\ep}|^2 \frac{1}{\al_{\ep}}e^{-2V_{\ep}}dx\geq  e^{2A_{\ep}} \int_{2\ep}^{\infty} \frac{1}{\al_{\ep}} e^{2B_{\ep}}dx= e^{2A_{\ep}+2B_{\ep}(2\ep)}\int_{2\ep}^{\infty} \frac{1}{\al_{\ep}(x)} e^{2B_{\ep}(x)-2B_{\ep}(2\ep)} dx.
\end{split}
\end{equation}

As $\La_0>0$, there is $\ga_0>0$ such that $\ka=\ka(\ga):=\frac{2b'(0)(1-\ga)}{|\si'(0)|^2}>1$ for all $\ga\in(0,\ga_0)$. By {\bf (H)}(1)-(3), we can make $x_*=x_{*}(\ga)$ smaller if necessary so that
$$
\frac{b(x)}{\al_{\ep}(x)}\geq \frac{(1-\ga) b'(0)}{\ep^2 a'(0)+|\si'(0)|^2 x}\,\,\andd \,\, \al_{\ep}(x)\leq (1+\ga)( \ep^2 a'(0)x+|\si'(0)|^2 x^2),\quad\forall x\in (0,x_*),
$$
As a consequence, we find 
\begin{equation*}
\begin{split}
2B_{\ep}(x)-2B_{\ep}(2\ep)&\geq 2\int_{2\ep}^{x}\frac{(1-\ga)b'(0)}{\ep^2 a'(0)+|\si'(0)|^2 s} ds\\
&=\ka \ln \frac{\ep^2 a'(0)+|\si'(0)|^2x}{\ep^2 a'(0)+2\ep |\si'(0)|^2},\quad \forall x\in (2\ep, x_*),
\end{split}
\end{equation*}
leading to 
\begin{equation*}
\begin{split}
&\int_{2\ep}^{\infty}\frac{1}{\al_{\ep}(x)}e^{2B_{\ep}(x)-2B_{\ep}(2\ep)}dx \\
&\qquad\geq \frac{1}{1+\ga}\int_{\frac{x_*}{2}}^{x_*}  \frac{\left(\ep^2 a'(0)+|\si'(0)|^2 x\right)^{\ka-1 }}{\left(\ep^2 a'(0)+2\ep |\si'(0)|^2\right)^{\ka}}\frac{1}{x} dx\\
 &\qquad
 \geq \frac{1}{1+\ga} \left[ \frac{1}{\ep^2 a'(0)+2\ep |\si'(0)|^2}\right]^{\ka} \min_{[\frac{x_*}{2},x_*]}\left(\ep^2 a'(0)+|\si'(0)|^2 x\right)^{\ka-1}\int_{\frac{x_*}{2}}^{x_*}\frac{1}{x} dx
 \\
 &\qquad
\geq \frac{1}{1+\ga} \left[ \frac{1}{\ep^2 a'(0)+2\ep |\si'(0)|^2}\right]^{\ka} \left[\frac{x_*|\si'(0)|^2}{2} \right]^{\ka-1} \ln 2
\\
 &\qquad
\geq \frac{\ln 2}{1+\ga}  \left[\frac{x_*|\si'(0)|^2}{2} \right]^{\ka-1} \left[3\ep|\si'(0)|^2\right]^{-\ka}=:\frac{C'_{\ga}}{\ep^{\ka}}.
\end{split}
\end{equation*}
%where $C_{\ep,\ga}$ satisfies $\inf_{\ep}C_{\ep,\ga}>0$. 
Hence, we see from \eqref{upper-bound-2022-03-13}, \eqref{eqn-2021-05-05-1} and  \eqref{eqn-2021-05-05-2} that $\la_{\ep,1}\lesssim_{\ep}C_{\ga}\ep^{\ka-1}$ for some $C_{\ga}>0$. Since this is true for all $\ga\in(0,\ga_{0})$, the constant $C_{\ga}$ can be replaced by $1$, establishing the upper bound.

% Recall  that 
% \begin{equation*}
% \la_{\ep,1}=\sup_{\phi\in C_0^{\infty}((0,\infty))}\frac{\int_0^{\infty}|\phi'|^2 e^{-2V_{\ep}}dx}{\int_0^{\infty}|\phi|^2 \frac{1}{\al_{\ep}}e^{-2V_{\ep}}dx},\quad\forall \ep>0.
% \end{equation*}

\medskip

\noindent{\bf Lower bound.} Recall from Lemma \ref{prop-spectral-structure} that $\phi_{\ep,1}$ is the positive eigenfunction  of $-\LL_{\ep}$ associated with $\la_{\ep,1}$ and satisfies $\|\phi_{\ep,1}\|_{L^2(u^G_{\ep})}=1$. In particular,
\begin{equation}\label{eqn-2022-02-09-2}
    \la_{\ep,1}=\EE_{\ep}(\phi_{\ep,1},\phi_{\ep,1})=\int_0^{\infty}|\phi'_{\ep,1}|^2 e^{-2V_{\ep}}dx.
\end{equation}
Let $0<\de\ll 1$. By Lemma \ref{lem-compactness-2-infty}, there is $x^*=x^*(\de)\gg1$ such that $\int_{x^*}^{\infty}|\phi_{\ep,1}|^2u^G_{\ep}dx\leq\de$. Then, 
\begin{equation}\label{eqn-2022-02-09-1}
\begin{split}
1=\|\phi_{\ep,1}\|^2_{L^2(u^G_{\ep})}\leq\de+\int_0^{x^*}|\phi_{\ep,1}|^2 \frac{1}{\al_{\ep}} e^{-2V_{\ep}}dx
\end{split}
\end{equation}

By Lemma \ref{bound-u_ep-near-0}, there are $C_{\ep}>0$ and $x_{\ep}>0$ such that $u_{\ep}\leq C_{\ep}$ in $(0,x_{\ep})$. This together with $u_{\ep}=\frac{\phi_{\ep,1} }{\|\phi_{\ep,1}\|_{L^1(u^G_{\ep})}}\frac{1}{\al_{\ep}}e^{-2V_{\ep}}$ (by Lemma \ref{qsd-existence-uniqueness}) yields $\phi_{\ep,1}\leq C_{\ep} \|\phi_{\ep,1}\|_{L^1(u^G_{\ep})} \al_{\ep}e^{2V_{\ep}}$ in $(0,x_{\ep})$.
Since $\al_{\ep}=\ep^2 a+\si^2$ and $V_{\ep}(0+)=\int_0^1 \frac{b}{\al_{\ep}}ds<\infty$, {\bf (H)}(2)-(3) ensures the existence of $C'_{\ep}>0$ such that $\phi_{\ep,1}(x)\leq C_{\ep} C'_{\ep} \|\phi_{\ep,1}\|_{L^1(u^G_{\ep})} x$ for $x\in (0,x_{\ep})$. In particular, $\phi_{\ep,1}(0+)=0$, and hence, $\phi_{\ep,1}=\int_0^{\bullet}\phi'_{\ep,1}dx$. This together with \eqref{eqn-2022-02-09-1} and H\"older's inequality yields 
\begin{equation*}
    \begin{split}
        1-\de&\leq \int_0^{x^*}\left| \int_0^x \phi'_{\ep,1}(y) e^{-V_{\ep}(y)}e^{V_{\ep}(y)}dy\right|^2\frac{1}{\al_{\ep}(x)} e^{-2V_{\ep}(x)}dx\\
    &\leq \int_0^{x^*}\left(\int_0^x |\phi'_{\ep,1}(y)|^2e^{-2V_{\ep}(y)} dy\right) \left(\int_0^x e^{2V_{\ep}(y)}dy\right) \frac{1}{\al_{\ep}(x)} e^{-2V_{\ep}(x)}dx\\
    &\leq \left(\int_0^{\infty} |\phi'_{\ep,1}(y)|^2e^{-2V_{\ep}(y)} dy\right)\II_{\ep},
    \end{split}
\end{equation*}
where $\II_{\ep}=\int_0^{x^*}\int_0^x  \frac{1}{\al_{\ep}(x)} e^{2[V_{\ep}(y)-V_{\ep}(x)]}dydx$. It then follows from  \eqref{eqn-2022-02-09-2} that  
\begin{equation}\label{eqn-2022-02-09-7}
\la_{\ep,1}\geq\frac{1-\de}{\II_{\ep}}.
\end{equation}

% Therefore, it suffices to deduce an appropriate upper bound for the integral $ \int_0^{\infty}\int_0^x  \frac{1}{\al_{\ep}(x)} e^{2[V_{\ep}(y)-V_{\ep}(x)]}dydx$ for $0<\ep\ll1$.

% Since $V_{\ep}=-2\int_1^{\bullet}\frac{b}{\ep^2a+\si^2}ds$, we apply  {\bf (H)} to see the existence of $x^*\gg 1$ such that for each $0<\ep\ll 1$, $V_{\ep}(x)>V_{\ep}(y)$ for all $x\in (x^*,\infty)$ and $y\in (0,x)$. We split the integral
% \begin{equation}\label{eqn-2021-07-17-4}
% \begin{split}
% \int_0^{\infty}\int_0^x  \frac{1}{\al_{\ep}(x)} e^{2[V_{\ep}(y)-V_{\ep}(x)]}dydx&=\int_0^{x^*}\int_0^x  \frac{1}{\al_{\ep}(x)} e^{2[V_{\ep}(y)-V_{\ep}(x)]}dydx+\int_{x^*}^{\infty}\int_0^x  \frac{1}{\al_{\ep}(x)} e^{2(V_{\ep}(y)-V_{\ep}(x))}dydx\\
% &=:\RN{1}+\RN{2},\quad\forall \ep>0.
% \end{split}
% \end{equation}

To finish the proof, it suffices to derive an appropriate upper bound for $\II_{\ep}$. From {\bf (H)}(1)-(3) we find $x_*=x_*(\de)>0$ such that 
\begin{equation*}
    b(x)\leq (1+\de)b'(0)x\,\,\andd\,\,  \ep^2a(x)+\si^2(x)\geq (1-\de)x[\ep^2 a'(0)+|\si'(0)|^2x],\quad\forall x\in (0,x_*).
\end{equation*}
Clearly, $x_{*}\ll x^{*}$. We write 
\begin{equation}\label{eqn-2022-02-09-5}
    \begin{split}
    \II_{\ep}=\left(\int_0^{x_*}+\int_{x_*}^{x^*}\right)\int_0^x  \frac{1}{\al_{\ep}(x)} e^{2[V_{\ep}(y)-V_{\ep}(x)]}dydx=:\RN{1}_{\ep}+\RN{2}_{\ep}.
    \end{split}
\end{equation}

We first treat $\RN{1}_{\ep}$. Note that for $0<y<x\leq x_{*}$,
\begin{equation}\label{eqn-2021-09-07-1}
    \begin{split}
    V_{\ep}(y)-V_{\ep}(x)\leq \frac{b'(0)(1+\de)}{(1-\de)}\int_y^x\frac{1}{\ep^2a'(0)+|\si'(0)|^2 s}ds=\frac{\ka}{2}\ln \frac{\ep^2 a'(0)+|\si'(0)|^2x}{\ep^2 a'(0)+|\si'(0)|^2 y},
    \end{split}
\end{equation}
where $\ka:=\frac{2(1+\de)b'(0)}{(1-\de)|\si'(0)|^2}>1$ due to the assumption $\La_0>0$. It follows that 
\begin{equation}\label{eqn-2022-02-16-1}
    \begin{split}
    \RN{1}_{\ep}&\leq \int_0^{x_*}\int_0^x\frac{1}{(1-\de)x[\ep^2 a'(0)+|\si'(0)|^2x]}\left[\frac{\ep^2 a'(0)+|\si'(0)|^2x}{\ep^2 a'(0)+|\si'(0)|^2 y}\right]^{\ka}dy dx\\
    &=\frac{1}{1-\de}\int_0^{x_*}\frac{[\ep^2 a'(0)+|\si'(0)|^2x]^{\ka-1}}{x}\int_0^x [\ep^2 a'(0)+|\si'(0)|^2y]^{-\ka} dydx. 
    \end{split}
\end{equation}

Clearly, $\int_0^{x}[\ep^2a'(0)+|\si'(0)|^2y]^{-\ka}dy\leq [\ep^2a'(0)]^{-\ka} x$. Calculating the integral and dropping the negative term (due to $\ka>1$), we find
\begin{equation}\label{eqn-2022-02-09-4}
    \begin{split}
    \int_0^{x}[\ep^2a'(0)+|\si'(0)|^2y]^{-\ka}dy\leq\frac{|\si'(0)|^{-2}}{\ka-1} [\ep^2a'(0)]^{1-\ka}.
    \end{split}
\end{equation}
Thus, for any $\ga\in(0,1)$, 
\begin{equation*}
    \begin{split}
    \int_0^{x}[\ep^2a'(0)+|\si'(0)|^2y]^{-\ka}dy&\leq \left([\ep^2a'(0)]^{-\ka} x\right)^{\ga} \left(\frac{|\si'(0)|^{-2}}{\ka-1} [\ep^2a'(0)]^{1-\ka}\right)^{1-\ga}\\
    &=C_1 \ep^{2-2\ka-2\ga}x^{\ga},
    \end{split}
\end{equation*}
where $C_1:=[a'(0)]^{1-\ka-\ga}|\si'(0)|^{2(\ga-1)}(\ka-1)^{\ga-1}$. It then follows from \eqref{eqn-2022-02-16-1} that
\begin{equation}\label{eqn-2022-02-09-6}
    \begin{split}
    \RN{1}_{\ep}&\leq \frac{C_1}{1-\de}\ep^{2-2\ka-2\ga}\int_0^{x_*}x^{\ga-1}[\ep^2 a'(0)+|\si'(0)|^2x]^{\ka-1} dx\\
    &\leq \frac{C_1}{1-\de}\ep^{2-2\ka-2\ga}[\ep^2 a'(0)+|\si'(0)|^2x_*]^{\ka-1}\frac{x_*^{\ga}}{\ga}\leq C_2 \ep^{2-2\ka-2\ga},
    \end{split}
\end{equation}
where $C_2:=\frac{2C_1}{(1-\de)\ga}|\si'(0)|^{2(\ka-1)}x_*^{\ka-1+\ga}$.

Now, we treat $\RN{2}_{\ep}$. By \eqref{eqn-2021-09-07-1}, for $x\in [x_*,x^*)$,
\begin{equation*}
    \begin{split}
    V_{\ep}(y)-V_{\ep}(x)&=\int_y^{x_*} \frac{b}{\ep^2 a+\si^2}ds+\int_{x_*}^x \frac{b}{\ep^2 a+\si^2}ds\\
    &\leq \frac{\ka}{2}\ln \frac{\ep^2 a'(0)+|\si'(0)|^2x_*}{\ep^2 a'(0)+|\si'(0)|^2 y}+ \int_{x_*}^{x^*}\frac{|b|}{\si^2}ds,\quad\forall y\in (0,x_*),
    \end{split}
\end{equation*}
and $V_{\ep}(y)-V_{\ep}(x)\leq \int_{x_*}^{x^*}\frac{|b|}{\si^2}ds$ for $y\in [x_*,x)$. Hence, there are $C_{3},C_{4}>0$ such that
\begin{equation}\label{eqn-2022-03-16-2}
    \begin{split}
        \RN{2}_{\ep}&=\int_{x_*}^{x^*}\left(\int_0^{x_*}+\int_{x_*}^x\right)\frac{1}{\al_{\ep}(x)} e^{2[V_{\ep}(y)-V_{\ep}(x)]}dydx\\
        &\leq \int_{x_*}^{x^*}\int_0^{x_*}\frac{1}{\al_{\ep}(x)} \left[\frac{\ep^2 a'(0)+|\si'(0)|^2x_*}{\ep^2 a'(0)+|\si'(0)|^2 y}\right]^{\ka} e^{2\int_{x_*}^{x^*}\frac{|b|}{\si^2}ds}dydx\\
        &\quad+\int_{x_*}^{x^*}\int_{x_*}^{x}\frac{1}{\al_{\ep}(x)}e^{2\int_{x_*}^{x^*}\frac{|b|}{\si^2}ds}dydx
        \\
        & 
        =[\ep^2 a'(0)+|\si'(0)|^2x_*]^{\ka} e^{2\int_{x_*}^{x^*}\frac{|b|}{\si^2}ds} \int_{x_*}^{x^*}\frac{1}{\al_{\ep}(x)}\int_0^{x_*}[\ep^2 a'(0)+|\si'(0)|^2 y]^{-\ka}dy dx
        \\
        & 
        \quad+e^{2\int_{x_*}^{x^*}\frac{|b|}{\si^2}ds}\int_{x_*}^{x^*}\frac{x-x_*}{\al_{\ep}(x)}dx
         \\
        &\leq 2[|\si'(0)|^2x_*]^{\ka} e^{2\int_{x_*}^{x^*}\frac{|b|}{\si^2}ds} \int_{x_*}^{x^*}\frac{1}{\si^2(x)}\int_0^{x_*}[\ep^2 a'(0)+|\si'(0)|^2 y]^{-\ka}dy dx
        \\
        & 
        \quad+x^*e^{2\int_{x_*}^{x^*}\frac{|b|}{\si^2}ds}\int_{x_*}^{x^*}\frac{1}{\si^2(x)}dx
        \\
        &\leq C_{3}\int_0^{x_*}[\ep^2 a'(0)+|\si'(0)|^2 y]^{-\ka}dy+C_{4}.\\
    \end{split}
\end{equation}
Applying \eqref{eqn-2022-02-09-4} to the integral $\int_0^{x_*}[\ep^2 a'(0)+|\si'(0)|^2 y]^{-\ka}dy$, we find for some $C_{5}>0$,
\begin{equation*}
    \begin{split}
        \RN{2}_{\ep}\leq C_{3}\frac{|\si'(0)|^{-2}}{\ka-1} [\ep^2a'(0)]^{1-\ka} +C_{4}\leq C_{5}\ep^{2-2\ka}%2[|\si'(0)|^2 x_*]^{\ka} e^{2\int_{x_*}^{x^*}\frac{|b|}{\si^2}ds}\frac{|\si'(0)|^{-2}}{\ka-1}[\ep^2 a'(0)]^{1-\ka} \int_{x_*}^{x^*}\frac{1}{\si^2}dx+x^*e^{2\int_{x_*}^{x^*}\frac{|b|}{\si^2}ds}\int_{x_*}^{x^*}\frac{1}{\si^2}dx\\
        %&\leq \left[\frac{2x_*^{\ka}|\si'(0)|^{2\ka-2}[a'(0)]^{1-\ka}}{\ka-1}+x^*\right]e^{2\int_{x_*}^{x^*}\frac{|b|}{\si^2}ds}\int_{x_*}^{x^*}\frac{1}{\si^2}dx\times \ep^{2-2\ka}=:C_3\ep^{2-2\ka}.
    \end{split}
\end{equation*}
This together with  \eqref{eqn-2022-02-09-5} and \eqref{eqn-2022-02-09-6} leads to $\II_{\ep}\leq(C_2+C_5)\ep^{2-2\ka-2\ga}$. Thanks to \eqref{eqn-2022-02-09-7}, the conclusion follows readily from  $\ka=\frac{2b'(0)(1+\de)}{|\si'(0)|^2(1-\de)}$ and the arbitrariness of $0<\de\ll 1$ and $\ga\in (0,1)$.
\end{proof}

Theorem \ref{thm-bounds-eigenvalues} (1) regarding the asymptotic lower bound of $\la_{\ep,1}$ under the condition $\La_0<0$ is restated as the next result.

\begin{thm}
Assume {\bf (H)} and $\La_0<0$. There exists $C>0$ such that $\la_{\ep,1}\gtrsim_{\ep} \frac{C}{|\ln \ep|}$.
\end{thm}
\begin{proof}
We proceed as in the proof of the lower bound in Theorem \ref{Lower-bound-principal-eigenvalue}. As $\La_0<0$, we let $0<\de\ll 1$ be such that $\ka:=\frac{2(1+\de)b'(0)}{(1-\de)|\si'(0)|^2}<1$. Following arguments leading to \eqref{eqn-2022-02-09-7}, we find 
\begin{equation}\label{eqn-2022-02-09-7-1}
\la_{\ep,1}\geq\frac{1-\de}{\II_{\ep}}
\end{equation}
for some $x^*=x^{*}(\de)\gg1$, where $\II_{\ep}=\int_0^{x^*}\int_0^x  \frac{1}{\al_{\ep}(x)} e^{2[V_{\ep}(y)-V_{\ep}(x)]}dydx$. Due to {\bf (H)}(1)-(3), there exists  $x_*=x_*(\de)>0$ such that
\begin{equation*}
    b(x)\leq (1+\de)b'(0)x\,\,\andd\,\,  \ep^2a(x)+\si^2(x)\geq (1-\de)x[\ep^2 a'(0)+|\si'(0)|^2x],\quad\forall x\in (0,x_*).
\end{equation*}
We split 
\begin{equation}\label{eqn-2022-02-09-5-1}
    \begin{split}
    \II_{\ep}=\left(\int_0^{x_*}+\int_{x_*}^{x^*}\right)\int_0^x  \frac{1}{\al_{\ep}(x)} e^{2[V_{\ep}(y)-V_{\ep}(x)]}dydx=:\RN{1}_{\ep}+\RN{2}_{\ep}.
    \end{split}
\end{equation}

We first treat $\RN{1}_{\ep}$. Since for any $0<y<x\leq x_{*}$
\begin{equation*}\label{eqn-2021-09-07-1-1}
    \begin{split}
    V_{\ep}(y)-V_{\ep}(x)\leq \frac{b'(0)(1+\de)}{(1-\de)}\int_y^x\frac{1}{\ep^2a'(0)+|\si'(0)|^2 s}ds=\frac{\ka}{2}\ln \frac{\ep^2 a'(0)+|\si'(0)|^2x}{\ep^2 a'(0)+|\si'(0)|^2 y},
    \end{split}
\end{equation*}
we deduce
\begin{equation}\label{eqn-2022-03-16-3}
    \begin{split}
        \RN{1}_{\ep}&\leq \int_0^{x_*}\int_0^x\frac{1}{(1-\de)x[\ep^2 a'(0)+|\si'(0)|^2x]}\left[\frac{\ep^2 a'(0)+|\si'(0)|^2x}{\ep^2 a'(0)+|\si'(0)|^2 y}\right]^{\ka}dy dx\\
    & 
    =\frac{1}{1-\de}\int_0^{x_*}\frac{[\ep^2 a'(0)+|\si'(0)|^2x]^{\ka-1}}{x} \int_0^x [\ep^2 a'(0)+|\si'(0)|^2y]^{-\ka} dydx
    \\
    & 
    =\frac{1}{1-\de}\int_0^{x_*}\frac{[\ep^2 a'(0)+|\si'(0)|^2x]^{\ka-1}}{x}\frac{\left([\ep^2 a'(0)+|\si'(0)|^2 x]^{1-\ka}-[\ep^2 a'(0)]^{1-\ka}\right)}{(1-\ka)|\si'(0)|^2}dx 
    \\
    &=\frac{1}{(1-\de)(1-\ka)|\si'(0)|^2}\int_0^{x_*}\frac{1}{x} \left[1-\left(\frac{\ep^2 a'(0)}{\ep^2 a'(0)+|\si'(0)|^2x}\right)^{1-\ka}\right]dx\\
    &=\frac{1}{(1-\de)(1-\ka)|\si'(0)|^2}\int_0^{\frac{|\si'(0)|^2x_*}{\ep^2 a'(0)}}\frac{1}{x}\left(1-\frac{1}{(1+x)^{1-\ka}}\right)dx\\
    &=\frac{1}{(1-\de)(1-\ka)|\si'(0)|^2}\int^1_{\frac{1}{1+\frac{|\si'(0)|^2x_*}{\ep^2 a'(0)}}}\frac{1-t^{1-\ka}}{t(1-t)}dt,
  \end{split}
\end{equation}
where the first equality follows from straightforward calculation, the second one is a result of an obvious change of variable, and the third one is due to the change of variable $t=\frac{1}{1+x}$. Since $\frac{1-t^{1-\ka}}{1-t}<1$ for $t\in(0,1)$, we deduce
$$
\int^1_{\frac{1}{1+\frac{|\si'(0)|^2x_*}{\ep^2 a'(0)}}}\frac{1-t^{1-\ka}}{t(1-t)}dt\leq \int^1_{\frac{1}{1+\frac{|\si'(0)|^2x_*}{\ep^2 a'(0)}}}\frac{1}{t}dt=\ln \left(1+\frac{|\si'(0)|^2x_*}{\ep^2 a'(0)}\right) \leq 3|\ln \ep|,
$$
which together with \eqref{eqn-2022-03-16-3} leads to
\begin{equation}\label{eqn-2022-03-16-4}
    \RN{1}_{\ep}\leq\frac{3}{(1-\de)(1-\ka)|\si'(0)|^2}|\ln\ep|=: C_1 |\ln \ep|.
\end{equation}

Now, we treat $\RN{2}_{\ep}$. Direct computation yields 
\begin{equation*}
\begin{split}
    \int_0^{x_*}[\ep^2 a'(0)+|\si'(0)|^2 y]^{-\ka}dy & \leq\frac{1}{(1-\ka)|\si'(0)|^2}[\ep^2 a'(0)+|\si'(0)|^2 x_*]^{1-\ka}\\
    &\leq \frac{2}{(1-\ka)|\si'(0)|^2}[|\si'(0)|^2 x_*]^{1-\ka}.
\end{split}
\end{equation*}
This together with similar arguments leading to \eqref{eqn-2022-03-16-2} yields $\RN{2}_{\ep}\leq C_{2}$ for some $C_{2}>0$.
% \begin{equation*}
%     \begin{split}
%         \RN{2}_{\ep}&\leq 2[|\si'(0)|^2x_*]^{\ka} e^{2\int_{x_*}^{x^*}\frac{|b|}{\si^2}ds} \int_{x_*}^{x^*}\frac{1}{\si^2(x)}\int_0^{x_*}[\ep^2 a'(0)+|\si'(0)|^2 y]^{-\ka}dy dx\\
%         &\quad+x^*e^{2\int_{x_*}^{x^*}\frac{|b|}{\si^2}ds}\int_{x_*}^{x^*}\frac{1}{\si^2}dx\leq \left(\frac{4x_*}{1-\ka}+x^*\right)e^{2\int_{x_*}^{x^*}\frac{|b|}{\si^2}ds}\int_{x_*}^{x^*}\frac{1}{\si^2}dx=:C_2.
%     \end{split}
% \end{equation*}
As a result of \eqref{eqn-2022-02-09-5-1} and \eqref{eqn-2022-03-16-4}, $\II_{\ep}\leq C_1 |\ln \ep|+C_2\leq 2C_1 |\ln \ep|$. From which and \eqref{eqn-2022-02-09-7-1}, the conclusion follows.
\end{proof}

%%%%%%%%%%%%%%%%%%%%%%%%%%

\subsection{Asymptotic bounds of the second eigenvalue}\label{subsec-asymptotic-second-eigenvalue}

The purpose of this subsection is to prove the asymptotic bounds of $\la_{\ep,2}$ stated in Theorem \ref{thm-bounds-eigenvalues} (2). Clearly, it follows from Lemma \ref{lem-upper-bound-eigenvalues} and the following result.
 
\begin{thm}\label{thm-lower-bound-lamda_2}
Assume {\bf (H)} and $\La_0>0$. Then, $\inf_{\ep}\la_{\ep,2}>0$. 
\end{thm}

We establish some lemmas before proving Theorem \ref{thm-lower-bound-lamda_2}. Recall from Lemma \ref{prop-spectral-structure} that $\la_{\ep,2}>\la_{\ep,1}>0$ are simple eigenvalues of $-\LL_{\ep}$ and have eigenfunctions in $D(\LL_{\ep})\cap L^1(u_{\ep}^G)\cap C^{2}((0,\infty))$. For $i=1,2$,  let $\tilde{\phi}_{\ep,i}$ be an eigenfunction of $-\LL_{\ep}$ associated with $\la_{\ep,i}$ and satisfy the normalization
\begin{equation}\label{normalization-2022-03-31}
    \|\tilde{\phi}_{\ep,i}\|_{L^1((0,2);u^G_{\ep})}+\|\tilde{\phi}_{\ep,i}\|^2_{L^2((1,\infty);u^G_{\ep})}=1.
\end{equation}
Such a normalization is chosen to acquire certain compactness of $\{\tilde{\phi}_{\ep,i}\}_{\ep}$ that plays a key role in the proof of Theorem \ref{thm-lower-bound-lamda_2}. Moreover, we let $\tilde{\phi}_{\ep,1}>0$. Direct calculations show that $u_{\ep, i}:=\tilde{\phi}_{\ep,i}u^G_{\ep}$ satisfies
\begin{equation}\label{eigen-equation-2022-04-27}
    \LL^*_{\ep}u_{\ep,i}=-\la_{\ep,i}u_{\ep,i}.
\end{equation}

The first lemma establishes an upper bound for $|u_{\ep,2}|$ near $0$.

\begin{lem}\label{lem-bound-u-ep-2-near-0}
Assume {\bf (H)}. For each $\ga>0$ and $0<\ep\ll 1$, there exist $C=C(\ep,\ga)>0$ and $x_*=x_*(\ep)>0$ such that $|u_{\ep,2}(x)|\leq Cx^{-\frac{1}{2}-\ga}$ for $x\in(0,x_*)$.
\end{lem}
\begin{proof}
Set $\psi_{\ep,2}:=\tilde{U}_{\ep}U_{\ep}\tilde{\phi}_{\ep,2}$, where $\tilde{U}_{\ep}$ and $U_{\ep}$ are unitary transforms defined in Subsection \ref{subsec-schrodinger-eqn}. Then, $\psi_{\ep,2}$ satisfies $-\LL_{\ep}^{S}\psi_{\ep,2}=\la_{\ep,2}\psi_{\ep,2}$ in $(0,y_{\ep,\infty})$, that is,
\begin{equation}\label{eqn-2021-09-28-2}
-\frac{1}{2}\psi''_{\ep,2}+W_{\ep}\psi_{\ep,2}=\la_{\ep,2}\psi_{\ep,2}\quad \text{in}\quad (0,y_{\ep,\infty}),
\end{equation}
where we recall that $y_{\ep,\infty}=\int_0^{\infty}\frac{1}{\sqrt{\ep^2 a+|\si|^2}}ds$. By {\bf (H)}(2)-(3), $y_{\ep,\infty}$ increases to $\infty$ as $\ep\to 0$.

Fix $\eta_{*}\in (0,1)$ (whose exactly value is to be determined) and $0<\de_{*}\ll1$. Let $\{\eta_{\de}\}_{0<\de<\de_{*}}$ be a family of  functions in $C^{\infty}_0((0,1))$, take values in $[0,\eta_*]$ and satisfy
\begin{equation}\label{eqn-2021-09-28-3}
\eta_{\de}(x)=\begin{cases}
0,& y\in (0,\de),\\
\eta_{*},& y\in (2\de,\frac{1}{2}),
\end{cases}
\quad 0\leq \eta'_{\de}\leq \frac{2\eta_{*}}{\de} \text{ on } [\de,2\de]\quad\andd\quad |\eta'_{\de}|\leq 4\eta_{*} \text{ on } \left[\frac{1}{2},1\right).
\end{equation}
Multiplying \eqref{eqn-2021-09-28-2} by $\eta^2_{\de}\psi_{\ep,2}$ and integrating by parts yield
\begin{equation}\label{eqn-2021-09-28-4}
    \frac{1}{2}\int_0^1 \eta^2_{\de}|\psi'_{\ep,2}|^2 dy+\int_0^1 \eta_{\de}\psi_{\ep,2} \eta'_{\de}\psi'_{\ep,2} dy+\int_0^1 W_{\ep}\eta^2_{\de} |\psi_{\ep,2}|^2dy=\la_{\ep,2}\int_0^1 \eta_{\de}^2 |\psi_{\ep,2}|^2dy.
\end{equation}

An application of H\"older's inequality and  \eqref{eqn-2021-09-28-3} yields
\begin{equation*}
    \begin{split}
    \left| \int_0^1 \eta_{\de}\psi_{\ep,2} \eta'_{\de}\psi'_{\ep,2} dy \right|%&\leq \left(\frac{1}{2} \int_0^1 \eta^2_{\de}|\psi'_{\ep,2}|^2 dy\right)^{\frac{1}{2}} \left(2\int_0^1 |\eta'_{\de}|^2|\psi_{\ep,2}|^2 dy\right)^{\frac{1}{2}}\\
    &\leq\frac{1}{4}\int_0^1 \eta^2_{\de}|\psi'_{\ep,2}|^2 dy+\int_0^1 |\eta'_{\de}|^2|\psi_{\ep,2}|^2 dy\\
    &\leq  \frac{1}{4}\int_0^1 \eta^2_{\de}|\psi'_{\ep,2}|^2 dy+\frac{4\eta_{*}^2}{\de^2}\int_{\de}^{2\de} |\psi_{\ep,2}|^2 dy+16\eta_{*}^2\int_{\frac{1}{2}}^{1} |\psi_{\ep,2}|^2 dy.
    \end{split}
\end{equation*}
Thanks to Lemma \ref{properties-schrodinger-potential} (2) and (4), we can find $C>0$ and $M>1$, and make $\de_*$ smaller if necessary (all independent of $\ep$) such that $W_{\ep}(y)\geq \frac{C}{y^2}\geq \frac{C}{4\de^2}$ for $y\in (\de,2\de)$ and $\de\in (0,\de_*)$ and $W_{\ep}+M\geq 1$. Setting $\eta_{*}:=\min\left\{\frac{\sqrt{C}}{4\sqrt{2}},\frac{1}{4\sqrt{2}}\right\}$, we derive
\begin{equation*}
    \begin{split}
     \left| \int_0^1 \eta_{\de}\psi_{\ep,2} \eta'_{\de}\psi'_{\ep,2} dy \right|&\leq \frac{1}{4}\int_0^1 \eta^2_{\de}|\psi'_{\ep,2}|^2 dy+\frac{C}{8\de^2}\int_{\de}^{2\de}|\psi_{\ep,2}|^2dy+\frac{1}{2}\int_{\frac{1}{2}}^1|\psi_{\ep,2}|^2 dy\\
     &\leq \frac{1}{4}\int_0^1 \eta^2_{\de}|\psi'_{\ep,2}|^2 dy+\frac{1}{2} \int_0^1 (W_{\ep}+M)\eta^2_{\de} |\psi_{\ep,2}|^2dy.
     \end{split}
\end{equation*}
It then follows from \eqref{eqn-2021-09-28-4} that
\begin{equation*}
     \frac{1}{4}\int_0^1 \eta^2_{\de}|\psi'_{\ep,2}|^2 dy+
    \frac{1}{2}\int_0^1 (W_{\ep}+M)\eta^2_{\de} |\psi_{\ep,2}|^2dy\leq \left(\la_{\ep,2}+\frac{M}{2}\right)\int_0^1 \eta^2_{\de} |\psi_{\ep,2}|^2dy,
\end{equation*}
and thus, 
\begin{equation*}
     \frac{1}{4}\int_0^1 \eta^2_{\de}|\psi'_{\ep,2}|^2 dy+
    \frac{1}{2}\int_0^1 \eta^2_{\de} |\psi_{\ep,2}|^2dy\leq \left(\la_{\ep,2}+\frac{M}{2}\right)\eta_*^2\int_0^1  |\psi_{\ep,2}|^2dy.
\end{equation*}

Note that it suffices to prove the result for each $0<\ga\ll1$. Let $\ga$ be such a number. As $\limsup_{\ep\to 0}\la_{\ep,2}<\infty$ (by Lemma \ref{lem-upper-bound-eigenvalues}) and  $\|\psi_{\ep,2}\|_{L^2((0,y_{\ep,\infty}))}=\|\tilde{\phi}_{\ep,2}\|_{L^2(u^G_{\ep})}<\infty$, the Sobolev embedding theorem gives $\|\eta_{\de}\psi_{\ep,2}\|_{C^{\frac{1}{2}-\ga}([0,1])}\leq C_1\|\tilde{\phi}_{\ep,2}\|_{L^2(u^G_{\ep})}$ for some $C_1=C_1(\ga)>0$  (independent of $\de$).
In particular, 
$$
|\eta_{\de}(y)\psi_{\ep,2}(y)|\leq C_1\|\tilde{\phi}_{\ep,2}\|_{L^2(u^G_{\ep})} y^{\frac{1}{2}-\ga},\quad \forall y\in\left(0,\frac{1}{2}\right).
$$
As $\eta_{\de}$ converges to $\eta_*$ on $(0,\frac{1}{2})$, we let $\de\to 0^+$ in the above inequality to find 
$$
|\psi_{\ep,2}(y)|\leq \frac{C_1}{\eta_*}\|\tilde{\phi}_{\ep,2}\|_{L^2(u^G_{\ep})} y^{\frac{1}{2}-\ga}=:C_2y^{\frac{1}{2}-\ga}, \quad \forall  y\in \left(0,\frac{1}{2}\right),
$$
where $C_2=C_{2}(\ep,\ga):=\frac{C_1}{\eta_*}\|\tilde{\phi}_{\ep,2}\|_{L^2(u^G_{\ep})}$. Setting $x^{\ep}_*:=\xi^{-1}_{\ep}(\frac{1}{2})$, we see from $u_{\ep,2}=\tilde{\phi}_{\ep,2}u^G_{\ep}$ and $\tilde{\phi}_{\ep,2}=U_{\ep}^{-1}\tilde{U}_{\ep}^{-1}\psi_{\ep,2}$ that 
\begin{equation}\label{eqn-2022-03-23-1}
\begin{split}
|u_{\ep,2}|=|\psi_{\ep,2}\circ\xi_{\ep}|\al_{\ep}^{-\frac{3}{4}} e^{-V_{\ep}}\leq C_2 \xi_{\ep}^{\frac{1}{2}-\ga}(\ep^2 a)^{-\frac{3}{4}} e^{-V_{\ep}}\quad\text{in}\quad(0,x^{\ep}_*).
\end{split}
\end{equation}

Since $\lim_{\ep\to 0} \xi_{\ep}(x)=\infty$ for any $x>0$, there must hold $\lim_{\ep\to 0}x^{\ep}_*=0$. Hence, $a(x)\geq \frac{1}{2}a'(0)x$ for $x\in (0,x^{\ep}_*)$. Then, 
$$
\xi_{\ep}(x)\leq \frac{\sqrt{2}}{\ep \sqrt{a'(0)}}\int_0^x\frac{1}{\sqrt{s}}ds = \frac{2\sqrt{2}}{\ep}\sqrt{\frac{x}{a'(0)}}, \quad \forall x\in (0,x^{\ep}_*).
$$
As $V'_{\ep}=-\frac{b}{\al_{\ep}}\leq0$ near $0$,  $V_{\ep}$ is non-increasing in $(0,x^{\ep}_*)$. It follows from \eqref{eqn-2022-03-23-1} that
\begin{equation*}
\begin{split}
|u_{\ep,2}(x)| %C_1[\xi_{\ep_0}(x)]^{\frac{1}{2}-\ga}[\ep^2_0 a(x)]^{-\frac{3}{4}} e^{-V_{\ep_0}(x)}\\
&\leq C_2  \left(\frac{2\sqrt{2}}{\ep}\sqrt{\frac{x}{a'(0)}}\right)^{\frac{1}{2}-\ga}\left[\frac{1}{2}\ep^2 a'(0)x\right]^{-\frac{3}{4}}e^{-V_{\ep}(x^{\ep}_*)}=:\frac{C}{x^{\frac{1}{2}+\frac{\ga}{2}}},\quad \forall x\in (0,x_*^{\ep}).
\end{split}
\end{equation*}
This completes the proof.
\end{proof}

The following result is in preparation for the contradiction arguments for $\tilde{\phi}_{\ep,i}$ to be used in the proof of Theorem \ref{thm-lower-bound-lamda_2}.

\begin{lem}\label{lem-compactness-2-0}
Assume {\bf (H)} and $\La_0>0$. 
\begin{enumerate}
    \item[\rm(1)] $\lim_{z\to0}\sup_{\ep}\int_0^{z}\tilde{\phi}_{\ep,1} u^G_{\ep}dx=0$.
    
    \item[\rm(2)] If $\lim_{\ep\to 0}\la_{\ep,2}=0$, then $\lim_{z\to0}\sup_{\ep}\int_0^{z}|\tilde{\phi}_{\ep,2}| u^G_{\ep}dx=0$.
\end{enumerate}
\end{lem}
\begin{proof}
As $\lim_{\ep\to0}\la_{\ep,1}=0$ by Theorem \ref{thm-0-super-limit-principal-eigenvalue}, the proof is done if we show that for $i=1,2$, the condition $\lim_{\ep\to 0}\la_{\ep,i}=0$ implies $\lim_{z\to0}\sup_{\ep}\int_0^{z}|\tilde{\phi}_{\ep,i}| u^G_{\ep}dx=0$. As $u_{\ep, i}=\tilde{\phi}_{\ep,i}u^G_{\ep}$, it is the same as showing  
\begin{equation}\label{limit-2022-04-27}
\lim_{z\to0}\sup_{\ep}\int_0^{z}|u_{\ep,i}|dx=0.
\end{equation}

We proceed as in the proof of Proposition \ref{prop-concentration-0}. Given $\lim_{\ep\to 0}\la_{\ep,i}=0$, we apply Lemma \ref{lem-upper-soln-x^-k} to find for fixed $k\in (\frac{1}{2},1)$ the existence of $x_1\in (0,1)$ such that $(\LL^*_{\ep}+\la_{\ep,i}) x^{-k}<0$ in $(0,x_1)$. Setting $v_{\ep,i}:=\frac{u_{\ep,i}}{x^{-k}}$, we compute using \eqref{eigen-equation-2022-04-27}
\begin{equation}\label{eqn-2022-03-17-1}
\frac{1}{2}\al_{\ep} v''_{\ep,i}+\left(\al'_{\ep}-b-\frac{k}{x}\al_{\ep}\right) v'_{\ep,i}+\frac{(\LL^*_{\ep}+\la_{\ep,i})x^{-k}}{x^{-k}}v_{\ep,i}=0 \quad \text{in}\quad (0,\infty).
\end{equation}

Note that $\lim_{x\to 0}|v_{\ep,i}(x)|=0$. Indeed, Lemma \ref{lem-bound-u-ep-2-near-0} implies that there are $C_{\ep}>0$ and $x_{\ep}>0$ such that $|u_{\ep,i}(x)|\leq C_{\ep} x^{-\frac{1}{2}(k+\frac{1}{2})}$ for $x\in(0,x_{\ep})$. Hence, 
$\lim_{x\to 0}|v_{\ep,i}(x)|\leq C_{\ep}\lim_{x\to 0}x^{k-\frac{1}{2}(k+\frac{1}{2})}=0$.

Let $x_2\in (0,x_1)$. Note that the equation \eqref{eigen-equation-2022-04-27} can be written as
\begin{equation}\label{eqn-2022-04-27}
    \frac{1}{2}\left(\al_{\ep}u'_{\ep,i}\right)'+\left[\left(\frac{1}{2}\al'_{\ep}-b\right)u_{\ep,i}\right]'+\la_{\ep,i}u_{\ep,i}=0.
\end{equation}
Due to the first limit in \eqref{limit-alpha-V} and $\lim_{\ep\to 0}\la_{\ep,i}=0$ (by assumption), we apply the classical interior De Giorgi-Nash-Moser estimates (see e.g. \cite{CW98, GT01}) in $(\frac{x_2}{2},\frac{3x_2}{2})$ to find $C>0$ (independent of $\ep$) such that
\begin{equation*}
    \sup_{(\frac{3 x_2}{4},\frac{5x_2}{4})}|u_{\ep,i}|\leq \frac{C}{x_2}\int_{\frac{x_2}{2}}^{\frac{3x_2}{2}} |u_{\ep,i}|dx= \frac{C}{x_2}\int_{\frac{x_2}{2}}^{\frac{3x_2}{2}} |\tilde{\phi}_{\ep,i}|u^G_{\ep}dx\leq \frac{C}{x_2}, 
\end{equation*}
where we used the normalization \eqref{normalization-2022-03-31} in the last inequality. Hence, $|v_{\ep,i}(x_2)|\leq C x_2^{k-1}$. 

Since $v_{\ep,i}$ satisfies \eqref{eqn-2022-03-17-1}, $\lim_{x\to 0}v_{\ep,i}(x)=0$ and  $(\LL^*_{\ep}+\la_{\ep,i}) x^{-k}<0$, we apply the maximum principle to $v_{\ep,i}$ in $(0,x_2)$ and conclude that $\sup_{x\in (0,x_2)}|v_{\ep,i}(x)|=|v_{\ep,i}(x_2)|\leq C x_2^{k-1}$, leading to $|u_{\ep,i}(x)|\leq C x_2^{k-1} x^{-k}$ for $x\in (0,x_2)$. Thus,
\begin{equation*}
    \int_0^{x^2_2}|\tilde{\phi}_{\ep,i}|u^G_{\ep}dx=\int_0^{x^2_2}|u_{\ep,i}|dx\leq C x_2^{k-1}\int_0^{x^2_2}x^{-k}dx=\frac{C}{1-k} x_2^{1-k}.
\end{equation*}
Since the above estimate holds for any $x_{2}\in(0,x_{1})$ and is uniform in $0<\ep\ll1$, we arrive at \eqref{limit-2022-04-27}, and hence, proves the lemma.
\end{proof}

The monotonicity of $\tilde{\phi}_{\ep,1}$ is addressed in the next result.

\begin{lem}\label{monotonicity-first-eigenfunction-2022-04-07}
Assume {\bf (H)}. There holds $\tilde{\phi}_{\ep,1}'>0$.
\end{lem}
\begin{proof}
Note that $\tilde{\phi}_{\ep,1}$ satisfies $\LL_{\ep}\tilde{\phi}_{\ep,1}=-\la_{\ep,1}\tilde{\phi}_{\ep,1}$, or $\left( e^{-2V_{\ep}}\tilde{\phi}'_{\ep,1}\right)'=-\frac{2}{\al_{\ep}}\la_{\ep,1}e^{-2V_{\ep}}\tilde{\phi}_{\ep,1}$. Since $\tilde{\phi}_{\ep,1}>0$, $e^{-2V_{\ep}}\tilde{\phi}'_{\ep,1}$ is strictly decreasing.  

Suppose on the contrary that $\tilde{\phi}'_{\ep,1}(x_0)\leq 0$ for some $x_0\in(0,\infty)$. Then, there is $x_1>x_0$ such that 
\begin{equation*}
    e^{-2V_{\ep}(x)}\tilde{\phi}'_{\ep,1}(x)<e^{-2V_{\ep}(x_1)}\tilde{\phi}'_{\ep,1}(x_1)<0,\quad \forall x>x_1,
\end{equation*}
yielding $\tilde{\phi}'_{\ep,1}(x)<e^{2(V_{\ep}(x)-V_{\ep}(x_1))}\tilde{\phi}_{\ep,1}'(x_1)$ for $x>x_1$. As $V_{\ep}(x)-V_{\ep}(x_1)=-\int_{x_1}^x\frac{b}{\al_{\ep}}ds\to\infty$ as $x\to\infty$, we find $\tilde{\phi}'_{\ep,1}(x)\to -\infty$ as $x\to \infty$, and hence, $\limsup_{x\to \infty}\tilde{\phi}_{\ep,1}(x)=-\infty$, contradicting $\tilde{\phi}_{\ep,1}>0$. Hence, $\tilde{\phi}'_{\ep,1}>0$. 
\end{proof}

The last lemma is elementary. 

\begin{lem}\label{lem-integral-gibb}
Assume {\bf (H)}. Then, $\sup_{\ep}\int_1^{\infty}u^G_{\ep}dx<\infty$.
\end{lem}
\begin{proof}
By {\bf (H)}(1)(4), there exist $x_1>0$ and $\ep_*>0$ such that $b<0$ and $\ep^2a+\si^2\leq 2 \si^2$ in $(x_1,\infty)$ for all $0<\ep<\ep_*$. Therefore, 
\begin{equation*}
    \begin{split}
    u^G_{\ep}(x)&=\exp\left\{2\int_1^x\frac{b}{\ep^2a+\si^2}ds-\ln (\ep^2 a(x)+\si^2(x))\right\}\\
    &\leq \exp\left\{2\int_1^{x_1}\frac{b}{\ep^2a+\si^2}ds\right\}\times \exp\left\{\int_{x_1}^x\frac{b}{2\si^2}ds-\ln \si^2(x)\right\},\quad \forall x>x_1.
    \end{split}
\end{equation*}

Obviously, the conclusion follows if we show 
$\int_{x_1}^{\infty}\exp\left\{\int_{x_1}^x\frac{b}{2\si^2}ds-\ln \si^2(x)\right\}dx<\infty$, which can be verified by arguments as in the proof of  Lemma \ref{lem-integral-gibbs-zero-demographic}.
\end{proof}

We are ready to prove Theorem \ref{thm-lower-bound-lamda_2}.

\begin{proof}[Proof of Theorem \ref{thm-lower-bound-lamda_2}]
Suppose on the contrary that $\inf_{\ep}\la_{\ep,2}=0$. Up to a subsequence, we may assume, without loss of generality, that $\lim_{\ep\to 0}\la_{\ep,2}= 0$. We derive a contradiction within four steps. 

\medskip

\noindent{\bf Step 1.} We show for $i=1,2$ the existence of $u_{i}\in C^{2}((0,\infty))\cap L^1((0,\infty))$ satisfying $\LL^*_0 u_i=0$ such that $\lim_{\ep\to0}u_{\ep,i}=u_{i}$ in $C^2((0,\infty))$. 

Recall that $u_{\ep,i}=\tilde{\phi}_{\ep,i}u^G_{\ep}$ and $\tilde{\phi}_{\ep,i}$ satisfies \eqref{normalization-2022-03-31}. We apply H\"older's inequality to find
\begin{equation}\label{eqn-2022-04-1-02}
\begin{split}
    \sup_{\ep}\int_{0}^{\infty}|u_{\ep,i}|dx%&\leq 1+\sup_{\ep}\int_{2}^{\infty}|u_{\ep,i}|dx\\
    &\leq1+\sup_{\ep}\int_{2}^{\infty}|\tilde{\phi}_{\ep,i}|u^G_{\ep}dx\\
    &\leq 1+\sup_{\ep}\left( \int_{2}^{\infty}|\tilde{\phi}_{\ep,i}|^2 u^G_{\ep}dx\right)^{\frac{1}{2}} \left(\int_2^{\infty} u^G_{\ep}dx\right)^{\frac{1}{2}}
    \\
    &\leq 1+\sup_{\ep}\left(\int_2^{\infty} u^G_{\ep}dx\right)^{\frac{1}{2}}<\infty,
    \end{split}
\end{equation}
where we used the normalization \eqref{normalization-2022-03-31} in the first and last inequalities, and Lemma \ref{lem-integral-gibb} to derive the final uniform boundedness. Considering the positive and negative parts of $\{u_{\ep,i}\}_{\ep}$ separately, we apply Helly's selection principle (see e.g. \cite[Theorem 4.3.3 and 4.4.1]{Chung2001}) to find a signed Borel measure $\mu_i$ on $(0,\infty)$ such that, up to a subsequence, $\lim_{\ep\to 0}\int_0^{\infty}\phi u_{\ep,i}dx=\int_0^{\infty}\phi d\mu_i$ for any $\phi\in C_c((0,\infty))$.

Letting $\ep\to0$ in \eqref{eigen-equation-2022-04-27}, we find $\LL^*_0 \mu_i=0$ in the weak sense from the first limit in \eqref{limit-alpha-V} and $\lim_{\ep\to 0}\la_{\ep,i}=0$ (by Theorem \ref{thm-0-super-limit-principal-eigenvalue} if $i=1$ and assumption if $i=2$). Moreover, we apply the classical interior De Giorgi-Nash-Moser estimates (see e.g. \cite{CW98, GT01}) to find that for any open intervals $\II$ and $\II'$ with $\II\stst\II'\stst(0,\infty)$, there exists $C=C(\II,\II')>0$ (independent of $\ep$) such that
$\sup_{\II}|u_{\ep,i}|\leq C\|u_{\ep,i}\|_{L^1(\II')}$. Then, we can follow the arguments as in the proof of Theorem \ref{thm-concentration} (2) to conclude that $\mu_{i}$ admits a density $u_{i}\in C^{2}((0,\infty))$ and $\lim_{\ep\to0}u_{\ep,i}=u_i$ in $C^2((0,\infty))$. The estimate \eqref{eqn-2022-04-1-02} and Fatou's lemma guarantee $u_i\in L^1((0,\infty))$. 

\medskip

\noindent{\bf Step 2.} We show the existence of $C_1>0$ and $C_2\neq 0$ such that $\lim_{\ep\to0}\tilde{\phi}_{\ep,i}=C_{i}$ in $C^{2}((0,\infty))$ for $i=1,2$.

By {\bf Step 1} and $\tilde{\phi}_{\ep,1}>0$, we apply Lemma \ref{lem-unique-L^1-soln} to find  $C_1\geq 0$ and $C_2\in \R$ such that $u_i=C_i u^G_0$ for $i=1,2$. Recall $\tilde{\phi}_{\ep,i}=\frac{u_{\ep,i}}{u^G_{\ep}}$. Thanks to $\lim_{\ep\to0}u_{\ep,i}=u_{i}$ in $C^{2}((0,\infty))$ (by {\bf Step 1}) and \eqref{gibbs-to-gibbs}, the limit $\lim_{\ep\to0}\tilde{\phi}_{\ep,i}=C_i$ holds in $C^{2}((0,\infty))$. 

By Lemmas \ref{lem-compactness-2-infty} and \ref{lem-compactness-2-0}, the normalization \eqref{normalization-2022-03-31} ensures the existence of some $\ka\gg1$ such that $\int_{\frac{1}{\ka}}^{2}|\tilde{\phi}_{\ep,i}|u^G_{\ep}dx + \int_{1}^{\ka}|\tilde{\phi}_{\ep,i}|^2 u^G_{\ep}dx\geq \frac{1}{2}$.
Letting $\ep\to 0$ yields $|C_i|\int_{\frac{1}{\ka}}^{2}u^G_{0}dx+C_i^2\int_{1}^{\ka}u^G_{0}dx\geq\frac{1}{2}$. Hence, $C_i\neq 0$. In particular, $C_1>0$.

\medskip

\noindent{\bf Step 3.} We show that $\lim_{\ep\to 0}\int_0^{\infty}\tilde{\phi}_{\ep,1}\tilde{\phi}_{\ep,2}u^G_{\ep}dx=C_1C_2\int_0^{\infty}u^G_0dx\neq 0$.

Obviously, for any  $\ka>1$, there holds 
\begin{equation}\label{eqn-2021-09-13-2}
\begin{split}
&\left|\int_0^{\infty}\tilde{\phi}_{\ep,1}\tilde{\phi}_{\ep,2}u^G_{\ep}dx-C_1C_2\int_0^{\infty} u^G_0dx\right|\\
&\qquad\leq \left|\int^{\ka}_{\frac{1}{\ka}}\tilde{\phi}_{\ep,1}\tilde{\phi}_{\ep,2}u^G_{\ep}dx-C_1C_2\int^{\ka}_{\frac{1}{\ka}} u^G_0 dx\right| + \left(\int_0^{\frac{1}{\ka}}+\int_{\ka}^{\infty}\right)\tilde{\phi}_{\ep,1}|\tilde{\phi}_{\ep,2}|u^G_{\ep}dx\\
&\qquad\quad +C_1|C_2|\left(\int_0^{\frac{1}{\ka}}+\int_{\ka}^{\infty}\right) u^G_0dx=:\RN{1}_{\ep}(\ka)+\RN{2}_{\ep}(\ka)+\RN{3}(\ka).
    \end{split}
\end{equation}
We claim that 
\begin{equation}\label{claim-2022-07-13}
\lim_{\ep\to0}\RN{1}_{\ep}(\ka)=0,\quad\forall \ka>1,\quad  \lim_{\ka\to\infty}\sup_{\ep}\RN{2}_{\ep}(\ka)=0\quad\text{and}\quad  \lim_{\ka\to\infty}\RN{3}(\ka)=0.
\end{equation}
Given \eqref{claim-2022-07-13}, the conclusion follows from taking the limit $\ep\to0$ and then $\ka\to\infty$ in \eqref{eqn-2021-09-13-2}.

We prove \eqref{claim-2022-07-13}. Clearly, $\int_0^{\infty}u_0^Gdx<\infty$ (by Lemma \ref{lem-integral-gibbs-zero-demographic}) yields $\lim_{\ka\to\infty}\RN{3}(\ka)=0$. Thanks to {\bf Step 2} and  \eqref{gibbs-to-gibbs}, we see $\lim_{\ep\to0}\RN{1}_{\ep}(\ka)=0$ for any $\ka>1$.

% Note that
% \begin{equation*}
%     \RN{1}_{\ep}(\ka)\leq \int_{\frac{1}{\ka}}^{\ka}\tilde{\phi}_{\ep,1}|\tilde{\phi}_{\ep,2}-C_2|u^G_{\ep}dx+|C_2|\int_{\frac{1}{\ka}}^{\ka}|\tilde{\phi}_{\ep,1}-C_1|u^G_{\ep}dx+C_1|C_2|\int_{\frac{1}{\ka}}^{\ka}|u^G_{\ep}-u^G_{0}|dx.
% \end{equation*}
% This together with {\bf Step 2}, \eqref{gibbs-to-gibbs} and $\sup_{\ep}\int_0^{\infty}|\tilde{\phi}_{\ep,1}|u^G_{\ep}dx<\infty$ (due to \eqref{eqn-2022-04-1-02}) implies that $\lim_{\ep\to0}\RN{1}_{\ep}(\ka)=0$ for any $\ka>1$.

For $\RN{2}_{\ep}(\ka)$, H\"{o}lder's inequality and Lemma \ref{lem-compactness-2-infty} yield $\lim_{\ka\to\infty}\sup_{\ep}\int_{\ka}^{\infty}\tilde{\phi}_{\ep,1}|\tilde{\phi}_{\ep,2}|u^G_{\ep}dx=0$. Note that Lemma \ref{monotonicity-first-eigenfunction-2022-04-07} and {\bf Step 2} imply
\begin{equation}\label{uniform-bound-2022-04-09}
    \sup_{\ep}\sup_{(0,x)}\tilde{\phi}_{\ep,1}\leq \sup_{\ep} \tilde{\phi}_{\ep,1}(x)<\infty,\quad\forall x>0,
\end{equation}
% \begin{equation*}\label{claim-2022-04-06}
% \sup_{\ep}\sup_{x\in (0,x_*)}\tilde{\phi}_{\ep,1}(x)<\infty,\quad\forall x_*>0.
% \end{equation*}
which together with Lemma \ref{lem-compactness-2-0} (2) leads to $\lim_{\ka\to\infty}\sup_{\ep}\int_{0}^{\frac{1}{\ka}} \tilde{\phi}_{\ep,1}|\tilde{\phi}_{\ep,2}|u^G_{\ep}dx=0$. It follows that $\lim_{\ka\to\infty}\sup_{\ep}\RN{2}_{\ep}(\ka)=0$. The claim \eqref{claim-2022-07-13}, and hence, the conclusion is proven. 

\medskip

\noindent{\bf Step 4.} Recall that $-\LL_{\ep}$ is self-adjoint in $L^2(u^G_{\ep})$. Since $\tilde{\phi}_{1,\ep}$ and $\tilde{\phi}_{2,\ep}$ are eigenfunctions associated with $\la_{\ep,1}$ and $\la_{\ep,2}$, respectively, they are orthogonal in $L^2(u^G_{\ep})$, namely, $\int_0^{\infty}\tilde{\phi}_{\ep,1}\tilde{\phi}_{\ep,2} u^G_{\ep}dx=0$, which is in contradictory to {\bf Step 3}. 

In conclusion, $\inf_{\ep} \la_{\ep,2}>0$ and the theorem is proven. 
\end{proof}
\section{\bf Multiscale dynamics} \label{s:multi}

In this section, we study the multiscale dynamics of the distribution of $X_{t}^{\ep}$ and prove Theorem \ref{thm-multiscale}. Recall from Lemma \ref{prop-spectral-structure} the semigroup $(P_{t}^{\ep})_{t\geq0}$ and for each $k\in\N$ the spectral projection  $Q_{k}^{\ep}$ of $\LL_{\ep}$ corresponding to the eigenvalues $\{-\la_{\ep,j}\}_{j\geq k}$. The following lemma plays a crucial role in the proof of Theorem \ref{thm-multiscale}. 

\begin{lem}\label{lem-semigroup-property}
Assume {\bf (H)} and $\La_0>0$. For each $k\in \N$, there is $C_{k}>0$ such that for $0<\ep\ll 1$,
\begin{equation*}
    |P^{\ep}_tQ^{\ep}_kf|\leq C_{k}\al_{\ep}^{\frac{1}{4}}e^{V_{\ep}} e^{-\la_{\ep,k}t}\|f\|_{L^{\infty}}\quad\text{in}\quad (0,\infty),\quad \forall t>2\andd f\in C_b((0,\infty)).
\end{equation*}
% \begin{enumerate}
%     \item If $\La_0>0$, then there exists $C=C(k)>0$ such that for $0<\ep\ll 1$,
% \begin{equation*}
%     |P^{\ep}_tQ^{\ep}_kf|\leq C\al_{\ep}^{\frac{1}{4}}e^{V_{\ep}} e^{-\la_{\ep,k}t}\|f\|_{L^{\infty}}\quad\text{in}\quad (0,\infty),\quad \forall t>2\andd f\in C_b((0,\infty)).
% \end{equation*}
%     \item If $\La_0<0$, then there exist $C>0$ and $\ga>0$ such that for each $0<\ep\ll 1$, 
%     \begin{equation*}
%     |P^{\ep}_{t}Q^{\ep}_kf|\leq  \frac{C}{\ep^{\ga}} \al_{\ep}^{\frac{1}{4}} e^{V_{\ep}}e^{-\la_{\ep,k}t} \|f\|_{L^\infty}\quad\text{in}\quad (0,\infty),\quad \forall t>2\andd f\in C_b((0,\infty)).
% \end{equation*}
% \end{enumerate}
\end{lem}
\begin{proof}
%We proceed as in as in the proof of \cite[Lemma 6.1]{JQSY21}. Fix $k\in\N$. 
Set $\Tilde{P}^{\ep}_t:=\tilde{U}_{\ep}U_{\ep}P^{\ep}_tU_{\ep}^{-1} \tilde{U}_{\ep}^{-1}$, where $U_{\ep}$ and $\tilde{U}_{\ep}$ are unitary transforms specified in Subsection \ref{subsec-schrodinger-eqn}. Then,  $(\tilde{P}^{\ep}_t)_{t\geq 0}$ is an analytic semigroup of contractions on $L^2((0,y_{\ep,\infty}))$ generated by $\LL_{\ep}^S$. The spectrum of $\LL_{\ep}^S$, being the same as that of $\LL_{\ep}$, consists of simple eigenvalues $\{-\la_{\ep,i}\}_{i\in\N}$ (see Lemma \ref{prop-spectral-structure}). We finish the proof within four steps.

\medskip

\noindent{\bf Step 1.} We show for each $p\in (2,\infty]$, there is $D_1(p)>0$  such that
$\sup_{\ep}\|\tilde{P}^{\ep}_1\|_{L^{2}\to L^p}\leq D_1(p)$.

According to Lemma \ref{properties-schrodinger-potential} (4), there is $M>0$ such that $W_{\ep}+M\geq 1$.  Since $\si(\LL_{\ep}^{S}-M)\subset(-\infty,-M)$ and
$\|(\la-(\LL_{\ep}^S-M))^{-1}\|_{L^2\to L^2}=\frac{1}{{\rm dist}(\la,\si(\LL_{\ep}^{S}-M))}$ for all $\la\in\rho(\LL_{\ep}^{S}-M)$, we find
\begin{equation}\label{eqn-2020-8-10-6}
\|(\la-(\LL_{\ep}^S-M))^{-1}\|_{L^2\to L^2}\leq \frac{1}{|\la|},\quad \forall  \la \in \C\text{ with }\Re\la>0.
\end{equation}
As $\LL_{\ep}^S-M$ generates the analytic semigroup  $(e^{-Mt}\tilde{P}^{\ep}_t)_{t\geq 0}$ of contractions on $L^2((0,y_{\ep,\infty}))$, and the right-hand side of \eqref{eqn-2020-8-10-6} is independent of $\ep$, we apply \cite[Theorem 2.5.2]{Pa83} to find $C_1>0$ (independent of $\ep$) such that  
\begin{equation}\label{eqn-2020-8-9-1}
    \|(\LL_{\ep}^S-M)e^{-Mt}\tilde{P}^{\ep}_{t}\|_{L^2\to L^{2}}\leq \frac{C_1}{t},\quad \forall  t>0.
\end{equation}

Let $D(\LL_{\ep}^S)$ be the domain of $\LL_{\ep}^{S}$. Since
\begin{equation*}
    \begin{split}
        \langle-(\LL_{\ep}^S-M)u,u\rangle_{L^2}=\frac{1}{2}\int_0^{y_{\ep,\infty}}|u'| dy+\int_0^{y_{\ep,\infty}} (W_{\ep}+M) |u|^2 dy,\quad\forall u\in D(\LL_{\ep}^S),
    \end{split}
\end{equation*}
we derive from $W_{\ep}+M\geq 1$ and \eqref{eqn-2020-8-9-1}  that for $\tilde{f}\in L^{2}((0,y_{\ep,\infty}))$ and $t>0$,
\begin{equation*}
\begin{split}
    &\frac{1}{2}\int_0^{y_{\ep,\infty}}|\pa_y \tilde{P}^{\ep}_{t}\tilde{f}|^2 dy+\int_0^{y_{\ep,\infty}} |\tilde{P}^{\ep}_{t}\tilde{f}|^2 dy\\
    &\qquad\leq\frac{1}{2}\int_0^{y_{\ep,\infty}}|\pa_y \tilde{P}^{\ep}_{t}\tilde{f}|^2 dy+\int_0^{y_{\ep,\infty}} (W_{\ep}+M) |\tilde{P}^{\ep}_{t}\tilde{f}|^2 dy\\
    &\qquad=\langle-(\LL_{\ep}^S-M)\tilde{P}^{\ep}_{t}\tilde{f},\tilde{P}^{\ep}_{t}\tilde{f}\rangle_{L^2}\\
    %&\qquad\leq \|(\LL_{\ep}^S-M)\tilde{P}^{\ep}_{t}\tilde{f}\|_{L^2} \|\tilde{P}^{\ep}_{t}\tilde{f}\|_{L^2}\\
    &\qquad\leq \frac{C_1e^{Mt}}{t}\|\tilde{f}\|_{L^2}\|\tilde{P}^{\ep}_{t}\tilde{f}\|_{L^2}\leq \frac{C^2_1e^{2Mt}}{2t^2}\|\tilde{f}\|^2_{L^2}+\frac{1}{2} \int_0^{y_{\ep,\infty}}  |\tilde{P}^{\ep}_{t}\tilde{f}|^2 dy,
\end{split}
\end{equation*}
leading to
$$
\int_0^{y_{\ep,\infty}}|\pa_y \tilde{P}^{\ep}_{t}\tilde{f}|^2 dy+\int_0^{y_{\ep,\infty}} |\tilde{P}^{\ep}_{t}\tilde{f}|^2 dy\leq \frac{C^2_1e^{2Mt}}{t^2}\|\tilde{f}\|^2_{L^2}.
$$
Since Lemma \ref{properties-schrodinger-potential} (2)-(3) ensures $W_{\ep}+M$ blows up at $0$ and $y_{\ep,\infty}$, we see that $\tilde{P}^{\ep}_t\tilde{f}$ belongs to $W_0^{1,2}((0,y_{\ep,\infty}))$ (the closure of $C_0^{\infty}((0,y_{\ep,\infty}))$ under the $W^{1,2}((0,y_{\ep,\infty}))$-norm). Hence, 
the Sobolev embedding theorem ensures that for each $p>2$ there is $C_2(p)>0$ such that
\begin{equation*}\label{eqn-2021-09-11-1}
\begin{split}
\|\tilde{P}^{\ep}_t\tilde{f}\|_{L^{p}}&\leq C_2(p)\left( \|\pa_y \tilde{P}^{\ep}_{t}\tilde{f}\|_{L^2}+\|\tilde{P}^{\ep}_{t}\tilde{f}\|_{L^2}\right)\leq \frac{\sqrt{2}C_1C_2(p)}{t} e^{Mt}\|\tilde{f}\|_{L^2}.
\end{split}
\end{equation*}
Setting $t=1$ yields the result with $D_1(p):=\sqrt{2}C_1C_2(p)e^M$.

\medskip

\noindent{\bf Step 2.} We prove that for each $p\in (1,2)$, there holds $\sup_{\ep}\|\tilde{P}^{\ep}_1\|_{L^p\to L^{2}}\leq D_1(p')$, where $p'$ is the dual exponent of $p$, namely, $\frac{1}{p}+\frac{1}{p'}=1$.

The result in {\bf Step 1} says $\|\tilde{P}^{\ep}_1\|_{L^2\to L^{p'}}\leq D_1(p')$, which together with the symmetry of  $\tilde{P}^{\ep}_1$ yields
\begin{equation*}\label{eqn-2020-8-10-4}
\|\tilde{P}^{\ep}_1\tilde{f}\|_{L^2}\leq D_1(p') \|\tilde{f}\|_{L^{p}},\quad\forall \tilde{f}\in L^2((0,y_{\ep,\infty}))\cap L^{p}((0,y_{\ep,\infty})).
\end{equation*}
Thus, $\tilde{P}^{\ep}_1$ uniquely extends to be a bounded linear operator from $L^{p}((0,y_{\ep,\infty}))$ to $L^{2}((0,y_{\ep,\infty}))$, and satisfies $\|\tilde{P}^{\ep}_1\|_{L^{p}\to L^2}\leq D_1(p')$.

\medskip

\noindent{\bf Step 3.} We show the existence of $p_*\in (1,2)$ and $D_2>0$ such that $\sup_{\ep}\|\tilde{U}_{\ep}U_{\ep}f\|_{L^{p_*}}\leq D_2 \|f\|_{\infty}$ for all $f\in C_b((0,\infty))$.

Let $f\in C_b((0,\infty))$ and set $\tilde{f}_{\ep}:=\tilde{U}_{\ep}U_{\ep}f$. Straightforward calculations yield that for each $p\in (1,2)$,
\begin{equation*}
\begin{split}
        \int_0^{y_{\ep,\infty}}|\tilde{f}_{\ep}|^p dy
        &=\int_0^{y_{\ep,\infty}} \left|\sqrt{v^{G}_{\ep}}f\circ \xi_{\ep}^{-1}\right|^p dy
        \\
        &=\int_0^{y_{\ep,\infty}}\left((u^G_{\ep})^{\frac{p}{2}}\al_{\ep}^{\frac{p}{4}}|f|^{p}\right)\circ \xi_{\ep}^{-1}dy\leq \|f\|^p_{\infty}\int_0^{\infty}\frac{e^{-pV_{\ep}}}{\al_{\ep}^{\frac{p}{4}+\frac{1}{2}}}dx. 
\end{split}
\end{equation*}
Note that if there exists $p_{*}\in(1,2)$  such that
\begin{equation}\label{eqn-2022-03-24-4}
   \sup_{\ep}\int_0^{\infty}\frac{e^{-p_{*}V_{\ep}}}{\al_{\ep}^{\frac{p_{*}}{4}+\frac{1}{2}}}dx<\infty,
\end{equation}
then the result holds with $D_{2}=\sup_{\ep}\left(\int_0^{\infty}\frac{e^{-p_{*}V_{\ep}}}{\al_{\ep}^{\frac{p_{*}}{4}+\frac{1}{2}}}dx\right)^{\frac{1}{p_{*}}}$.

We show \eqref{eqn-2022-03-24-4} for some $p_{*}\in(1,2)$. Fix $0<\de\ll 1$. By {\bf (H)}(1)-(3), there is $x_1\in(0,1)$ such that 
\begin{gather}\label{eqn-2022-02-18-1}
    b(x)\geq (1-\de)b'(0)x\,\, \andd \,\, 
    1-\de\leq \frac{\al_{\ep}(x)}{ x(\ep^2 a'(0)+|\si'(0)|^2 x)}\leq 1+\de,\quad\forall x\in (0,x_1). 
\end{gather}

We split 
\begin{equation*}
    \begin{split}
        \int_0^{\infty}\frac{e^{-pV_{\ep}}}{\al_{\ep}^{\frac{p}{4}+\frac{1}{2}}}dx&=\int_0^{x_1}\frac{e^{-pV_{\ep}}}{\al_{\ep}^{\frac{p}{4}+\frac{1}{2}}}dx+\int_{x_1}^{\infty}\frac{e^{-pV_{\ep}}}{\al_{\ep}^{\frac{p}{4}+\frac{1}{2}}}dx=:\RN{1}_{\ep}(p)+\RN{2}_{\ep}(p).
    \end{split}
\end{equation*}
Following arguments as in the proof of Lemma \ref{lem-integral-gibbs-zero-demographic}, we find $\sup_{\ep}\RN{2}_{\ep}(p)<\infty$ for each $p\in (1,2)$. 

Now, we treat $\RN{1}_{\ep}(p)$. We deduce from \eqref{eqn-2022-02-18-1} that
\begin{equation}\label{eqn-2022-03-24-2}
    \begin{split}
    -V_{\ep}(x)&=\int_{x_1}^x\frac{b}{\al_{\ep}}ds+\int_1^{x_1}\frac{b}{\al_{\ep}}ds\\
    &\leq \int_{x_1}^x\frac{(1-\de)b'(0)}{(1+\de)(\ep^2 a'(0)+|\si'(0)|^2 s)}ds+\int_1^{x_1}\frac{b}{\al_{\ep}}ds\\
    &\leq \frac{(1-\de)b'(0)}{(1+\de)|\si'(0)|^2}\ln \frac{\ep^2 a'(0)+|\si'(0)|^2 x}{\ep^2 a'(0)+|\si'(0)|^2 x_1}+\int_{x_1}^1\frac{|b|}{\si^2}ds,\quad\forall x\in (0,x_1).
    \end{split}
\end{equation}
As $\La_0>0$, $\ka:=\frac{2(1-\de)b'(0)}{(1+\de)|\si'(0)|^2}>1$. Fix some $p_*\in\left(\max\left\{1,\left(\ka-\frac{1}{2}\right)^{-1}\right\},2\right)$. It follows from \eqref{eqn-2022-02-18-1} and  \eqref{eqn-2022-03-24-2} that
\begin{equation*}
    \begin{split}
        \RN{1}_{\ep}(p_*)
        & 
        \leq \int_0^{x_1}\left[(1-\de)(\ep^2 a'(0)x+|\si'(0)|^2 x^2)\right]^{-\frac{p_*}{4}-\frac{1}{2}}\left[\frac{\ep^2 a'(0)+|\si'(0)|^2 x}{\ep^2 a'(0)+|\si'(0)|^2 x_1}\right]^{\frac{p_*\ka}{2}}e^{p_*\int_{x_1}^1\frac{|b|}{\si^2}ds}dx
        \\
        &\leq\frac{C}{[\ep^2 a'(0)+|\si'(0)|^2 x_1]^{\frac{p_*\ka}{2}}}\int_0^{x_1}\frac{[\ep^2 a'(0)+|\si'(0)|^2 x]^{\frac{p_*\ka}{2}-\frac{p_*}{4}-\frac{1}{2}}}{x^{\frac{p_*}{4}+\frac{1}{2}}}dx\\
        &\leq \frac{C}{\left[(\ep^2 a'(0)+|\si'(0)|^2 x_1)\right]^{\frac{p_*}{4}+\frac{1}{2}}}\int_0^{x_1}\frac{1}{x^{\frac{p_*}{4}+\frac{1}{2}}}dx
        %&\leq (1-\de)^{-\frac{p_*}{4}-\frac{1}{2}}\left[|\si'(0)|^2x_1\right]^{-\frac{p_*}{4}-\frac{1}{2}} \frac{x_1^{-\frac{p_*}{4}-\frac{1}{2}}}{\frac{1}{2}-\frac{p_*}{4}}\\
\leq\frac{C}{(\frac{1}{2}-\frac{p_*}{4})\left[|\si'(0)|^2\right]^{\frac{p_*}{4}+\frac{1}{2}}x_1^{\frac{p_*}{2}}}, 
    \end{split}
\end{equation*}
where $C=(1-\de)^{-\frac{p_*}{4}-\frac{1}{2}}e^{p_*\int_{x_1}^1\frac{|b|}{\si^2}ds}$, and we used in the second inequality the fact $\frac{p_*\ka}{2}-\frac{p_*}{4}-\frac{1}{2}>0$ so that  
$$
\int_0^{x_1}\frac{[\ep^2 a'(0)+|\si'(0)|^2 x]^{\frac{p_*\ka}{2}-\frac{p_*}{4}-\frac{1}{2}}}{x^{\frac{p_*}{4}+\frac{1}{2}}}dx\leq [\ep^2 a'(0)+|\si'(0)|^2 x_1]^{\frac{p_*\ka}{2}-\frac{p_*}{4}-\frac{1}{2}}\int_0^{x_1}\frac{1}{x^{\frac{p_*}{4}+\frac{1}{2}}}dx.
$$

As a result, $\sup_{\ep}\int_0^{\infty}\frac{e^{-p_{*}V_{\ep}}}{\al_{\ep}^{\frac{p_{*}}{4}+\frac{1}{2}}}dx=\sup_{\ep}[\RN{1}_{\ep}(p_*)+\RN{2}_{\ep}(p_*)]<\infty$, that is, \eqref{eqn-2022-03-24-4} is true.

\medskip

\noindent{\bf Step 4.} We finish the proof. Note that $\tilde{Q}^{\ep}_k:=\tilde{U}_{\ep}U_{\ep}Q_{k}^{\ep}U_{\ep}^{-1} \tilde{U}_{\ep}^{-1}$ is the spectral projection of $\LL_{\ep}^S$ corresponding to $\{-\la_{\ep,j}\}_{j\geq k}$. As  $\tilde{P}^{\ep}_t$ and $\tilde{Q}^{\ep}_k$ are commutative, we apply {\bf Steps 1-2} to deduce for $\tilde{f}\in L^{p_*}((0,y_{\ep,\infty}))$ (where $p_{*}$ is given in {\bf Step 3}) and $t>2$ that
\begin{equation*}
    \begin{split}
        \|\tilde{P}^{\ep}_t\tilde{Q}^{\ep}_k\tilde{f}\|_{\infty}&\leq D_1(\infty) \|\tilde{P}^{\ep}_{t-1}\tilde{Q}^{\ep}_k\tilde{f}\|_{L^2}\\
        &\leq D_1(\infty)  e^{-\la_{\ep,k}(t-2)}\|\tilde{P}^{\ep}_{1}\tilde{f}\|_{L^2}\leq D_1(\infty)D_1(p'_*)  e^{-\la_{\ep,k}(t-2)}\|\tilde{f}\|_{L^{p_*}},
    \end{split}
\end{equation*}
where $p_{*}'$ is the dual exponent of $p_{*}$. This together with {\bf Step 3} yields for $f\in C_b((0,\infty))$ and $t>2$,
\begin{equation*}
    \begin{split}
        |P^{\ep}_t Q^{\ep}_kf|
        &=|U_{\ep}^{-1}\tilde{U}_{\ep}^{-1}\tilde{P}^{\ep}_t\tilde{Q}_{k}^{\ep}\tilde{U}_{\ep}U_{\ep}f|\\
        &=|(\tilde{P}^{\ep}_t\tilde{Q}^{\ep}_k\tilde{U}_{\ep}U_{\ep}f)\circ \xi_{\ep}|\left( u_{\ep}^G \sqrt{\al_{\ep}} \right)^{-\frac{1}{2}}\\
        &\leq \|\tilde{P}^{\ep}_t\tilde{Q}^{\ep}_k\tilde{U}_{\ep}U_{\ep}f\|_{L^{\infty}} \al_{\ep}^{\frac{1}{4}} e^{V_{\ep}}\\
        &\leq D_1(\infty)D_1(p'_*)  e^{-\la_{\ep,k}(t-2)}\|\tilde{U}_{\ep}U_{\ep}f\|_{L^{p_*}}\al_{\ep}^{\frac{1}{4}} e^{V_{\ep}}\\
        &\leq D_1(\infty)D_1(p'_*)D_2  e^{-\la_{\ep,k}(t-2)}\al_{\ep}^{\frac{1}{4}} e^{V_{\ep}}\|f\|_{\infty}.
     \end{split}
\end{equation*}
As $\sup_{\ep} \la_{\ep,k}<\infty$ by Lemma \ref{lem-upper-bound-eigenvalues}, the result follows. 
\end{proof}

Now, we prove Theorem \ref{thm-multiscale}. 

\begin{proof}[Proof of Theorem \ref{thm-multiscale}]
Let $\mu\in \PP((0,\infty))$ be such that $\supp(\mu)\subset\KK$. Recall that $\phi_{\ep,1}$ is the positive eigenfunction of $-\LL_{\ep}$ associated with $\la_{\ep,1}$ and satisfies the normalization $\|\phi_{\ep,1}\|_{L^2(u^G_{\ep})}=1$. 
We apply Lemma \ref{prop-spectral-structure} (6) to find that for $f\in C_b((0,\infty))$ and $t>0$, 
\begin{equation*}\label{eqn-2021-07-27-2}
\begin{split}
\E^{\ep}_{\mu}[f(X^{\ep}_t)\mathbbm{1}_{t<T^{\ep}_0} ]&=e^{-\la_{\ep,1}t}\langle f, \phi_{\ep,1}\rangle_{L^2(u^G_{\ep})}\int_0^{\infty}\phi_{\ep,1}d\mu+\int_0^{\infty}P^{\ep}_tQ^{\ep}_2fd\mu.
\end{split}
\end{equation*}
Recall the density of the QSD $\mu_{\ep}$ is given by  $u_{\ep}=\frac{\phi_{\ep,1}u^G_{\ep}}{\|\phi_{\ep,1}\|_{L^1(u^G_{\ep})}}$. Let $\al_{\ep,1}=\|\phi_{\ep,1}\|_{L^1(u^G_{\ep})}\phi_{\ep,1}$ be as in the statement. Then, 
\begin{equation*}
\begin{split}
\E^{\ep}_{\mu}[f(X^{\ep}_t)\mathbbm{1}_{t<T^{\ep}_0}] &=e^{-\la_{\ep,1}t} \|\phi_{\ep,1}\|_{L^1(u^G_{\ep})}\int_0^{\infty}\phi_{\ep,1}d\mu\int_0^{\infty}f u_{\ep}dx+\int_0^{\infty}P^{\ep}_tQ^{\ep}_2fd\mu\\
&=e^{-\la_{\ep,1}t}\langle \mu,\al_{\ep,1}\rangle\int_0^{\infty}f u_{\ep}dx+\int_0^{\infty}P^{\ep}_tQ^{\ep}_2fd\mu.
\end{split}
\end{equation*}
Setting $f \equiv 1$ yields $\P^{\ep}_{\mu}[t<T^{\ep}_0]=e^{-\la_{\ep,1}t}\langle \mu,\al_{\ep,1}\rangle+\int_0^{\infty}P^{\ep}_tQ^{\ep}_2\mathbbm{1}d\mu$.
Hence, 
\begin{equation*}
    \begin{split}
        \E^{\ep}_{\mu}[f(X^{\ep}_t)]&=\E^{\ep}_{\mu}[f(X^{\ep}_t)\mathbbm{1}_{t<T^{\ep}_0}]+f(0)\left(1-\P^{\ep}[t<T^{\ep}_0]\right)\\
        &=e^{-\la_{\ep,1}t}\langle \mu,\al_{\ep,1}\rangle \int_0^{\infty}f u_{\ep}dx+\left(1-e^{-\la_{\ep,1}t }\langle \mu,\al_{\ep,1}\rangle\right) f(0)\\
        &\quad+\int_0^{\infty}P^{\ep}_tQ^{\ep}_2(f-f(0))d\mu.
    \end{split}
\end{equation*}
It follows from Lemma \ref{lem-semigroup-property} that there is $C>0$ such that for $t>2$,
\begin{equation*}
    \begin{split}
        & \left| \E^{\ep}_{\mu}[f(X^{\ep}_t)]-\left[e^{-\la_{\ep,1}t}\langle \mu,\al_{\ep,1}\rangle \int_0^{\infty}f u_{\ep}dx+\left(1-e^{-\la_{\ep,1}t }\langle \mu,\al_{\ep,1}\rangle\right) f(0)\right] \right| 
        \\ 
        &\qquad\leq C\|f\|_{\infty}e^{-\la_{\ep,2}t}
        \int_0^{\infty}\al_{\ep}^{\frac{1}{4}}e^{V_{\ep}} d\mu\leq \tilde{C} e^{-\la_{\ep,2}t}\|f\|_{L^{\infty}},
    \end{split}
\end{equation*}
where $\tilde{C}=\tilde{C}(\KK):=1+\sup_{\KK}\sqrt{|\si|}e^{\int_1^{\bullet}\frac{b}{\si^2}ds}$. As a result, we find 
\begin{equation*}
    \begin{split}
       & \left\|\P^{\ep}_{\mu}[X^{\ep}_t\in \bullet]-\left[e^{-\la_{\ep,1}t }\langle \mu,\al_{\ep,1}\rangle \mu_{\ep}+\left(1-e^{-\la_{\ep,1}t }\langle \mu,\al_{\ep,1}\rangle\right)\de_0\right]\right\|_{TV}\\
       &
         \qquad =\sup_{\substack{f\in C_b([0,\infty))\\\|f\|_{\infty}\leq 1}}\left| \E^{\ep}_{\mu}[f(X^{\ep}_t)]-\left[e^{-\la_{\ep,1}t}\langle \mu,\al_{\ep,1}\rangle \int_0^{\infty}f u_{\ep}dx+\left(1-e^{-\la_{\ep,1}t }\langle \mu,\al_{\ep,1}\rangle\right) f(0)\right] \right| 
        \\
        &
        \qquad =\sup_{\substack{f\in C_b([0,\infty))\\ 
        \|f\|_{\infty}\leq 1}} \left|\int_0^{\infty}P^{\ep}_tQ^{\ep}_2(f-f(0))d\mu \right|
        \\
        &\qquad \leq \tilde{C} e^{-\la_{\ep,2}t}, \quad \forall t>2.
    \end{split}
\end{equation*}
Note that if we establish the limit
\begin{equation}\label{coefficient-limit-2022-04-09}
    \lim_{\ep\to0}\al_{\ep}=1\quad\text{locally uniformly in}\,\,(0,\infty),
\end{equation}
then the conclusion of the theorem follows for $t>2$. Making $\tilde{C}$ larger if necessary, the conclusion holds for all $t\geq0$. Thus, it remains to show \eqref{coefficient-limit-2022-04-09}.

To do so, we let $\tilde{\phi}_{\ep,1}$ be the positive eigenfunction of $-\LL_{\ep}$ associated with $\la_{\ep,1}$ and satisfy the normalization \eqref{normalization-2022-03-31}, namely, $\|\tilde{\phi}_{\ep,1}\|_{L^1((0,2);u^G_{\ep})}+\|\tilde{\phi}_{\ep,1}\|_{L^2((1,\infty);u^G_{\ep})}=1$. Since $\tilde{\phi}_{\ep,1}$ is proportional to $\phi_{\ep,1}$, there holds $\al_{\ep,1}=\frac{\|\tilde{\phi}_{\ep,1}\|_{L^1(u^G_{\ep})}}{\|\tilde{\phi}_{\ep,1}\|^2_{L^2(u^G_{\ep})}}\tilde{\phi}_{\ep,1}$. As {\bf Step 2} in the proof of Theorem \ref{thm-lower-bound-lamda_2} says 
\begin{equation}\label{a-convergent-result-2022-04-09}
\lim_{\ep\to0}\tilde{\phi}_{\ep,1}=C_{1}\quad\text{locally uniformly in}\,\, (0,\infty)
\end{equation}
for some constant $C_1>0$, \eqref{coefficient-limit-2022-04-09} follows if we can show 
\begin{equation}\label{eqn-2021-09-14-1}
    \lim_{\ep\to 0}\|\tilde{\phi}_{\ep,1}\|_{L^1(u^G_{\ep})}=C_1\int_0^{\infty}u_0^G dx\quad\andd \quad \lim_{\ep\to 0}\|\tilde{\phi}_{\ep,1}\|^2_{L^2(u^G_{\ep})}=C^2_1\int_0^{\infty}u_0^G dx.
\end{equation}

For any $\ka>1$, we split 
\begin{equation}\label{eqn-2022-07-13-1}
    \begin{split}
    &\int_0^{\infty}\tilde{\phi}_{\ep,1}u^G_{\ep}dx- C_1\int_0^{\infty} u_0^G dx\\
    &\quad=\int_{\frac{1}{\ka}}^{\ka}\tilde{\phi}_{\ep,1}u^G_{\ep}dx-C_1\int_{\frac{1}{\ka}}^{\ka}u^G_{0}dx+\left(\int_{0}^{\frac{1}{\ka}}+\int_{\ka}^{\infty}\right)\tilde{\phi}_{\ep,1}u^G_{\ep}dx+C_1\left(\int_{0}^{\frac{1}{\ka}}+\int_{\ka}^{\infty}\right)u^G_{0}dx
    \end{split}
\end{equation}
and 
\begin{equation}\label{eqn-2022-07-13-2}
    \begin{split}
        &\int_{0}^{\infty}\tilde{\phi}_{\ep,1}^2u^G_{\ep}dx-C_1^2\int_0^{\infty}u^G_{0}dx\\
        &\quad =\int_{\frac{1}{\ka}}^{\ka}\tilde{\phi}_{\ep,1}^2u^G_{\ep}dx-C_1^2\int_{\frac{1}{\ka}}^{\ka}u^G_{0}dx+\left(\int_0^{\frac{1}{\ka}}+\int_{\ka}^{\infty}\right)\tilde{\phi}_{\ep,1}^2 u^G_{\ep}dx+C_1^2\left(\int_0^{\frac{1}{\ka}}+\int_{\ka}^{\infty}\right) u^G_{0}dx.
    \end{split}
\end{equation}

By \eqref{gibbs-to-gibbs} and \eqref{a-convergent-result-2022-04-09}, we see that 
\begin{equation*}\label{eqn-2022-07-13-3}
    \lim_{\ep\to 0}\left|\int_{\frac{1}{\ka}}^{\ka}\tilde{\phi}_{\ep,1}u^G_{\ep}dx-C_1\int_{\frac{1}{\ka}}^{\ka} u^G_{0}dx\right|=0,\,\,\lim_{\ep\to 0}\left|\int_{\frac{1}{\ka}}^{\ka}\tilde{\phi}_{\ep,1}^2u^G_{\ep}dx-C_1^2\int_{\frac{1}{\ka}}^{\ka} u^G_{0}dx\right|=0,\,\,\forall \ka>1.
\end{equation*}

Lemmas \ref{lem-compactness-2-infty} and \ref{lem-compactness-2-0} yield
\begin{equation}\label{eqn-2021-09-14-2}
    \lim_{\ka\to \infty}\sup_{\ep}\int_0^{\frac{1}{\ka}}\tilde{\phi}_{\ep,1}u^G_{\ep}dx=0,\quad\lim_{\ka\to \infty}\sup_{\ep}\int_{\ka}^{\infty}\tilde{\phi}_{\ep,1}^2u^G_{\ep}dx=0.
\end{equation}
This together with Lemma \ref{lem-integral-gibb} and H\"older's inequality yields
$\lim_{\ka\to \infty}\sup_{\ep}\int_{\ka}^{\infty}\tilde{\phi}_{\ep,1} u^G_{\ep}dx=0$. Furthermore, we see from \eqref{uniform-bound-2022-04-09} and \eqref{eqn-2021-09-14-2} that $\lim_{\ka\to 0}\sup_{\ep}\int_0^{\frac{1}{\ka}}\tilde{\phi}_{\ep,1}^2 u^G_{\ep}dx=0$.

Given these limits, \eqref{eqn-2021-09-14-1} follows immediately from taking the limit $\ep\to 0$ and then $\ka\to \infty$ in \eqref{eqn-2022-07-13-1} and \eqref{eqn-2022-07-13-2}. This completes the proof. 
\end{proof}

\section{\bf Asymptotic bounds of the mean extinction time} \label{s:ext}

In this section, we adopt probabilistic methods to study the asymptotic of the mean extinction time $\E^{\ep}_x[T^{\ep}_0]$. In particular, we prove Theorem \ref{thm-asymptotic-mean-extinction}.

% More precisely, we prove 
% \begin{thm}\label{thm-asymptotic-mean-extinction2}
% For each $x>0$, there hold
% \begin{enumerate}
%   \item If $\La_0<0$, then there exists $C_1>0$ and $C_2>0$ such that
%   $$
%   C_1|\ln \ep|\leq  \E^{\ep}_{x}[T^{\ep}_{0}]\leq C_2|\ln \ep|^2. 
%   $$
   
%   \item If $\La_0>0$, then for any $0<\ga\ll 1$, 
%   $$
%   \ep^{2-(1+\ga)\frac{4b'(0)}{|\si'(0)|^2}}\E^{\ep}_{x}[T^{\ep}_0]\leq \ep^{2-(1-\ga)\frac{4b'(0)}{|\si'(0)|^2}}.
%   $$
% \end{enumerate}
% \end{thm}

We begin with the introduction of some notations that are used frequently in the rest of this section. For $0<\de\ll 1$, {\bf (H)} (1)-(3) ensures the existence of $\be=\be(\de)\in (0,1)$ such that 
\begin{equation}\label{eqn-coefficients-approx}
%1-\ga\leq \frac{b(x)(\ep^2a'(0)+|\si'(0)|^2 x)}{\al_{\ep}(x)b'(0)}\leq 1+\ga. 
1-\de\leq \frac{b(x)}{b'(0)x}, \frac{\al_{\ep}(x)}{x[\ep^2a'(0) +|\si'(0)|^2x]}\leq 1+\de,\quad \forall x\in (0,\be )\andd 0<\ep\ll 1. 
\end{equation}
Set 
\begin{equation}\label{eqn-def-ratios}
    \ka_{-}=\ka_{-}(\de):=\frac{2(1-\de)b'(0)}{(1+\de)|\si'(0)|^2}\quad\andd\quad  \ka_{+}=\ka_+(\de):=\frac{2(1+\de)b'(0)}{(1-\de)|\si'(0)|^2}.
\end{equation}
Note that $\ka_{-}<\ka_{+}<1$ when $\Lambda_{0}<0$, and $\ka_{+}>\ka_{-}>1$ when $\Lambda_{0}>0$.

Fix $x_*=x_*(\de)\in (0,\be)$. Denote by $\tau^{\ep}=\tau^{\ep}(\de)$ the first time  $X^{\ep}_t$ exits from $(0,\be)$, namely, $\tau^{\ep}:=\inf\{t\geq 0: X^{\ep}_t=0\text{ or }\be\}$, and by $\tau^{\ep}_{x_*}=\tau^{\ep}_{x_*}(\de)$ the first time $X^{\ep}_t$ hits $x_*$, namely, $\tau^{\ep}_{x_*}(\de):=\inf\{t\geq 0: X^{\ep}_t=x_*\}$.

For each $0<\ep\ll1$ and $x\in(0,\infty)$, we define 
\begin{equation*}
\begin{split}
     s_{\ep}(x)=s_{\ep}(x,\de):&=\int_{x_*}^x e^{-2\int_{x_*}^y \frac{b}{\al_{\ep}}ds}dy,\\ 
     r_{\ep}(x)=r_{\ep}(x,\de):&=\int_{x_*}^x e^{-2\int_{x_*}^y \frac{b}{\al_{\ep}}ds} \int_{x_*}^y \frac{1}{\al_{\ep}(z)}e^{2\int_{x_*}^z \frac{b}{\al_{\ep}}ds}dz dy.
\end{split}
\end{equation*}
In literature (see e.g. \cite{Kallenberg02}), $s_{\ep}$ is referred to as the scale function. The function $r_{\ep}$ arises naturally in the study of the mean exit time $\E_{\bullet}^{\ep}[\tau^{\ep}]$ (see \cite{IW81,Kallenberg02} or the proof of Lemma \ref{lem-mean-exit-time-bdd-interval}). It is easy to check that $s_{\ep}(0+)\in(-\infty,0)$ and $r_{\ep}(0+)\in(0,\infty)$. 

Replacing $\al_{\ep}$ by $\si^{2}$ in the definition of $s_{\ep}$ and $r_{\ep}$, we define $s_{0}$ and $r_{0}$. It is straightforward to check that $s_{0}(0+)\in(-\infty,0)$ when $\La_0<0$, and $s_0(0+)=-\infty$ when $\La_0>0$. Moreover, $r_0(0+)=\infty$.

We establish three lemmas before proving Theorem \ref{thm-asymptotic-mean-extinction}. The first one concerns the asymptotic of $s_{\ep}(0+)$ and $\P^{\ep}_{x}[X^{\ep}_{\tau^{\ep}}=\be]$ for $x\in(0,\beta)$.  

\begin{lem}\label{lem-asymptotic-scale-function}
Assume {\bf(H)}. Then, $\lim_{\ep\to 0}s_{\ep}=s_0$. Moreover, 
\begin{enumerate}
    \item[\rm(1)] if $\La_0<0$, then $\lim_{\ep\to 0}s_{\ep}(0+)=s_0(0+)>-\infty$ and
    $$
\lim_{\ep\to 0}\P^{\ep}_x[X^{\ep}_{\tau^{\ep}}=\be]=\frac{s_0(x)-s_0(0+)}{s_0(\be)-s_0(0+)}\in (0,1),\quad \forall x\in (0,\be);
    $$
    
    \item[\rm(2)] if $\La_0>0$, then there are $C_1, C_2>0$ (depending on $\de$) such that 
    $$
    C_1 \ep^{-2(\ka_{-}-1)}\lesssim_{\ep} -s_{\ep}(0+)\lesssim_{\ep} C_2 \ep^{-2(\ka_+-1)},
    $$
    and  for each $x\in (0,\be)$, there are $C_3,C_4>0$ such that
    $$
    1-C_3 \ep^{2(\ka_{-}-1)}\lesssim_{\ep} \P^{\ep}_x[X^{\ep}_{\tau^{\ep}}=\be]\lesssim_{\ep} 1-C_4\ep^{2(\ka_{+}-1)}.
    $$
\end{enumerate}
\end{lem}
\begin{proof}
Since $\al_{\ep}\downarrow\si^2$ on $(0,\be)$ as $\ep\to0$, we apply the monotone convergence theorem to find $\lim_{\ep\to 0}s_{\ep}=s_0$ and $\lim_{\ep\to 0}s_{\ep}(0+)=s_0(0+)$.

It is well known (see e.g. \cite[Theorem 6.3.1]{IW81}) that 
\begin{equation}\label{eqn-2022-05-24-1}
    \P^{\ep}_{x}[X^{\ep}_{\tau^{\ep}}=\be]=\frac{s_{\ep}(x)-s_{\ep}(0+)}{s_{\ep}(\be)-s_{\ep}(0+)},\quad\forall x\in (0,\be).
\end{equation}

\medskip

(1)  It is easy to see that $-s_0(0+)<\infty$. The limiting equality follows by letting  $\ep\to0$ in \eqref{eqn-2022-05-24-1}.

% Thanks to \eqref{eqn-coefficients-approx}, we deduce
% \begin{equation*}
%     \begin{split}
%         -s_0(0+)=\int_0^{x_*} e^{2\int_y^{x_*}\frac{b}{\si^2}ds}dy\leq \int_0^{x_*}e^{\frac{1+\de}{1-\de}\int_y^{x_*}\frac{2b'(0)}{|\si'(0)|^2 s}ds}=\int_0^{x_*}\left(\frac{x_*}{y}\right)^{\ka_{+}}dy.
%     \end{split}
% \end{equation*}
% As $\ka_+<1$, there holds $-s_0(0+)<\infty$. 
%Letting $\ep\to0$ in \eqref{eqn-2022-05-24-1} yields $\lim_{\ep\to0}\P^{\ep}_x[X^{\ep}_{\tau_{\ep}}=\be]=\frac{s_0(x)-s_0(0+)}{s_0(\be)-s_0(0+)}\in (0,1)$. 

\medskip

(2) Using \eqref{eqn-coefficients-approx}, we find  
\begin{equation*}\label{eqn-2022-05-15-1}
    \begin{split}
        -s_{\ep}(0+)\leq \int_0^{x_*} e^{\frac{1+\de}{1-\de}\int_y^{x_*}\frac{2b'(0)}{\ep^2 a'(0)+|\si'(0)|^2 s}ds}dy
        %&=\int_0^{x_*}e^{\ka_{+}\ln \frac{\ep^2 a'(0)+|\si'(0)|^2 x_*}{\ep^2 a'(0)+|\si'(0)|^2 y}}dy\\
        =\int_0^{x_*} \left[\frac{\ep^2 a'(0)+|\si'(0)|^2 x_*}{\ep^2 a'(0)+|\si'(0)|^2 y}\right]^{\ka_+} dy.
    \end{split}
\end{equation*}
Note that $\ka_+>\ka_{-}>1$ in this case. Calculating the last integral leads to 
\begin{equation*}
    \begin{split}
      -s_{\ep}(0+)&\leq  2[|\si'(0)|^2x_*]^{\ka_{+}}\frac{1}{(-\ka_{+}+1)|\si'(0)|^2}\left[\ep^2a'(0)+|\si'(0)|^2 y\right]^{-\ka_{+}+1}\Big|_{y=0}^{x_*}\\
      &\leq\frac{2x_*^{\ka_{+}} |\si'(0)|^{2(\ka_{+}-1)}}{(\ka_{+}-1)[\ep^2 a'(0)]^{\ka_{+}-1}}\left\{1-\frac{[\ep^2 a'(0)]^{\ka_{+}-1}}{[\ep^2 a'(0)+|\si'(0)|^2x_*]^{\ka_{+}-1}}\right\}\\
      &\approx_{\ep} \frac{2x_*^{\ka_{+}} |\si'(0)|^{2(\ka_{+}-1)}}{(\ka_{+}-1)[a'(0)]^{\ka_{+}-1}}\ep^{-2(\ka_{+}-1)}=:C_1\ep^{-2(\ka_+-1)},
    \end{split}
\end{equation*}
which together with  \eqref{eqn-2022-05-24-1} leads to
\begin{equation*}
    \begin{split}
        \P^{\ep}_x[X^{\ep}_{\tau}=\be]&=1-\frac{s_{\ep}(\be)-s_{\ep}(x)}{s_{\ep}(\be)-s_{\ep}(0+)}\\
        &\approx_{\ep} 1-\frac{s_0(\be)-s_0(x)}{s_0(\be)-s_{\ep}(0+)}\\
        &\lesssim_{\ep}1-[s_0(\be)-s_0(x)]C_1^{-1}\ep^{2(\ka_{+}-1)}=:1-C_2\ep^{2(\ka_{+}-1)}.
    \end{split}
\end{equation*}

Similarly, there exist $C_3,C_4>0$ such that $-s_{\ep}(0)\gtrsim_{\ep} C_3 \ep^{-2(\ka_{-}-1)}$ and $\P^{\ep}_x[X^{\ep}_{\tau}=\be]\gtrsim_{\ep} 1-C_4 \ep^{2(\ka_{-}-1)}$. This proves (2). 
\end{proof}

%The following lemma concerns about the asymptotics of $\P^{\ep}_{x_*}[X^{\ep}_{\tau^{\ep}}=0]$ and $\E^{\ep}_{x_*}[\tau^{\ep}]$ as $\ep\to 0$.

In the second lemma, we study the asymptotic bounds of the mean exit time $\E^{\ep}_{x_*}[\tau^{\ep}]$.

\begin{lem}\label{lem-mean-exit-time-bdd-interval}
Assume {\bf(H)}. 
\begin{enumerate}
   \item If $\La_0<0$, then there are $C_1,C_2>0$ (depending on $\de$) such that
   $$
  % \lim_{\ep\to 0}\P^{\ep}_{x_*}[X^{\ep}_{\tau^{\ep}}=0]=\frac{s_0(\be)-s_0(x_*)}{s(\be)-s(0^+)}\in (0,1)\quad \andd\quad 
  C_1 |\ln \ep|\lesssim_{\ep} \E^{\ep}_{x_*}[\tau^{\ep}]\lesssim_{\ep} C_2|\ln \ep|. 
   $$
   
   \item If $\La_0>0$, then $\inf_{\ep} \E^{\ep}_{x_*}[\tau^{\ep}]>0$ and $\E^{\ep}_{x_*}[\tau^{\ep}]\lesssim_{\ep}\ep^{-2(\ka_{+}-\ka_{-}+\de)}$.
\end{enumerate}
\end{lem}
\begin{proof}
We first show that
\begin{equation}\label{eqn-formula-mean-extinction-time-atx*}
    \E^{\ep}_{x_*}[\tau^{\ep}]=\frac{2\left[r_{\ep}(0+)s_{\ep}(\be)-r_{\ep}(\be)s_{\ep}(0+)\right]}{s_{\ep}(\be)-s_{\ep}(0+)}.
\end{equation}

It is well known that $u_{\ep}:=\E^{\ep}_{\bullet}[\tau^{\ep}]$ solves
\begin{equation*}\label{eqn-mean-exit-time-bdd-interval}
\begin{cases}
\frac{1}{2}\al_{\ep} u''_{\ep}+bu'_{\ep}=-1\quad \text{in}\quad (0,\be),\\
u_{\ep}(0)=0=u_{\ep}(\be).
\end{cases}
\end{equation*}
% Dividing the equation in \eqref{eqn-mean-exit-time-bdd-interval} by $\frac{1}{2}\al_{\ep}$, we arrive at 
% $$
% u''_{\ep}+\frac{2b}{\al_{\ep}}u'_{\ep}=-\frac{2}{\al_{\ep}}\quad \text{in}\quad (0,\be).
% $$
% Multiplying the above equation by $e^{2\int_{x_*}^x \frac{b}{\al_{\ep}}ds}$ and integrating the resulting equation, we deduce 
% $$
% e^{2\int_{x_*}^x \frac{b}{\al_{\ep}}ds} u'_{\ep}(x)=-2\int_{x_*}^x \frac{1}{\al_{\ep}} e^{2\int_{x_*}^y \frac{b}{\al_{\ep}}ds} dy +C_1(\ep),\quad \forall x\in [0,\be],
% $$
% for some constant $C_{1}(\ep)$. Hence, there exists $C_2(\ep)\in\R$ so that 
% {\bl
% The general solution of the equation is given by
% \begin{equation*}
% \begin{split}
% u_{\ep}(x)&=-2\int_{x_*}^x e^{-2\int_{x_*}^y \frac{b}{\al_{\ep}}ds} \int_{x_*}^y \frac{1}{\al_{\ep}} e^{2\int_{x_*}^z \frac{b}{\al_{\ep}}ds} dz+C_1(\ep) \int_{x_*}^x e^{-2\int_{x_*}^y \frac{b}{\al_{\ep}}ds}dy +C_2(\ep)\\
% &=-2r_{\ep}(x)+C_1(\ep) s_{\ep}(x)+C_2(\ep),
% \end{split}
% \end{equation*}
% where $C_{1}(\ep)$ and $C_{2}(\ep)$ are constants. We then deduce from the boundary conditions that
% $$
% C_1(\ep)=\frac{2(r_{\ep}(\be)-r_{\ep}(0+))}{s_{\ep}(\be)-s_{\ep}(0+)} \quad \andd \quad C_2(\ep)=\frac{2}{s_{\ep}(\be)-s_{\ep}(0+)}\left[r_{\ep}(0)s_{\ep}(\be)-r_{\ep}(\be)s_{\ep}(0+)\right],
% $$
% leading to
% }
% {\rd (Delete?)}
Direct calculations yield
\begin{equation*}
\begin{split}
u_{\ep}(x)%&=-2r_{\ep}(x)+\frac{2(r_{\ep}(\be)-r_{\ep}(0+))}{s_{\ep}(\be)-s_{\ep}(0+)} s_{\ep}(x)+\frac{2}{s_{\ep}(\be)-s_{\ep}(0+)}\left[r_{\ep}(0+)s(\be)-r_{\ep}(\be)s_{\ep}(0+)\right]\\
&=-2r_{\ep}(x)+\frac{2(s_{\ep}(x)-s_{\ep}(0+))}{s_{\ep}(\be)-s_{\ep}(0+)} r_{\ep}(\be)+\frac{2(s_{\ep}(\be)-s_{\ep}(x))}{s_{\ep}(\be)-s_{\ep}(0+)} r_{\ep}(0+).
\end{split}
\end{equation*}
Setting $x=x_{*}$, we derive \eqref{eqn-formula-mean-extinction-time-atx*} from $s_{\ep}(x_{*})=0$ and $r_{\ep}(x_{*})=0$.

\medskip

(1) Note $\lim_{\ep\to 0}r_{\ep}(\be)=r_0(\be)$. Since $\lim_{\ep\to0}s_{\ep}(\beta)=s_{0}(\beta)$ and $\lim_{\ep\to0}s_{\ep}(0+)=s_{0}(0+)$ by Lemma \ref{lem-asymptotic-scale-function}, we find from \eqref{eqn-formula-mean-extinction-time-atx*} that
\begin{equation}\label{eqn-2022-05-13-1}
\begin{split}
\E^{\ep}_{x_*}[\tau^{\ep}]%&=\frac{2}{s_{\ep}(\be)-s_{\ep}(0)}\left[r_{\ep}(0)s_{\ep}(\be)-r_{\ep}(\be)s_{\ep}(0)\right]\\
&\approx_{\ep} \frac{2s_0(\be)r_{\ep}(0+)}{s_0(\be)-s_0(0+)}-\frac{2r_0(\be)s_0(0+)}{s_0(\be)-s_0(0+)}.
\end{split}
\end{equation}
%the asymptotic of $\E^{\ep}_{x_*}[\tau^{\ep}]$ as $\ep\to 0$ is determined by that of $r_{\ep}(0^+)$. 

If there are $C_1,C_2>0$ (depending on $\de$) such that  \begin{equation}\label{eqn-kappa-bounds}
C_1|\ln \ep|\lesssim_{\ep} r_{\ep}(0+)\lesssim_{\ep} C_2|\ln \ep|,
\end{equation}
we deduce from \eqref{eqn-2022-05-13-1} that
$\frac{2C_1 s_0(\be)}{s_0(\be)-s_0(0+)}|\ln \ep|\lesssim_{\ep} \E^{\ep}_{x_*}[\tau^{\ep}]\lesssim_{\ep} \frac{2C_2 s_0(\be)}{s_0(\be)-s_0(0+)}|\ln \ep|$, leading to the conclusion.

It remains to show \eqref{eqn-kappa-bounds}. Thanks to \eqref{eqn-coefficients-approx}, we compute
\begin{equation*}
\begin{split}
r_{\ep}(0+)&=\int_0^{x_*} \int_y^{x_*} \frac{1}{\al_{\ep}(z)} e^{2\int_y^z \frac{b}{\al_{\ep}}ds}dzdy\\
&\geq \int_0^{x_*}\int_y^{x_*}\frac{1}{(1+\de)z[\ep^2a'(0)+|\si'(0)|^2 z]} e^{\frac{1-\de}{1+\de}\int_y^z \frac{2b'(0)}{\ep^2a'(0)+|\si'(0)|^2 s}ds}dzdy\\
&=\frac{1}{1+\de}\int_0^{x_*}\int_y^{x_*}\frac{1}{z} [\ep^2a'(0)+|\si'(0)|z]^{\ka_{-}-1} [\ep^2a'(0)+|\si'(0)|y]^{-\ka_{-}}dzdy\\
&\geq \frac{1}{1+\de} \int_0^{\frac{x_*}{2}} \int_{y}^{2y}[\ep^2a'(0)+|\si'(0)|y]^{-1} \frac{1}{z}\left[\frac{\ep^2a'(0)+|\si'(0)|z}{\ep^2a'(0)+|\si'(0)|y}\right]^{\ka_{-}-1}dz dy.
\end{split}
\end{equation*}
Noting that 
$$
\frac{\ep^2a'(0)+|\si'(0)|z}{\ep^2a'(0)+|\si'(0)|y}\leq \frac{2\ep^2a'(0)+|\si'(0)|2y}{\ep^2a'(0)+|\si'(0)|y}=2,\quad\forall y\in \left(0,\frac{x_*}{2}\right)\andd z\in (y,2y),
$$
we deduce from the fact $\ka_{-}<1$ that
\begin{equation*}\label{eqn-kappa-lower}
\begin{split}
r_{\ep}(0+)&\geq \frac{1}{1+\de}\int_0^{\frac{x_*}{2}} [\ep^2a'(0)+|\si'(0)|y]^{-1}\int_y^{2y} \frac{1}{z}2^{\ka_{-}-1}dz dy\\
%&=\frac{2^{\ka_{-}-1}\ln 2}{1+\de}\int_0^{\frac{x_*}{2}} [\ep^2a'(0)+|\si'(0)|y]^{-1} dy\\
&
=\frac{2^{\ka_{-}-1}\ln 2}{(1+\de)|\si'(0)|^2}\ln \frac{\ep^2 a'(0)+|\si'(0)|^2 \frac{x_*}{2}}{\ep^2 a'(0)}
\\
&=\frac{2^{\ka_{-}-1}\ln 2}{(1+\de)|\si'(0)|^2}\ln \left(1+\frac{|\si'(0)|^2 x_*}{2\ep^2 a'(0)}\right)\approx_{\ep} C_1|\ln \ep|,
\end{split}
\end{equation*} 
where $C_1:= \frac{2^{\ka_{-}}\ln 2}{(1+\de)|\si'(0)|^2}$, and the equality follows from direct calculations of the double integral.

To derive an upper bound, we change the order of integration to rewrite $r_{\ep}(0+)$ as
\begin{equation}\label{eqn-2022-05-16-1-1}
    r_{\ep}(0+)=\int_0^{x_*}\int_0^{z}\frac{1}{\al_{\ep}(z)} e^{2(V_{\ep}(y)-V_{\ep}(z))}dydz,
\end{equation}
which is just $\RN{1}_{\ep}$ in \eqref{eqn-2022-02-09-5-1}. By \eqref{eqn-2022-03-16-4}, $r_{\ep}(0+)\lesssim_{\ep} C_2 |\ln \ep|$ for some $C_2>0$. Hence, \eqref{eqn-kappa-bounds} follows.

\medskip

(2) Let $\hat{x}\in (0,x_*)$. Obviously,  $\E^{\ep}_{x_*}[\tau^{\ep}]\geq \E^{\ep}_{x_*}[\hat{\tau}^{\ep}]$, where $\hat{\tau}^{\ep}:=\inf\{t\geq 0: X^{\ep}_t=\hat{x}\text{ or }\be\}$. Note that $\E^{\ep}_{\bullet}[\hat{\tau}^{\ep}]$ solves 
\begin{equation*}
\begin{cases}
\frac{1}{2}\al_{\ep} u''+bu'=-1\quad \text{in}\quad (\hat{x},\be),\\
u(\hat{x})=0=u(\be).
\end{cases}
\end{equation*}
As $\lim_{\ep\to 0}\al_{\ep}=\si^2$ uniformly in $[\hat{x},\beta]$, the classical PDE theory ensures that $\lim_{\ep\to0}\E^{\ep}_{\bullet}[\hat{\tau}^{\ep}]=u_{0}$ uniformly in $[\hat{x},\beta]$, where $u_{0}$ is the unique solution of 
\begin{equation*}
\begin{cases}
\frac{1}{2}\si^2 u_{0}''+bu_{0}'=-1\quad \text{in}\quad (\hat{x},\be),\\
u_{0}(\hat{x})=0=u_{0}(\be).
\end{cases}
\end{equation*}
Since $u_{0}(x_{*})>0$ by the maximum principle, we conclude $\inf_{\ep} \E^{\ep}_{x_*}[\tau^{\ep}]>0$.

It remains to derive the upper bound for $\E^{\ep}_{x_*}[\tau^{\ep}]$.  Note that $\ka_{+}>\ka_{-}>1$. Using \eqref{eqn-2022-05-16-1-1},
% namely, 
% \begin{equation*}
%     \begin{split}
%         r_{\ep}(0+)%&\leq \frac{1}{1-\ga}\int_0^{x_*}\int_y^{x_*}\frac{1}{z} [\ep^2a'(0)+|\si'(0)|^2z]^{\ka_{+}-1} [\ep^2a'(0)+|\si'(0)|^2y]^{-\ka_{+}}dzdy\\
%         &\leq\frac{1}{1-\de}\int_0^{x_*}\int_0^{z}\frac{1}{z} [\ep^2a'(0)+|\si'(0)|^2z]^{\ka_{+}-1} [\ep^2a'(0)+|\si'(0)|^2y]^{-\ka_{+}}dydz,
%     \end{split}
% \end{equation*}
we apply \eqref{eqn-2022-02-09-6} to find $r_{\ep}(0+)\lesssim_{\ep} \ep^{-2(\ka_{+}-1+\de)}$. Since $-s_{\ep}(0+)\gtrsim_{\ep} C_1 \ep^{-2(\ka_{-}-1)}$ due to Lemma \ref{lem-asymptotic-scale-function}, we deduce from \eqref{eqn-formula-mean-extinction-time-atx*} that 
$$
\E^{\ep}_{x_*}[\tau^{\ep}]\approx_{\ep} \frac{2s_0(\be)r_{\ep}(0+)}{s_0(\be)-s_{\ep}(0+)}+2r_0(\be)\lesssim_{\ep} \frac{2s_0(\be)\ep^{-2(\ka_{+}-1+\de)}}{C_1\ep^{-2(\ka_{-}-1)}}=\frac{2s_0(\be)}{C_1}\ep^{-2(\ka_{+}-\ka_{-}+\de)}.
$$
This completes the proof. 
% For notational simplicity, we denote $u_{\ep}:=\E^{\ep}_{\bullet}[\tau^{\ep}]$. Since $u_{\ep}(x)>0$ for $x\in (0,\be)$, it follows that $u'_{\ep}(\be)\leq 0$. Noting that \eqref{eqn-coefficients-approx} ensures $b(\be)>0$, we set $w(x):=\frac{2}{b(\be)}(\be-x)$ for $x\in (0,\be]$. We compute
% $$
% \LL_{\ep}w=-\frac{2b}{b(\be)}\quad \text{in}\quad (0,\be).
% $$
% The continuity of $b$ yields  $\be_1\in (0,\be)$ such that $\LL_{\ep}w\leq -\frac{3}{2}$ in $(\be_1, \be)$. Thus, 
% $$
% \LL_{\ep}u_{\ep}\geq \LL_{\ep}w\quad \text{in}\quad (\be_1,\be)\quad \andd \quad u_{\ep}(\be)=w(\be)=0.
% $$
% Applying the Hopf lemma, we find $u'_{\ep}(\be)<w'(\be)$ and thus, there exists $\be_2\in (\be_1,\be)$ such that 
% $$
% u_{\ep}(\be_2)<w(\be_2),\quad \forall 0<\ep\ll 1.
% $$
% Since $\lim_{\ep\to 0}\al_{\ep}=\si^2$ and $\si^2>0$ in $(0,\be)$, we apply Harnack's inequality to $\{u_{\ep}\}_{\ep}$ and find $C>0$ such that 
% $$
% \sup_{\ep}u_{\ep}(x_*)\leq \sup_{\ep} C u_{\ep}(\be_2)<\infty.
% $$
% This shows (2). 
\end{proof}

The third lemma addresses the uniform-in-$\ep$ finiteness of the mean hitting time $\E^{\ep}_{\bullet}[\tau^{\ep}_{x_*}]$.

\begin{lem}\label{lem-mean-hitting-time-x*}
Assume {\bf(H)}. Then, $\sup_{\ep}\E^{\ep}_{x}[\tau^{\ep}_{x_*}]<\infty$ for each $x>x_*$.
\end{lem}
\begin{proof}
Fix $x>x_{*}$. As in the proof of Proposition \ref{prop-concentration-infty}, we can find a function $V\in C^2(0,\infty)$ and a number $N_0\in (x,\infty)$ such that $V(N_0)>0$ and  $\LL_{\ep} V\leq -\frac{b^2}{2\si^2}$ in $(N_0,\infty)$. Since $\lim_{x\to\infty}\frac{b}{|\si|}=-\infty$ by {\bf (H4)}, we may assume $\LL_{\ep}V\leq -1$ in $(N_0,\infty)$. Set $\tau^{\ep}_{N_0}:=\inf\{t\geq 0: X^{\ep}_t=N_0\}$. An application of It\^o-Dynkin's formula yields $0\leq\E^{\ep}_{N_0+1}[V(\tau^{\ep}_{N_0})]\leq V(N_0+1)-\E^{\ep}_{N_0+1}[\tau^{\ep}_{N_0}]$,
% \begin{equation*}
%     0\leq\E^{\ep}_{N_0+1}[V(\tau^{\ep}_{N_0})]=V(N_0+1)+\E^{\ep}_{N_0+1}\left[\int_0^{\tau^{\ep}_{N_0}}\LL_{\ep}V(X^{\ep}_s)ds\right]\leq V(N_0+1)-\E^{\ep}_{N_0+1}[\tau^{\ep}_{N_0}],
% \end{equation*}
leading to
\begin{equation}\label{eqn-2022-05-16-2}
    \sup_{\ep}\E^{\ep}_{N_0+1}[\tau^{\ep}_{N_0}]\leq V(N_0+1).
\end{equation}

%Since $\lim_{x\to\infty}\frac{b}{|\si|}=-\infty$ is ensured by {\bf (H4)}, we see from the classical result in \cite{Khas80} that $X^{\ep}_t$ is positive recurrent and thus, $\E^{\ep}_{x}[\tau^{\ep}_{x_*}]<\infty$ for any $x\in [x_*,\infty)$. 

Set $\tau^{\ep}_{(x_*,N_0+1)}:=\{t\geq 0: X^{\ep}_t=x_*\text{ or }N_0+1\}$. Then, $\E^{\ep}_{\bullet}[\tau^{\ep}_{(x_*,N_0+1)}]$ on $[x_*,N_0+1]$ solves
\begin{equation*}
    \begin{cases}
    \LL_{\ep}u=-1\quad \text{in}\quad (x_*,N_0+1),\\
    u(x_*)=0=u(N_0+1).
    \end{cases}
\end{equation*}
Arguing as in the proof of Lemma \ref{lem-mean-exit-time-bdd-interval} (2), we find
\begin{equation}\label{eqn-2022-05-16-3}
    \sup_{\ep}\E^{\ep}_{N_0}[\tau^{\ep}_{(x_*,N_0+1)}]<\infty.
\end{equation}

Let $X^{\ep}_0=N_0$, $\hat{\tau}^{\ep}_0=0$ and define recursively the following sequences of stopping times: before the first time $X_{t}^{\ep}$ reaches $x_*$ (i.e., $\tau_{x_*}^{\ep}$), for $n\in\N$, we let $\tau^{\ep}_{n}$ be the first time after $\hat{\tau}^{\ep}_{n-1}$ at which $X^{\ep}_t$ reaches $N_0+1$, and $\hat{\tau}^{\ep}_n$ be the first time after $\tau^{\ep}_{n}$ at which $X^{\ep}_t$ reaches $N_0$; since $\P^{\ep}_{N_0}[\tau_{x_*}^{\ep}<\infty]=1$, $\tau_{n}^{\ep}$ and $\hat{\tau}^{\ep}_{n}$ are defined up to some random index $n_{0}\in\N\cup\{0\}$; let $\tau^{\ep}_n=\hat{\tau}^{\ep}_n=\tau^{\ep}_{x^*}$ for all $n\geq n_0+1$. To be more specific, we recursively define for each $n\in\N$,
\begin{equation*}
\begin{split}
    \tau^{\ep}_{n}:=\inf\left\{t\geq \hat\tau^{\ep}_{n-1}: X^{\ep}_t=N_0+1\right\}\wedge \tau^{\ep}_{x_*},\quad
    \hat{\tau}^{\ep}_n:=\inf\left\{t\geq \tau^{\ep}_n: X^{\ep}_t=N_0\right\}\wedge \tau^{\ep}_{x_*}. 
\end{split}
\end{equation*}
%See Figure \ref{figure-stoppingtimes} for an illustration.

Clearly, $\tau^{\ep}_{n}=\inf\{t\geq \hat{\tau}^{\ep}_{n-1}: X^{\ep}_t=x_*\text{ or }N_0+1\}$ and $\tau^{\ep}_{n}\uparrow \tau^{\ep}_{x_*}$ as $n\to\infty$ for $\P^{\ep}_{N_0}$-a.e. Hence, 
\begin{equation}\label{eqn-2021-02-24-1-1}
    \begin{split}
    \E^{\ep}_{N_0}[\tau^{\ep}_{x_*}]&=\lim_{n\to\infty}\E^{\ep}_{N_0}[\tau^{\ep}_n]=\sum_{n=1}^{\infty}\E^{\ep}_{N_0}[\tau^{\ep}_n-\hat\tau^{\ep}_{n-1}]+\sum_{n=1}^{\infty}\E^{\ep}_{N_0}[\hat{\tau}^{\ep}_n-\tau^{\ep}_{n}].
    \end{split}
\end{equation}

Thanks to \cite[Theorem 6.3.1]{IW81} and Lemma \ref{lem-asymptotic-scale-function}, $p_{\ep}:=\P_{N_0}[X^{\ep}_{\tau^{\ep}_1}=N_0+1]$ satisfies
\begin{equation}\label{eqn-2022-05-16-4}
    \lim_{\ep\to 0}p_{\ep}=\lim_{\ep\to 0}\frac{s_{\ep}(N_0)-s_{\ep}(x_*)}{s_{\ep}(N_0+1)-s_{\ep}(x_*)}=\frac{s_0(N_0)-s_0(x_*)}{s_0(N_0+1)-s_0(x_*)}\in (0,1).
\end{equation}

For $n\geq 1$, we show
\begin{equation}\label{eqn-2021-02-23-1-1}
\begin{split}
    \P^{\ep}_{N_0}[X^{\ep}_{\tau^{\ep}_n}=\be]&=p^{n}_{\ep},\\
    \E^{\ep}_{N_0}[\tau^{\ep}_{n}-\hat\tau^{\ep}_{n-1}]&=p^{n-1}_{\ep} \E^{\ep}_{N_0}[\tau^{\ep}_1],\\
    \E^{\ep}_{N_0}[\hat{\tau}^{\ep}_n-\tau^{\ep}_n]&=p^{n}_{\ep}\E^{\ep}_{N_0+1}[\tau^{\ep}_{N_0}].
\end{split}
\end{equation} 
The first two equalities for $n=1$ are obvious. Thanks to the strong Markov property and time-homogeneity of $X^{\ep}_t$, we find 
\begin{equation*}
    \begin{split}
        \E^{\ep}_{N_0}[\hat{\tau}^{\ep}_1-\tau^{\ep}_1]&=\E^{\ep}_{N_0}[\hat{\tau}^{\ep}_1-\tau^{\ep}_1|X^{\ep}_{\tau^{\ep}_1}=x_*]\times\P^{\ep}_{N_0}[X^{\ep}_{\tau^{\ep}_1}=x_*]\\
        &\quad+\E^{\ep}_{N_0}[\hat{\tau}^{\ep}_1-\tau^{\ep}_1|X^{\ep}_{\tau^{\ep}_1}=N_0+1]\times\P^{\ep}_{N_0}[X^{\ep}_{\tau^{\ep}_1}=N_0+1]\\
        &=\E^{\ep}_{N_0}[\hat{\tau}^{\ep}_1-\tau^{\ep}_1|X^{\ep}_{\tau^{\ep}_1}=N_0+1]\times\P^{\ep}_{N_0}[X^{\ep}_{\tau^{\ep}_1}=N_0+1]=\E^{\ep}_{N_0+1}[\tau^{\ep}_{N_0}]p_{\ep}.  
    \end{split}
\end{equation*}
Hence, \eqref{eqn-2021-02-23-1-1} holds for $n=1$. Suppose it is true for $n=k-1$ with $k\geq 2$. By the strong Markov property and time-homogeneity of $X^{\ep}_t$,
\begin{equation*}
    \begin{split}
\P^{\ep}_{N_0}[X^{\ep}_{\tau^{\ep}_k}=N_0+1]        &=\P^{\ep}_{N_0}[X^{\ep}_{\tau^{\ep}_k}=N_0+1|X^{\ep}_{\hat{\tau}^{\ep}_{k-1}}=N_0]\\
&\quad\times\P^{\ep}_{N_0}[X^{\ep}_{\hat{\tau}^{\ep}_{k-1}}=N_0|X^{\ep}_{\tau^{\ep}_{k-1}}=N_0+1]\times \P^{\ep}_{N_0}[X^{\ep}_{\tau^{\ep}_{k-1}}=N_0+1]=p_{\ep}^k,\\
\E^{\ep}_{N_0}[\tau^{\ep}_k-\hat\tau^{\ep}_{k-1}]        &=\E^{\ep}_{N_0}\left[\tau^{\ep}_k-\hat\tau^{\ep}_{k-1}|X^{\ep}_{\hat{\tau}^{\ep}_{k-1}}=N_0\right]\\
&\quad\times \P^{\ep}_{N_0}[X^{\ep}_{\hat{\tau}^{\ep}_{k-1}}=N_0|X^{\ep}_{\tau^{\ep}_{k-1}}=N_0+1]\times \P^{\ep}_{N_0}[X^{\ep}_{\tau^{\ep}_{k-1}}=N_0+1]\\
&=\E^{\ep}_{N_0}\left[\tau^{\ep}_k-\hat\tau^{\ep}_{k-1}|X^{\ep}_{\hat{\tau}^{\ep}_{k-1}}=N_0\right]\times \P^{\ep}_{N_0}[X^{\ep}_{\tau^{\ep}_{k-1}}=N_0+1]=\E^{\ep}_{N_0}[\tau^{\ep}_{1}] p_{\ep}^{k-1},\\
\E^{\ep}_{N_0}[\hat{\tau}^{\ep}_k-\tau^{\ep}_k]&=\E^{\ep}_{N_0}\left[\hat{\tau}^{\ep}_k-\tau^{\ep}_k|X^{\ep}_{\tau^{\ep}_k}=N_0+1\right]\times\P^{\ep}_{N_0}[X^{\ep}_{\tau^{\ep}_k}=N_0+1]=\E^{\ep}_{N_0+1}[\tau^{\ep}_{N_0}]p_{\ep}^k. 
\end{split}
\end{equation*}
Consequently, \eqref{eqn-2021-02-23-1-1} holds for $n=k$ and thus, holds for all $n\in\N$ by induction.

Given \eqref{eqn-2021-02-23-1-1}, we see from \eqref{eqn-2021-02-24-1-1} that 
\begin{equation*}
\begin{split}
    \E^{\ep}_{N_0}[\tau^{\ep}_{x_*}]&= \sum_{n=1}^{\infty}\left(p^{n-1}_{\ep} \E^{\ep}_{N_0}[\tau^{\ep}_{1}]+p^{n}_{\ep}\E^{\ep}_{N_0+1}[\tau^{\ep}_{N_0}]\right)= \frac{1}{1-p_{\ep}} \E^{\ep}_{N_0}[\tau^{\ep}_{1}]+\frac{p_{\ep}}{1-p_{\ep}}\E^{\ep}_{N_0+1}[\tau^{\ep}_{N_0}],
\end{split}
\end{equation*}
which together with \eqref{eqn-2022-05-16-2},  \eqref{eqn-2022-05-16-3} and \eqref{eqn-2022-05-16-4} yields 
$\sup_{\ep}\E^{\ep}_{x}[\tau^{\ep}_{x_*}]\leq \sup_{\ep}\E^{\ep}_{N_0}[\tau^{\ep}_{x_*}]<\infty$. 
\end{proof}

We are ready to prove Theorem \ref{thm-asymptotic-mean-extinction}.

\begin{proof}[Proof of Theorem \ref{thm-asymptotic-mean-extinction}]
Clearly, it suffices to prove the result for $\E_{x}^{\ep}[T_{0}^{\ep}]$ for each $x\in(0,\infty)$.

Fix $x\in(0,\infty)$. Let  $0<\de\ll 1$ (depending on $x$) and then take $\be=\be(\de)$ and $x_*=x_*(\de)\in (0,\be)$ so that $x_*\in(0,x)$. They are introduced at the beginning of this section. The strong Markov property and time-homogeneity of $X^{\ep}_t$ then imply that 
\begin{equation}\label{eqn-2022-05-17-2}
    \E^{\ep}_x[T^{\ep}_0]=\E^{\ep}_x\left[\E^{\ep}_x\left[(T^{\ep}_0-\tau^{\ep}_{x_*}+\tau^{\ep}_{x_*})\big|X^{\ep}_{\tau^{\ep}_{x_*}}\right]\right]=\E^{\ep}_{x_*}[T^{\ep}_0]+\E^{\ep}_x[\tau^{\ep}_{x_*}].
\end{equation}
Since $\sup_{\ep}\E^{\ep}_x[\tau^{\ep}_{x_*}]<\infty$ by Lemma \ref{lem-mean-hitting-time-x*}, it suffices to study the asymptotic bounds of $\E^{\ep}_{x_*}[T^{\ep}_0]$. 

We follow the same idea as in the proof of Lemma \ref{lem-mean-hitting-time-x*}. Let $X^{\ep}_0=x_*$, $\hat{\tau}^{\ep}_0=0$ and define recursively the following sequences of stopping times: before the first time $X_{t}^{\ep}$ reaches $0$ (i.e., $T_{0}^{\ep}$), for $n\in\N$, we let $\tau^{\ep}_{n}$ be the first time after $\hat{\tau}^{\ep}_{n-1}$ at which $X^{\ep}_t$ reaches $\be$, and $\hat{\tau}^{\ep}_n$ be the first time after $\tau^{\ep}_{n}$ at which $X^{\ep}_t$ reaches $x_*$; since $\P_{x_*}[T_{0}^{\ep}<\infty]=1$, $\tau_{n}^{\ep}$ and $\hat{\tau}^{\ep}_{n}$ are defined up to some random index $n_{0}\in\N\cup\{0\}$; let $\tau^{\ep}_n=\hat{\tau}^{\ep}_n=T^{\ep}_{x^*}$ for all $n\geq n_0+1$. To be more specific, we recursively define for each $n\in\N$,
\begin{equation*}
\begin{split}
    \tau^{\ep}_{n}:=\inf\left\{t\geq \hat\tau^{\ep}_{n-1}: X^{\ep}_t=\be\right\}\wedge T^{\ep}_{0},\quad
    \hat{\tau}^{\ep}_n:=\inf\left\{t\geq \tau^{\ep}_n: X^{\ep}_t=x_*\right\}\wedge T^{\ep}_{0}. 
\end{split}
\end{equation*}

Clearly, $\tau^{\ep}_{n}=\inf\{t\geq \hat{\tau}^{\ep}_{n-1}: X^{\ep}_t=0\text{ or }\be\}$ and $\tau^{\ep}_{n}\uparrow T^{\ep}_{0}$ as $n\to\infty$ for $\P^{\ep}_{x_*}$-a.e. Hence, 
\begin{equation}\label{eqn-2021-02-24-1}
    \begin{split}
    \E^{\ep}_{x_*}[T^{\ep}_{0}]&=\lim_{n\to\infty}\E^{\ep}_{x_*}[\tau^{\ep}_n]=\sum_{n=1}^{\infty}\E^{\ep}_{x_*}[\tau^{\ep}_n-\hat\tau^{\ep}_{n-1}]+\sum_{n=1}^{\infty}\E^{\ep}_{x_*}[\hat{\tau}^{\ep}_n-\tau^{\ep}_{n}].
    \end{split}
\end{equation}
Set $p_{\ep}:=\P_{x_*}[X^{\ep}_{\tau^{\ep}_1}=\be]$. Following arguments as in the proof of Lemma \ref{lem-mean-hitting-time-x*}, we have for each $n\geq 1$,
\begin{equation*}\label{eqn-2021-02-23-1}
\begin{split}
    \P^{\ep}_{x_*}[X^{\ep}_{\tau^{\ep}_n}=\be]=p^n_{\ep},\quad\E^{\ep}_{x_*}[\tau^{\ep}_n-\hat\tau^{\ep}_{n-1}]=p^{n-1}_{\ep} \E^{\ep}_{x_*}[\tau^{\ep}_1]\quad\text{and}\quad\E^{\ep}_{x_*}[\hat{\tau}^{\ep}_n-\tau^{\ep}_n]=p^{n}_{\ep}\E^{\ep}_{\be}[\tau^{\ep}_{x_*}],
\end{split}
\end{equation*} 
This together with  \eqref{eqn-2021-02-24-1} yields 
\begin{equation}\label{eqn-2022-05-17-1}
\begin{split}
    \E^{\ep}_{x_*}[T^{\ep}_{0}]&= \sum_{n=1}^{\infty}\left(p^{n-1}_{\ep} \E^{\ep}_{x_*}[\tau^{\ep}_{1}]+p^{n}_{\ep}\E^{\ep}_{\be}[\tau^{\ep}_{x_*}]\right)= \frac{1}{1-p_{\ep}} \E^{\ep}_{x_*}[\tau^{\ep}_{1}]+\frac{p_{\ep}}{1-p_{\ep}}\E^{\ep}_{\be}[\tau^{\ep}_{x_*}].
\end{split}
\end{equation}

\medskip

\noindent{\bf Case: $\La_0<0$.} Thanks to Lemmas \ref{lem-asymptotic-scale-function}, \ref{lem-mean-exit-time-bdd-interval} and \ref{lem-mean-hitting-time-x*}, there are $C_1,C_2>0$ such that $C_1|\ln \ep|\lesssim_{\ep} \E^{\ep}_{x_*}[T^{\ep}_0]\lesssim_{\ep} C_2|\ln \ep|$. From which and \eqref{eqn-2022-05-17-2}, the desired result follows. 

\medskip

\noindent{\bf Case: $\La_0>0$.} We rewrite \eqref{eqn-2022-05-17-1} as
\begin{equation}\label{eqn-2022-05-26-1}
    \E^{\ep}_{x_*}[T^{\ep}_0]=\frac{1}{1-p_{\ep}}\left(\E^{\ep}_{x_*}[\tau^{\ep}_1]+\E^{\ep}_{\be}[\tau^{\ep}_{x_*}]\right)-\E^{\ep}_{\be}[\tau^{\ep}_{x_*}].
\end{equation}

By Lemmas  \ref{lem-asymptotic-scale-function}, \ref{lem-mean-exit-time-bdd-interval} and \ref{lem-mean-hitting-time-x*}, there are positive constants $C_3, C_4, C_5$ and $C_6$ such that
\begin{equation*}
    C_3\ep^{-2(\ka_{-}-1)}\lesssim_{\ep} \frac{1}{1-p_{\ep}}\lesssim_{\ep} C_4\ep^{-2(\ka_{+}-1)},\,\, C_5\lesssim_{\ep} \E^{\ep}_{x_*}[\tau^{\ep}_1]+\E^{\ep}_{\be}[\tau^{\ep}_{x_*}]\lesssim_{\ep} C_6 \ep^{-2(\ka_{+}-\ka_{-}+\de)},
\end{equation*}
which together with \eqref{eqn-2022-05-26-1} yield $C_3 C_5\ep^{-2(\ka_{-}-1)}\lesssim_{\ep} \E^{\ep}_{x_*}[T^{\ep}_0]\lesssim_{\ep}  C_4C_6\ep^{-2(\ka_{+}-1+\ka_{+}-\ka_{-}+\de)}$.

We see from the definition of $\ka_{+}$ and $\ka_{-}$ in \eqref{eqn-def-ratios} that for any $0<\ga\ll 1$, there exists $\de>0$ (and corresponding $x_*=x_*(\de)$) so that 
$$
-(\ka_{-}-1)\leq 1-(1-\ga)\frac{2b'(0)}{|\si'(0)|}\,\, \andd\,\, 1-(1+\ga)\frac{2b'(0)}{|\si'(0)|}\leq -(\ka_{+}-1+\ka_{+}-\ka_{-}+\de),
$$
leading to $C_3C_5 \ep^{2-(1-\ga)\frac{4b'(0)}{|\si'(0)|}}\lesssim_{\ep} \E^{\ep}_{x_*}[T^{\ep}_0]\lesssim_{\ep} C_4 C_6 \ep^{2-(1+\ga)\frac{4b'(0)}{|\si'(0)|}}$. This together with \eqref{eqn-2022-05-17-2} yields the result.

The proof is complete.
\end{proof}

%%%%%%%%%%%%%%%%%%%%%%
%%%%%%%%%%%%%%%%%%%%%%

\section*{Acknowledgement}

We are grateful to the referees and the editors for their careful reading of the manuscript and for providing
many constructive critiques and helpful suggestions which led to a significant improvement of the
paper.

\begin{funding}
A.H. was supported by the NSF through the grants DMS 2147903 and CAREER 2339000. W.Q. was partially supported by a postdoctoral fellowship from the University of Alberta. Z.S. was partially supported by a start-up grant from the University of Alberta and NSERC RGPIN-2018-04371.  Y.Y. was partially supported by NSERC RGPIN-2020-04451, PIMS CRG grant, a faculty development grant from the University of Alberta, and a Scholarship from Jilin University.
% The first author was supported by ...
%
% The second author was supported in part by ...
\end{funding}

\end{document}